\documentclass[11pt,a4paper]{amsart}
\usepackage{a4wide}
\usepackage{amsthm, amsmath, amsfonts, amssymb,graphicx}
\usepackage{mathrsfs,url,color}

\usepackage[dvipsnames]{xcolor}
\usepackage[normalem]{ulem}
\usepackage[utf8]{inputenc}
\usepackage{bbm}

\numberwithin{equation}{section}

\newcommand{\Tr}{\operatorname{Tr}}
\newcommand{\bTr}{\mathbf{Tr}}

\newcommand{\fB}{\boldsymbol{\mathfrak{B}}}
\newcommand{\balpha}{{\boldsymbol{\alpha}}}

\newcommand{\brho}{{\boldsymbol{\rho}}}
\newcommand{\bchi}{{\boldsymbol{\chi}}}

\newcommand{\bsim}{\boldsymbol{\sim}}
\newcommand{\fY}{\boldsymbol{\mathfrak{Y}}}
\newcommand{\fS}{\boldsymbol{\mathfrak{S}}}

\newcommand{\bh}{\mathbf{h}}
\newcommand{\bj}{\mathbf{j}}
\newcommand{\bF}{\mathbf{F}}

\newcommand{\GL}{\mathrm{GL}}

\newcommand{\pow}{\mathcal{P}} 

\newcommand{\rainbow}{\mathrel{\smash{\raise-.1ex\hbox to 0pt{$\kern1pt\smallfrown$\hss}\raise.3ex\hbox{$\frown$}}}}

\newcommand{\bpi}{\bar{\pi}}

\newcommand{\bU}{\bar{U}}
\newcommand{\bV}{\bar{V}}
\newcommand{\bW}{\bar{W}}
\newcommand{\bbU}{\bar{\mathbf{U}}}
\newcommand{\bfU}{\mathbf{U}}
\newcommand{\bfV}{\mathbf{V}}
\newcommand{\bfW}{\mathbf{W}}
\newcommand{\bbW}{\bar{\mathbf{W}}}

\newcommand{\radcl}{\rad_{\mathrm{cl}}}
\newcommand{\radop}{\rad_{\mathrm{op}}}

\newcommand{\LL}{\mathscr{L}}
\newcommand{\Lring}{\LL_{\mathrm{ring}}}
\newcommand{\Lvf}{\LL_{\mathrm{VF}}}
\newcommand{\Leq}{\LL^\eq}
\newcommand{\Lnoval}{\LL_{0}} 
\newcommand{\Lnovaleq}{\Lnoval^\eq}
\newcommand{\Lnovalfine}{\LL_{0,\fine}} 
\newcommand{\Lnovalfineeq}{\LL_{0,\fine}^\eq}

\newcommand{\TT}{\mathscr{T}}
\newcommand{\Teq}{\TT^\eq}
\newcommand{\Tnoval}{\TT_{0}}

\newcommand{\Sloc}{S^{\mathrm{loc}}}

\newcommand{\valring}{\boldsymbol{\mathcal O}}
\newcommand{\maxid}{\boldsymbol{\mathcal M}}

\newcommand{\fine}{{\mathrm{fine}}}

\newcommand{\Sh}{\mathrm{Sh}}
\newcommand{\bSh}{\mathbf{Sh}}

\newcommand{\NN}{\mathbb{N}}
\newcommand{\bbL}{\mathbb{L}}
\newcommand{\ZZ}{\mathbb{Z}}
\newcommand{\FF}{\mathbb{F}}
\newcommand{\QQ}{\mathbb{Q}}
\newcommand{\RR}{\mathbb{R}}
\newcommand{\CC}{\mathbb{C}}

\newcommand{\bGr}{\mathbf{Gr}}

\newcommand{\CCt}[1][t]{\CC(\!(#1)\!)}

\newcommand{\kt}[1][t]{k(\!(#1)\!)}

\newcommand{\cC}{\mathscr{C}}

\newcommand{\Crig}{\operatorname{Crig}}
\newcommand{\Prig}{\operatorname{Prig}}
\newcommand{\bCrig}{\mathbf{Crig}}
\newcommand{\bPrig}{\mathbf{Prig}}

\DeclareMathOperator{\im}{im}
\DeclareMathOperator{\pr}{pr}

\DeclareMathOperator{\Gal}{Gal}

\DeclareMathOperator{\rv}{rv}
\DeclareMathOperator{\ac}{\overline{ac}}
\newcommand{\RV}{\mathbf{RV}}
\newcommand{\VF}{\mathbf{VF}}
\newcommand{\VG}{\mathbf{VG}}
\newcommand{\RF}{\mathbf{RF}}
\DeclareMathOperator{\dcl}{dcl}
\newcommand{\dclVF}{\dcl_{\LL}^{\scriptscriptstyle{\VF}}}
\newcommand{\dclRV}{\dcl_{\Leq}^{\scriptscriptstyle{\RV}}}
\newcommand{\dclRF}{\dcl_{\Leq}^{\scriptscriptstyle{\RF}}}
\newcommand{\dclVG}{\dcl_{\Leq}^{\scriptscriptstyle{\VG}}}

\DeclareMathOperator{\res}{res}

\newcommand{\eq}{{\mathrm{eq}}}

\DeclareMathOperator{\rtsp}{rtsp}
\DeclareMathOperator{\drtsp}{drtsp}

\def\Int{\operatorname{int}}
\def\Cl{\operatorname{cl}}

\definecolor{refkey}{RGB}{150,150,150}
\definecolor{labelkey}{RGB}{150,150,150}

\title[Riso-stratifications and a tree invariant]{Riso-stratifications and a tree invariant}

\author{David Bradley-Williams}
\author{Immanuel Halupczok}
\address{David Bradley-Williams, Institute of Mathematics, Czech Academy of Sciences, \v Zitn\'a 25, 115 67 Praha 1, Czech Republic.}
\address{Immanuel Halupczok, Heinrich Heine University D\"usseldorf,
Faculty of Mathematics and Natural Sciences, 
Mathematical Institute, 
Universit\"atsstr.\ 1, 40225 D\"usseldorf, Germany.}

\thanks{
Both authors would like to thank the organisers of the \emph{Model Theory, Combinatorics and Valued fields} trimester at the IHP, Paris in 2018, where some of this work was carried out; the participation of D.B-W. in the trimester was partially supported by a HeRA travel grant. D.B-W. would also like to thank the Fields Institute, Toronto for hospitality and support during the 2021 thematic program ``Trends in Pure and Applied Model Theory''. Both authors were partially supported at Heinrich Heine University D\"usseldorf by the research training group \emph{GRK 2240: Algebro-Geometric Methods in Algebra, Arithmetic and Topology}, funded by the DFG. I.H. was additionally partially supported by the individual research grant No. 426488848, \emph{Archimedische und nicht-archimedische Stratifizierungen höherer Ordnung}, also funded by the DFG. D.B-W. was additionally supported by project EXPRO 20-31529X of the Czech Science Foundation (GA\v CR) and by the Czech Academy of Sciences CAS (RVO 67985840).}


\usepackage{cite}
\usepackage{amsmath}
\usepackage{amsthm}
\usepackage{amssymb}
\usepackage{mathrsfs}
\usepackage{enumerate}
\usepackage{graphicx}

\theoremstyle{plain}

\newtheorem{thm}{Theorem}[subsection]

\newtheorem{prop}[thm]{Proposition}

\newtheorem{lem}[thm]{Lemma}

\newtheorem{cor}[thm]{Corollary}

\newtheorem{qu}[thm]{Question}
\newtheorem{claim}{Claim}

\newtheorem*{claim*}{Claim}

\theoremstyle{definition}
\newtheorem{defn}[thm]{Definition}
\newtheorem{notn}[thm]{Notation}
\newtheorem{conv}[thm]{Convention}
\newtheorem{exa}[thm]{Example}

\newtheorem{rmk}[thm]{Remark}
\newtheorem{modrmk}[thm]{Model Theoretic Remark}

\newtheorem{hyp}[thm]{Hypothesis}

\makeatletter
\newcommand{\thorn}{{\fontencoding{T1}\selectfont\th}}
\newcommand{\@indepsymbol}[2]{#1\setbox0=\hbox{$#1x$}\kern\wd0\hbox to 0pt{\hss$#1\mid$\hss}\lower.9\ht0\hbox to 0pt{\hss$#1\smile$\hss}\kern\wd0}
\newcommand{\@nindepsymbol}[2]{#1\setbox0=\hbox{$#1x$}\kern\wd0\hbox to 0pt{\mathchardef
	\nn=12854\hss$#1\nn$\kern1.4\wd0\hss}\hbox to
	0pt{\hss$#1\mid$\hss}\lower.9\ht0 \hbox to
	0pt{\hss$#1\smile$\hss}\kern\wd0}
\newcommand{\ind}[1][]{\mathop{\mathpalette\@indepsymbol{}^{\!\!\!\!\rlap{$\scriptstyle\textnormal{#1}$}\,\,\,\,}}}
\newcommand{\nind}[1][]{\mathop{\mathpalette\@nindepsymbol{}^{\!\!\!\rlap{$\scriptstyle\textnormal{#1}$}\,\,\,}}}
\newcommand{\@Ind}[1][]{\mathpalette\@indepsymbol{}^{\!\!\!\!\mbox{$\scriptstyle\textnormal{#1}$}}}
\newcommand{\Ind}[1][]{\@Ind[\ \,]}
\newcommand{\newind}[4]{
	\newcommand{#1}{{\!\@Ind[#4]}}
	\newcommand{#2}{\ind[#4]}
	\newcommand{#3}{\nind[#4]}
}
\newind{\thInd}{\thind}{\nthind}{\thorn}
\renewcommand{\thInd}{\text{$\@Ind[\thorn]$\;}}
\makeatother

\newcommand{\bA}{\mathbf{A}}
\newcommand{\bB}{\mathbf{B}}
\newcommand{\bC}{\mathbf{C}}
\newcommand{\bQ}{\mathbf{Q}}

\newcommand{\bS}{\mathbf{S}}
\newcommand{\bX}{\mathbf{X}}
\newcommand{\bY}{\mathbf{Y}}
\newcommand{\bZ}{\mathbf{Z}}

\newcommand{\pt}{\mathbf{pt}}
\newcommand{\bff}{\mathbf{f}}
\newcommand{\bfg}{\mathbf{g}}

\newcommand{\ra}{\rightarrow}

\DeclareMathOperator{\rad}{rad}

\DeclareMathOperator{\Th}{Th}

\DeclareMathOperator{\id}{id}

\newcommand{\sseteq}{\subseteq}

\renewcommand{\phi}{\varphi}

\begin{document}
\maketitle

\begin{abstract}
We introduce a new notion of stratification (``riso-stratification''), which is canonical and which exists in a variety of settings, including different topological fields like $\CC$, $\RR$ and $\QQ_p$, and also including different o-minimal structures on $\RR$. Riso-stratifications are defined directly in terms of a suitable notion of triviality along strata; the key difficulty and main result is that 
the strata defined in this way are ``algebraic in nature'', i.e., definable in the corresponding first-order language.
As an example application, we show that local motivic Poincaré series are, in some sense, trivial along the strata of the riso-stratification.
Behind the notion of riso-stratification lies a new invariant of singularities, which we call the ``riso-tree'', and which captures, in a canonical way, information that was contained in the non-canonical strata of a Lipschitz stratification. On our way to the Poincaré series application, we show, among others, that our notions interact well with motivic integration.
\end{abstract}

\section{introduction}
\label{sec:intro}

Since Whitney's seminal paper \cite{Whi.strat},
a classical approach to get some understanding of the singularities of a ``geometric'' set $X$ (e.g., an algebraic set $X \subseteq \CC^n$) consists in considering a stratification of $X$, i.e., a partition into smooth subsets which satisfies certain regularity conditions. Often, one is not interested in the regularity conditions themselves, but in the fact that they imply a certain triviality along strata, e.g.\ topological triviality in the case of Whitney stratifications or bi-Lipschitz triviality in the case of Mostowski's Lipschitz stratifications.
In the present paper, we avoid the detour through regularity conditions: We associate, to every geometric set $X$, a 
canonical stratification, its \emph{riso-stratification}, which is directly defined in terms of a suitable triviality condition.

Intuitively, that  triviality condition is very simple and natural. It says: on a tiny neighbourhood of a point of a stratum $S$, $X$ is almost translation invariant along $S$. To make this formal, we use ideas from non-standard analysis.
A variant of the notion of triviality already appeared in \cite{i.whit}; what is really new in the present paper is the (very non-trivial) result that if we use this triviality notion to define a stratification, then the strata live in the right category: they are constructible/globally subanalytic/etc.\ if $X$ is. We are not aware of any previous result that stratifications can be obtained directly from a natural triviality notion.

Canonical stratifications have been constructed before: For many kinds of regularity conditions, one obtains a minimal stratification satisfying those conditions by simply moving points into lower-dimensional strata whenever one is forced to (see e.g. \cite{Ran.compStrat}).
Riso-stratifications are strictly stronger than (minimal) Whitney stratifications, as one sees in the application to Poincaré series mentioned below.

An interesting benefit of a notion of stratification being canonical is that it yields an invariant of singularities: to a singularity $a \in X$, one associates the dimension of the stratum containing $a$.
In the case of Whitney stratifications over $\CC$, Teissier \cite{Tei.varPol2}
gave an explicit description of that invariant in terms of polar varieties. Taking such a point of view, the riso-stratification is merely the shadow of a much stronger invariant which we call the \emph{riso-tree} of $X$. 
It is this riso-tree which is really the central object of study of this paper.

As an example,
consider the trumpet $X = \bX(\CC) \subseteq \CC^3$ from Figure~\ref{fig:trumpet} (on p.~\pageref{fig:trumpet}). Its riso-stratification consists of one (zero-dimensional) stratum at the origin and one (two-dimensional) stratum consisting of the rest of $X$. The riso-tree associated to $X$ additionally sees a kind of one-dimensional germ at the origin, reflecting the fact that the tangent cone of $\bX(\CC)$ at $0$ is singular along the $x$-axis. This singularity of the tangent cone is also detected by Mostowski's Lipschitz stratifications: Lipschitz-stratifying $X$ yields an actual one-dimensional stratum leaving the origin to the right. However, while this one-dimensional stratum is inherently non-canonical, the riso-tree captures that information in a canonical way. Another advantage of the riso-tree is the simplicity of its definition, compared to the ingenious but highly non-trivial definition of Lipschitz stratifications.

One of our motivations for introducing riso-trees was to get a better understanding of the Poincaré series associated to an algebraic set $X = \bX(\CC) \subseteq \CC^n$, and more specifically to understand the poles of that series.
A powerful tool to determine those poles is Denef's formula \cite[Theorem~3.1]{Den.degIgusa}. However, this formula depends on a resolution of singularities of $\bX$, and it can yield false candidate poles depending on the chosen resolution.
In \cite{i.whit}, the second author introduced \emph{t-stratifications} (which are defined on valued fields like $\CCt$) and showed that a t-stratification of $\bX(\CCt)$ yields an alternative method to determine the Poincaré series of $X$. But again, t-stratifications are inherently non-canonical and tend to yield false candidate poles.
In some sense, the riso-trees introduced in the present paper capture exactly the canonical part of the information contained in t-stratifications (and thereby can be considered as a positive answer to the question in \cite[Section~9.4]{i.whit}). In particular, by applying the method described in \cite[Section~8]{i.whit}, one could obtain a concrete formula for the poles of the Poincaré series of $X$ in terms of the (combinatorial) data describing the riso-tree of $X$. However, the explicit formula would typically be long and cumbersome.
Therefore, instead of writing out such an explicit formula, we present the following concrete application to Poincaré series, whose proof illustrates how the riso-tree of $X$ controls its Poincaré series.

Let $X \subseteq \CC^n$ be an algebraic subset and fix a point $x \in X$.
Recall that the \emph{local motivic Poincaré series} $P_{X,x}(T)$ of $X$ at $x$ is a certain formal power series in $T$.
If $X$ can be written as a cartesian product $X = Y \times Z$, $x = (y,z)$, and $Z$ is smooth at $z$, then there is an easy way to compute $P_{X,x}(T)$ from $P_{Y,y}(T)$ and $d := \dim Z$, namely, $P_{X,x}(T) = \bbL^{-d}\cdot P_{Y,y}(\bbL^{d}\cdot T)$.
This raises the following question:
For an arbitrary algebraic subset $X \subseteq \CC^n$, if $x$ lies in a $d$-dimensional stratum $S$ of a stratification of $X$, 
do we also have $P_{X,x}(T) = \bbL^{-d}\cdot P_{X \cap V,x}(\bbL^{d}\cdot T)$
where $V$ is an affine space through $x$ transversal to the tangent space $T_xS$? This does not hold for an arbitrary Whitney or Verdier stratification (we provide a counter-example), but we show that this is indeed true if one uses the riso-stratification of $X$.

A key step in this proof consists in showing that riso-trees control motivic integrals -- a result which potentially opens up applications of riso-trees to a whole other range of questions.
In particular, building on this, a completely different
application of riso-trees is given in \cite{iC.vit}: To a subset $X \subseteq \RR^n$, one classically associates certain invariants $V_0(X), \dots, V_n(X)$ called Vitushkin variations. In \cite{iC.vit}, it is shown that an analogue of those $V_i(X)$ can also be defined over certain non-Archimedean fields, by examining the riso-tree of $X$.

Most constructions and proofs in this paper work uniformly in a wide range of settings, including e.g.\ algebraic sets $X \subseteq K^n$ for any algebraically closed field $K$ of characteristic $0$, semi-algebraic and globally subanalytic sets $X \subseteq \RR^n$, and also algebraic and semi-algebraic sets $X$ in other topological fields like the $p$-adic numbers $\QQ_p$ and fields $\kt$ of Laurent series of characteristic $0$.

In the following, we give a more detailed overview over the paper. A much more gentle introduction to riso-stratifications is given in the lecture notes \cite{i.rino}.

\subsection{Riso-stratifications}
\label{sec:i:risostrat}
We now define the riso-stratification of a set $X$.
For the sake of this introduction, we suppose that $X = \bX(\CC)$ is a constructible subset of $\CC^n$. The associated riso-stratification will be a stratification of the ambient space $\CC^n$.

The intuition behind the definition is simple: We just want to impose that if $x \in \CC^n$ lies in a stratum $S$ of the riso-stratification, then there should exist a ``very small neighbourhood'' $U_x$ of $x$ on which $X$ is ``almost translation invariant'' in the direction of the tangent space $T_xS$ of $S$.\footnote{More precisely, one should say: the smaller the neighbourhood, the closer to translation invariant; see \cite[Section~1]{i.rino} for details about why this is really what the following definition yields.}
To make ``very small neighbourhood'' and ``almost translation invariant'' precise, we choose a field $K \supseteq \CC$ which contains infinitesimal elements. Concretely, one can choose $K$ to be 
the field $\CCt[t^{\QQ}]$ of Hahn series, where $t$ is considered as being infinitesimal. (Equivalently, $K$ is the maximal immediate extension of the field of Puiseux series over $\CC$.)
Every point $x \in \CC^n$ now has a (unique) \emph{infinitesimal neighbourhood} in $K^n$, namely the set $U_x := x + \maxid(K)^n$, where $\maxid(K)^n \subseteq K^n$ denotes the $n$-th cartesian power of the maximal ideal corresponding to the $t$-adic valuation on $K$. This $U_x$ is our ``very small neighbourhood'' of $x$.

\begin{figure}
 \includegraphics{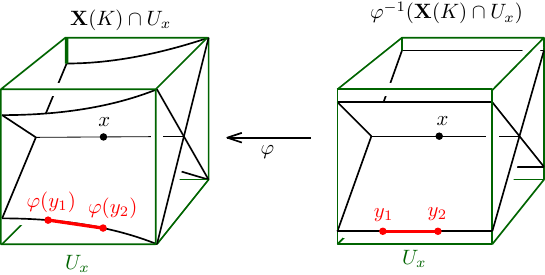}
 \caption{On the left hand side, the intersection of some set $\bX(K)$ with some infinitesimal neighbourhood $U_x$, and on the right hand side, its preimage under a risometry $\varphi$ making it horizontally translation invariant. That $\varphi$ is a risometry means that the differences $y_1 - y_2$ and $\varphi(y_1) - \varphi(y_2)$ are almost equal.}
 \label{fig:riso}
\end{figure}

Next, we need the notion of a \emph{risometry}, which one should think of as an infinitesimal perturbation (see Figure~\ref{fig:riso}). A risometry is a bijection  $\phi\colon U_x \to U_x$ satisfying
\[
|\phi(y_1) - \phi(y_2)-(y_1 - y_2)| < |y_1-y_2|
\]
for every $y_1\ne y_2$ in $U_x$; here, $|\cdot|$ denotes the $t$-adic valuation, written multiplicatively. We now make our desired notion of
``almost translation invariant'' precise by defining it as ``translation invariant up to risometry''. More formally, we define the \emph{riso-triviality space} $\rtsp_{U_x}(\bX(K))$ of $\bX(K)$ on $U_x$ as
the vector subspace of $\CC^n$ consisting of all those $v \in \CC^n$ such that there exists a risometry
$\phi\colon U_x \to U_x$ such that 
$\phi^{-1}(\bX(K) \cap U_x)$ is translation invariant in the direction of $v$ (within $U_x$).\footnote{The definition of the riso-triviality space given here differs from Definition~\ref{defn:rtsp}, but both definitions can easily be seen to be equivalent using Remark~\ref{rmk:lift:no:matter} and Proposition~\ref{prop:rtsp}.}

We now give a tentative definition of the riso-stratification, which we call the \emph{shadow} of $\bX$: 
It is a partition of $\CC^n$ into sets $\Sh_0, \dots, \Sh_n$, where the \emph{$d$-dimensional skeleton} $\Sh_d$ consist of all those points $x \in \CC^n$ such that the riso-triviality space $\rtsp_{U_x}(\bX(K))$ is $d$-dimensional.
We prove that $\Sh_d$ is constructible and that $\dim \Sh_d = d$ (unless it is empty). However, it might be that the topological closure of $\Sh_d$ is not contained in $\Sh_0 \cup \dots \cup \Sh_d$. We
correct this by iterating the above procedure, taking the shadow of the tuple $(X, \Sh_0, \Sh_1, \dots, \Sh_n)$, and so on. This process stabilizes after $n$ steps at most (Proposition~\ref{prop:shadow:stabilizes}); the final result is the riso-stratification of $X$. To justify the term ``stratification'', we show that it satisfies Whitney's regularity conditions (Theorem~\ref{thm:whit}).

\subsection{Riso-trees}
\label{sec:i:risotree}
The riso-triviality space of $\bX(K)$ does not only make sense on the above infinitesimal neighbourhoods $U_x$, but more generally on arbitrary valuative balls $B \subseteq K^n$. Recall that those balls form a tree. The \emph{riso-tree} of $\bX(K)$ is 
simply the partition of the tree of all closed valuative balls $B \subseteq K^n$ into the sets $\Tr_d := \{B \mid \dim \rtsp_B(\bX(K)) = d\}$.
The central result of this paper (Corollary~\ref{cor:riso-tree}) states that each set $\Tr_d$ is definable in the language of valued fields. From this, we deduce the above-mentioned result about the shadow $\Sh_d$ being constructible.
The proof of definability of the sets $\Tr_r$ heavily relies on model theory of valued fields, and even if one is only interested in constructibility of the shadow (as needed for the riso-stratifications), we have no clue how model theory could be avoided.

\begin{figure}
 \includegraphics{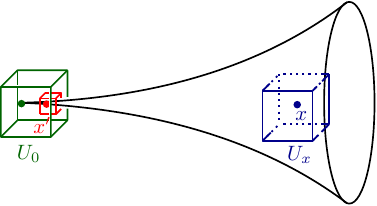}
 \caption{The set $\bX(K) \subseteq K^3$ given by $z^2 + y^2 = x^3$ shows that the riso-tree captures more information than the riso-stratification.}
 \label{fig:trumpet}
\end{figure}

We can now describe the trumpet example from Figure~\ref{fig:trumpet} in more detail: A riso-stratification of that trumpet only sees that the riso-triviality space on the infinitesimal neighbourhood $U_0$ of the origin is trivial and that the riso-triviality space on $U_x$ for any other point $x \in \bX(\CC)$ is two-dimensional (namely, equal to the tangent space $T_x(\bX(\CC))$). However, if we choose a valuative ball $B$ around some $x' \in \bX(K)$ which is infinitesimally close to $0$ but such that $0 \notin B$, then $\bX(K) \cap B$ is almost translation invariant horizontally (in the real picture, it is almost a cylinder), and indeed,
$\rtsp_B(\bX(K))$ is equal to the horizontal line.

\subsection{Relation to t-stratifications}
\label{sec:i:tstrat}

As already mentioned, the riso-tree can be considered as isolating the canonical information contained in the t-stratifications introduced in \cite{i.whit}. Let us now make this precise.

A t-stratification is a partition of the ambient space $K^n$ of $\bX(K)$ into definable skeleta $S_0, \dots, S_n$ (with $\dim S_d = d$) that in particular has the following property: If $B \subseteq K^n$ is any valuative ball with $\dim \rtsp_B(\bX(K)) = d$, then $B$ has non-empty intersection with $S_{\le d} :=  S_0 \cup \dots \cup S_d$ (see Definition~\ref{defn:tstrat} for details). Stratifications with this property are typically very non-unique in a similar way as Lipschitz stratifications: In Figure~\ref{fig:trumpet} for example, $S_1$ has to meet the small valuative ball around $x'$; to this end, it could be any curve along the trumpet close to the $x$-axis, or it could be the $x$-axis itself.

The riso-tree forgets the t-stratification and instead only recalls the key ingredient, namely the dimensions of the riso-triviality spaces $\rtsp_B(\bX(K))$. While this is clearly canonical, there is \emph{a priori} no reason why the riso-tree constructed in this way should be definable in the valued field language. Existence of t-stratifications is a useful ingredient for this proof, but interestingly, we also make use of techniques which are more analytic than purely model theoretic methods. Concretely, we therefore work in a valued field that is actually spherically complete, and not merely definably spherically complete.

As a side remark, note that our definition of the riso-triviality space differs slightly from the one in \cite{i.whit}. One reason for this difference is that, in at least some pathological cases, a riso-tree defined with the notion from \cite{i.whit} would not have been definable (see \cite[Example~3.15]{i.whit}). While it seems so far that in cases of interest this pathology should not arise (in which cases it is to be expected the two notions of riso-triviality space coincide), we were not able to show this even in the nicest possible setting, namely when $K = \CCt[t^{\QQ}]$ (see Question~\ref{qu:autodef}).

Despite the key difference between t-stratifications and the riso-tree, the latter inherits some properties from t-stratifications, making it into some kind of ``tree-shaped stratification''. Concretely, from the fact that a t-stratification has a finite set $S_0$ containing the ``worst singularities'', one deduces that the ``worst'' part $\Tr_0$ of the riso-tree is a tree with only finitely many branching points (see Lemma~\ref{lem:Crig}). A similar but more technical result holds for all $\Tr_d$, but instead of formulating this explicitly, we show that for $d \ge 1$, all the relevant information about a ball $B \in \Tr_d$ is already contained
in any fixed fiber $F \subseteq B$ ``sufficiently transversal'' to $\rtsp_B(\bX(K))$ 
(see Lemmas~\ref{lem:fiber_compat} to~\ref{lem:fiber_triv}). This yields a notion of induction over the riso-tree, which plays a central role in many proofs in this paper.

\subsection{Poincaré series}
\label{sec:i:poincare}

With the definition of riso-stratification at hand, our result about local motivic Poincaré series mentioned at the beginning of the introduction (and stated precisely in Theorem~\ref{thm:poincare}) becomes quite natural. Suppose that $X = \bX(\CC) \subseteq \CC^n$ is an algebraic subset and that $a \in X$ lies in the $d$-th skeleton $S_d$ of the riso-stratification. As before, denote by $U_a = a + \maxid(K)^n$ the infinitesimal neighbourhood of $a$, where $K = \CCt[t^{\QQ}]$.
One can verify that
the riso-triviality space $\rtsp_{U_a}(X)$ contains the tangent space $T_aS_d$, i.e., $X$ is ``almost $T_aS_d$-translation invariant'' on $U_a$.
Recall that the $T^r$-coefficient of the local motivic Poincaré series $P_{X,a}(T)$ of $X$ at $a$ is essentially the motivic measure $\mu(X_r)$ of a certain subset $X_r$ of $U'_{a} := a + (t\CC[[t]])^n \subseteq \CCt^n$. Note that $U'_{a}$ is, like $U_a$, an infinitesimal neighbourhood of $a$, but over the field $\CCt$; more precisely, $U'_a = U_a \cap \CCt^n$. The key steps to obtain that $P_{X,a}(T)$ is determined by  $P_{X \cap V,a}(T)$ (for some affine space $V$ through $a$ transversal to $T_aS_d$) are to show that (a) restricting from $K$ to $\CCt$ preserves the riso-triviality space (Proposition~\ref{prop:fields}), that (b) almost translation invariance of $X$ implies almost translation invariance of the above sets $X_r$, and that (c) risometries preserve motivic measure (Proposition~\ref{prop:presMeas}).

We carry out the above strategy for two variants of the local motivic Poincaré series (corresponding to the two kinds of $p$-adic Poincaré series considered e.g.\ by Denef in \cite{Den.rat}): $\tilde P_{X,a}(T)$, whose $T^r$-coefficient is the motivic cardinality of the sets of jets in $\bX(\CC[t]/t^r)$ that are centered at $a$, and $P_{X,a}(T)$, which additionally requires the jets to lift to $\bX(\CC[[t]])$. Treating $\tilde P_{X,a}(T)$ needs an additional quirk:
The notion of riso-stratification presented in Section~\ref{sec:i:risostrat} is defined set-theoretically in the sense that it captures riso-triviality of the set $\bX(K)$. To obtain the result for $\tilde P_{X,a}(T)$,
we introduce an algebraic variant of the riso-stratification, which, in a suitable sense, captures riso-triviality of the ideal defining $\bX$ (see Definition~\ref{defn:alg:riso}).

Concerning the above Step~(c):
While at first glance, it might seem very natural that risometries preserve motivic measure, recall that
the risometries we consider are set-theoretic maps, with no condition of algebraicity or definability imposed.
To prove Proposition~\ref{prop:presMeas} nevertheless, we use the full strength of riso-trees, doing an induction over the riso-tree as explained in Section~\ref{sec:i:tstrat}.

\subsection{Summary of the results}
\label{sec:i:details}

Many of our results hold under quite a large generality, and on our way, we also obtain some more results of interest. Here are some details:

\begin{enumerate}
\item\label{it:rtt}
The heart of this paper is the \emph{Riso-Triviality Theorem}~\ref{thm:RTT}, which states that the riso-triviality space is definable. From this, one easily deduces the definability of the riso-tree mentioned in Section~\ref{sec:i:risotree} (Corollary~\ref{cor:riso-tree}). This result holds not only in $K = \CCt[t^{\QQ}]$, but more generally in a varied class of spherically closed valued fields $K$ of equi-characteristic $0$. Moreover, instead of working in the pure valued field language, one can also work in any $1$-h-minimal expansion of $K$ in the sense of \cite{iCR.hmin}. This includes for example expansions by analytic functions and the power-bounded $T$-convex structures introduced by van den Dries and Lewenberg \cite{DL.Tcon1}.

Moreover, instead of considering riso-triviality of a subset of $K^n$, one can make sense of riso-triviality of a partition of $K^n$. Such a partition can
be given as the fibers of a map from $K^n$ another set, and in that sense, the Riso-Triviality Theorem holds more generally for definable maps from $K^n$ to
(a cartesian power of) the leading term structure $\RV(K)$ (or a quotient thereof). This more general version is needed to obtain the algebraic riso-stratifications mentioned at the end of Section~\ref{sec:i:poincare}.
\item\label{it:ret}
The Riso-Triviality Theorem is proved in a common induction with the \emph{Riso-Equivalence Theorem}~\ref{thm:RET}, which states that being in risometry is a definable condition. More precisely, we obtain that risometry types can be parametrized by elements of $\RV^\eq$; see Corollary~\ref{cor:RET} for the precise statement. This can be considered as a key step towards a classification of definable sets up to risometry. (We expect it should be possible, yet somewhat laborious, to finish the classification along the lines of \cite[Section~8]{i.whit}.)
\item\label{it:risostrat}
We introduce \emph{riso-stratifications} (Definition~\ref{defn:risoStrat}) in a variety of different settings, namely in algebraically closed fields of characteristic $0$, in real closed fields, in $p$-adically closed fields and in equi-characteristic $0$ valued fields of the form $\kt$ and $\kt[t^\QQ]$ (see Hypothesis~\ref{hyp:can_whit}). In all those settings, the set we are stratifying can be any definable set, e.g.\ a constructible set in an algebraically closed field or a semi-algebraic set in a real closed field. In the case of real closed fields, we also allow to work in certain other o-minimal structures, including e.g.\ the one consisting of all globally subanalytic sets in $\RR^n$. The precise condition we need is that the o-minimal structure is power-bounded (see Definition~\ref{defn:powbd}), which amounts to requiring that no definable function grows exponentially.
\item \label{it:whit}
To put riso-stratifications in relation to other notions of stratifications, we prove that they satisfy Whitney's regularity conditions (Theorem~\ref{thm:whit}). While in almost all of the settings from Hypothesis~\ref{hyp:can_whit}, the existence of Whitney (and even Verdier) stratifications was known before (see \cite{Loi.ominStrat} in the o-minimal case and \cite{CCL.cones,For.motCones} for the valued field cases with discrete value group), our approach yields a uniform existence proof for all those fields at once. (In $\kt[t^\QQ]$, even existence of Whitney stratifications seems to be new, though the methods of \cite{CCL.cones,For.motCones} should be expected to extend to this context also.)

Note that we do not know whether a riso-stratification is always a Verdier stratification, and we rather believe that this is not the case.
\item\label{it:algriso}
To any closed subscheme $\bX$ of the affine space $\mathbb A^n$, we associate an \emph{algebraic riso-stratification} (Definition~\ref{defn:alg:riso}), which captures information about the non-reduced structure of $\bX$ (see Example~\ref{ex:nonred}). When $\bX$ is reduced, we do not know whether it captures more information than the riso-stratification of the set $\bX(K)$ for $K$ algebraically closed of characteristic $0$.
\item\label{it:fields}

Riso-triviality in $K$ restricts well to suitable sub-fields $K_0 \subseteq K$:
If $X \subseteq K^n$ is definable, then riso-triviality of $X$ on a ball $B \subseteq K^n$ implies the corresponding riso-triviality of $X \cap K_0^n$ on $B \cap K_0^n$
(Proposition~\ref{prop:fields}). A similar result also holds for existence of a risometry between two different definable sets. Note however that our proof of this does not work in the same generality as the results from Items~(\ref{it:rtt}) and (\ref{it:ret}), but require
an additional technical assumption stated in Hypothesis~\ref{hyp:RFlin}, which ensures that induction over the riso-tree works particularly well.
\item\label{it:mot}
As mentioned before, risometries preserve motivic measure (Proposition~\ref{prop:presMeas}), and modulo some technicality, they also preserve motivic integrals (Proposition~\ref{prop:presInt}). Those results are obtained under the same assumptions as Proposition~\ref{prop:fields}, together with the assumptions needed for motivic integration to exist (Hypothesis~\ref{hyp:mot:int}). We use the Cluckers--Loeser version of motivic integration.
\item\label{it:poincare}
Finally, we obtain the application to Poincaré series (Theorem~\ref{thm:poincare}), which we already mentioned above.
We also give an example showing that this result would not hold if one would use an arbitrary Verdier stratification in place of the riso-stratification (Example~\ref{exa:notVerdier}).
\end{enumerate}

\subsection{Outline of the paper}
\label{sec:i:outline}

In Section~\ref{sec:setting}, we fix our notation and conventions. We also fix general assumptions about the first order language we work with (Hypothesis~\ref{hyp:KandX_can}) and we introduce (in Section~\ref{sec:riso}) the central notions of risometry and riso-triviality.

Section~\ref{sec:main-results} is devoted to the Riso-Triviality Theorem and the Riso-Equivalence Theorem
(items (\ref{it:rtt}) and (\ref{it:ret}) in the above list).
The results are stated precisely in Section~\ref{sec:thms}, and the remaining subsections of Section~\ref{sec:main-results} are devoted to their proof: In Section~\ref{sec:tstrat}, we relate some of the basic notions of the present paper to those introduced for developing t-stratifications in \cite{i.whit}. In
Section~\ref{sec:prel}, we adapt many basic results from \cite{i.whit} to the amended notions employed here; for convenience of the reader, we repeat the proofs even though much of the adaptation in this subsection is straight forward. Finally, the main work of proving definability happens in Section~\ref{sec:proofs}.

The main objects of study of Section~\ref{sec:can_whit} are the riso-stratifications.
After fixing the assumptions and showing how this relates to the valued fields considered previously in the paper (Section~\ref{sec:whit-ass}), Sections~\ref{sec:sh}, \ref{sec:iter:sh} and \ref{sec:riso-strat} are devoted to defining and constructing riso-stratifications (item (\ref{it:risostrat})) and proving their most basic properties. We then have an intermediate section relating riso-triviality to manifolds, before we come, in Section~\ref{sec:whit}, to the proof 
that riso-stratifications satisfy Whitney's regularity conditions (item~(\ref{it:whit}), Theorem~\ref{thm:whit}).
Finally, in Section~\ref{sec:alg:riso}, we introduce algebraic riso-stratifications (item~(\ref{it:algriso})).

The last section of the paper, Section~\ref{sec:appl}, is guided by our desired application to Poincaré series (item (\ref{it:poincare})), which is given in Subsection~\ref{sec:poincare}. It starts with a subsection explaining the additional assumption mentioned in item~(\ref{it:fields}) and how it helps to improve induction over the riso-tree.
In Sections~\ref{sec:fields} and \ref{sec:mot}, we prove the two missing ingredients, namely restriction to subfields (item~(\ref{it:fields})) and compatibility with motivic measure (item~(\ref{it:mot})).

\section{Setting}
\label{sec:setting}

\subsection{Notation and conventions}
\label{sec:notn}

We fix some terminology for the entire paper, first related to model theory in general and then more specifically about valued fields.

\begin{notn}[Model-independent definable objects]\label{notn:uniform}
Suppose that $\LL$ is a language and that $\TT$ is an (often implicitly fixed) $\LL$-theory.
\begin{enumerate}
 \item 
 By an \emph{$\LL$-definable set $\bX$ in the theory $\TT$}, we mean an $\LL$-formula up to equivalence modulo $\TT$. If $\TT$ is clear from the context, we will just write \emph{$\LL$-definable set}.
 If $K \models \TT$ is a model, we write $\bX(K)$ for the set of realizations of $\bX$ in $K$. Notationally, we treat such $\bX$ like sets, e.g.\ writing
$\bX^n$ for the cartesian power, $\bX \cup \bX'$ for the disjunction, $\bX \subseteq \bX'$ if $\bX$ implies $\bX'$, $\{x \in \bX \mid \phi(x)\}$ for the conjunction of $\bX$ and $\phi$, and so on.
\item
In a similar way, when $\bX$ and $\bY$ are $\LL$-definable sets, then by an ``$\LL$-definable map $\balpha\colon \bX \to \bY$'', we mean a formula defining a map $\balpha_K$ from $\bX(K)$ to $\bY(K)$ for every model $K \models \TT$. We will sometimes also use non-bold letters for $\LL$-definable maps and omit the index $K$ (writing $\alpha\colon \bX(K) \to \bY(K)$), mainly when we think of $\alpha$ as a 
function symbol from $\LL$.
\item
Other definable objects (like $\LL$-definable families, $\LL$-definable equivalence relations, $\LL$-definable elements) are treated similarly.
\end{enumerate}
\end{notn}

\begin{notn}[Definable objects in a model]
If $K \models \TT$, then we also speak about $\LL$-definable sets in $K$, with the usual meaning. More precisely, \emph{$\LL$-definable} means definable without parameters, whereas \emph{definable} means definable with parameters, and for $A \subseteq K$, $A$-definable means $\LL(A)$-definable.
\end{notn}

Typically, we will use boldface letters for model-independent $\LL$-definable objects and non-boldface letters for $\LL$-definable objects in a specific model.

\begin{notn}[Definable closure]
If $K$ is an $\LL$-structure, $\bS$ is a sort of $\LL$ and $A \subseteq K$ is a set of elements in arbitrary sorts,
we write $\dcl_{\LL}^{\scriptscriptstyle{\bS}}(A)$ for the elements in $\bS(K)$ which are $\LL(A)$-definable.
\end{notn}

\begin{conv}[Imaginaries]\label{conv:imaginaries}
Given a language $\LL$ and an $\LL$-theory $\TT$, we will often work in $\Leq$ and $\TT^\eq$. Recall that $\Leq$ and $\TT^\eq$ are obtained form $\LL$ and $\TT$ by adding, for every $\LL$-definable set $\bX$ and every $\LL$-definable equivalence relation $\bsim$ on $\bX$, a sort for $\bX/\mathord{\bsim}$
and the quotient map $\bX \to \bX/\mathord{\bsim}$.
Given a model $K \models \TT$, we denote the corresponding model of $\Teq$ by $K^\eq$, but to keep notation lighter,
if $\bX$ is an $\Leq$-definable set, then we just write $\bX(K)$ (instead of $\bX(K^\eq)$) for the set of realizations of $\bX$ in $K^\eq$. (And similarly for definable maps, etc.)
\end{conv}

Recall that this convention is harmless in the sense that it does not change definability in the original sorts of $\LL$. Recall also that this yields a natural notion of definability of sets of sets (by identifying a set with its code). More precisely, we will use the following convention:

\begin{conv}[Definable sets of sets]\label{conv:sets:of:sets}
Let $\bX$ be an $\Leq$-definable set.
By an $\Leq$-definable subset $\fY$ of the power set $\pow(\bX)$, we mean an $\Leq$-definable family $\tilde\fY$ of subset of $\bX$, parametrized by some $\Leq$-definable set $\bQ$, where each family member is unique, i.e., for every $K \models \TT$ and every $q,q' \in \bQ(K)$, $\tilde\fY(K)_q = \tilde\fY(K)_{q'}$ implies $q = q'$. Notationally, we treat $\fY$ like a subset of $\pow(\bX)$; for example, given a model $K \models \TT$ and a subset $Z \subseteq \bX(K)$, ``$Z \in \fY(K)$'' means that there exists a $q \in \bQ(K)$ such that
$\tilde\fY(K)_q = Z$. Other terminology is used accordingly; for example, an $\Leq$-definable subset of $\fY$ is given by an $\Leq$-definable subset of $\bQ$, an $\Leq$-definable
map from/to $\fY$ is an $\Leq$-definable map from/to $\bQ$, etc.
\end{conv}

Next, we fix some languages and theories that will be central in this paper.

\begin{notn}[The languages $\Lring$ and $\Lvf$]\label{notn:Lvf}
We write $\Lring = \{+,\cdot\}$ for the ring language and $\Lvf$ for the language of valued fields, i.e., the extension of $\Lring$ by a predicate $\valring$ for the valuation ring.
\end{notn}

In reality, in the entire paper, we care about languages only up to interdefinability, and we do not care whether additional sorts are added.

The following languages $\LL$ and $\Lnoval$ are the ones we will work with most. They come with corresponding theories $\TT$ and $\Tnoval$, which are the ones we implicitly think of when using model-independent definable sets as in Notation~\ref{notn:uniform}.

\begin{notn}[$\Lnoval$ and $\Tnoval$]\label{conv:Lnoval}
Any language called $\Lnoval$ in this paper will be an extension of $\Lring$, and there will be a corresponding $\Lnoval$-theory $\Tnoval$ which will always contain the theory of fields of characteristic $0$. We denote the home sort of $\Lnovaleq$ (i.e., the field sort) by $\bA$. Even if $\Lnoval$ is multi-sorted, we identify models $K_0 \models \Tnoval$ with their field sort, i.e., $K_0 = \bA(K_0)$.
\end{notn}

Note the analogy between the $\Lnoval$-definable set $\bA^n$ and the affine space $\mathbb A^n$.

\begin{notn}[$\LL$ and $\TT$]\label{conv:LL}
From now on (for the entire paper), $\LL$ will always be an extension of the valued field language $\Lvf$ and $\TT$ will be an $\LL$-theory containing the theory of henselian valued fields of equi-characteristic $0$. 
We write $\VF$ for the valued field sort of $\Leq$, $\RF$ for the residue field sort and $\VG^\times$ for the value group sort. 
We use multiplicative notation for the value group, and the valuation is denoted by $|\cdot|\colon \VF \to \VG := \VG^\times \cup \{0\}$.
We write $\valring \subseteq \VF$ for the valuation ring and $\maxid \subseteq \valring$ for its maximal ideal, considered as $\Leq$-definable sets. The residue map is denoted by $\res\colon \valring \to \RF$. Even if $\LL$ is multi-sorted, we identify models $K \models \TT$ with their valued field sort, i.e., $K = \VF(K)$.
\end{notn}

\begin{notn}
We write $|\cdot|\colon \VF^n \to \VG$ for the supremum norm, i.e., $|(x_1, \dots, x_n)| = \max_i |x_i|$.
\end{notn}

We also need the leading term structure and a higher dimensional variant of it:

\begin{defn}[$\RV^{(n)}$]\label{defn:rvn}
For every $n \ge 0$,
the \emph{$n$-dimensional leading term structure} is the $\Leq$-sort $\RV^{(n)} := \VF^n/\mathord{\bsim}$, where 
 \[x\bsim y \Leftrightarrow (|y-x| < |x| \vee x=y).\]
The canonical map $\VF^n \ra \RV^{(n)}$ is denoted by $\rv^{(n)}$. The element $\rv^{(n)}(0)$ of $\RV^{(n)}$ is also denoted by $0$.
In the case $n = 1$, we often omit the superscript: $\RV := \RV^{(1)}$ and $\rv := \rv^{(1)}$.
\end{defn}

\begin{notn}[$\RV^\eq$]
We write $\RV^\eq$ for the collection of those sorts of $\Leq$ which are $\Leq$-definable quotients of the form $\RV^m/\mathord{\bsim}$, for some $\Leq$-definable equivalence relation $\bsim$. (In particular, $\RF$, $\VG$, and $\RV^{(n)}$ are sorts of $\RV^\eq$.) By an $\Leq$-definable map into $\RV^\eq$, we mean an $\Leq$-definable map whose range is any sort of $\RV^\eq$.
\end{notn}

\begin{defn}[Balls]
Let $K \models \TT$ be a valued field and let $n \ge 1$.
We call a subset $B \subseteq K^n$ a \emph{ball} if $B$ is infinite and for all $x,x'\in B$ and all $y\in K$ with $|x-y|\leq |x-x'|$ one has $y\in B$.
The \emph{cut radius} of a ball $B$ is the cut
$\{|x-y| \in \VG^\times(K) \mid x, y \in B, x \ne y\}$.
A \emph{closed ball} is a ball of the form
\[B_{\le\lambda}(x) := \{x' \in K^n \mid |x' - x| \le \lambda\},\]
for $x \in K^n$ and $\lambda \in \VG^{\times}(K)$; we also call $\lambda$ the \emph{closed radius} of $B_{\le\lambda}(x)$, and we write $\radcl(B)$ for the closed radius of a closed ball $B$.
Similarly,
\[B_{<\lambda}(x):= \{x' \in K^n \mid |x' - x|< \lambda\}\]
is an \emph{open ball} of \emph{open radius} $\lambda$, and we write $\radop(B)$ for the open radius of an open ball $B$. We also consider $B_{<\infty}(x) := K^n$ as an open ball (of open radius $\infty$).
\end{defn}

Note that by this definition, $K^n$ is a ball, but individual points are not.

\begin{conv}[$\LL$-definable balls]
By an $\LL$-definable ball in $\VF^n$, we mean an $\LL$-definable subset $\bB \subseteq \VF^n$ such that $\bB(K)$ is a ball in $K^n$ for every model $K \models \TT$. In particular, whether $\bB(K)$ is open or closed or neither/nor may depend on $K$.
\end{conv}

\begin{defn}[Dimension]\label{defn:dim}
Given a definable set $X \subseteq K^n$, for some valued field $K$ and some $n \ge 0$, by $\dim X$, we mean its topological dimension, i.e., the maximal $d$ such that there exists a coordinate projection $\pi\colon K^n \to K^d$ such that $\pi(X)$ has non-empty interior. We set $\dim \emptyset := -\infty$. For an $\LL$-definable set $\bX \subseteq \VF^n$, we set $\dim \bX := \max_K \dim (\bX(K))$, where $K$ runs over all models of $\TT$.
\end{defn}

In the next section, we will impose conditions on $\LL$ on $\TT$ (for the entire paper) which in particular will ensure that this notion of dimension behaves well.

\subsection{The main hypothesis: $1$-h-minimality}
\label{sec:hmin}

We now fix the setting in which we can prove our main definability results (the Riso-Triviality Theorem~\ref{thm:RTT} and the Riso-Equivalence Theorem~\ref{thm:RET}). The most basic setting is the one where $\LL$ is the pure valued field language and $\TT$ is the $\LL$-theory of henselian valued field of equi-characteristic $0$.
However, our results also hold in various extensions of that language, provided that $\LL$-definable sets still behave in a ``tame'' way. The notion of tameness we impose is $1$-h-minimality from \cite{iCR.hmin}:

\begin{hyp}
\label{hyp:KandX_can}
For the entire paper, we fix a language $\LL$ and an $\LL$-theory $\TT$ as in Convention~\ref{conv:LL} (in particular, every model of $\TT$ is a valued field of equi-characteristic $0$) and we assume that $\TT$ is $1$-h-minimal in the sense of \cite[Definition~1.2.3]{iCR.hmin}.%
\footnote{Formally, in \cite{iCR.hmin}, the language is assumed to be one-sorted (consisting only of the valued field sort), but if one wants, one can allow multi-sorted languages $\LL$ by applying
\cite[Definition~1.2.3]{iCR.hmin} to the induced structure on $\VF$.}
\end{hyp}

Note that while there is a notion of $1$-h-minimality for valued fields of mixed characteristic (see \cite{iCRV.hmin2}) and it seems plausible that a version of our results holds in such a setting, our proof works only in equi-characteristic $0$.

Instead of recalling the definition of $1$-h-minimality, we just recall some of the main examples, and we cite the consequences that we will need. (In addition to the consequences cited below, we cite one more in Lemma~\ref{lem:tstrat:hyp}.)

\begin{exa}
The following theories are $1$-h-minimal, by \cite[Section~6]{iCR.hmin}:
\begin{enumerate}
 \item the theory of henselian valued fields of equi-characteristic $0$.
 \item 
 the theory of henselian valued fields of equi-characteristic $0$ with analytic $\mathcal A$-structure, for any separated Weierstraß system $\mathcal A$ in the sense of \cite{CL.analyt}. (This is an abstract version of extending $\Lvf$ by analytic functions; see \cite[Section~4.4]{CL.analyt} for many concrete examples.)
 \item The theory of any $\Tnoval$-convex valued field in the sense of \cite{DL.Tcon1}, where $\Tnoval$ is a power-bounded o-minimal theory. (The notion of power-bounded is recalled in Definition~\ref{defn:powbd}.
 A $\Tnoval$-convex valued field is obtained by 
 taking a model $K \models \Tnoval$ and turning it into a valued field using the convex closure of an elementary substructure $K_0 \prec_{\Lnoval} K$ as  a valuation ring, provided that this yields a non-trivial valuation; see also the proof of Lemma~\ref{lem:consKK0}.)
\end{enumerate}
In addition, $1$-h-minimality is preserved 
under arbitrary expansions of $\RV^\eq$ (see \cite[Theorem 4.1.19]{iCR.hmin}). In particular, all of the above examples stay $1$-h-minimal if one e.g.\ adds an angular component map $\ac\colon \VF \to \RF$ to the language (since $\ac$ factors over $\RV$). 
\end{exa}

\begin{rmk}\label{rmk:hmin:const}
Adding $\VF$-constants to the language also preserves $1$-h-minimality. We will use this implicitly throughout this paper.
(Be careful however that adding constants from arbitrary sorts of $\Leq$ might destroy $1$-h-minimality.)
\end{rmk}

\begin{prop}[{\cite[Proposition~5.3.4]{iCR.hmin}}]\label{prop:hmin:dim}
In $1$-h-minimal theories, 
the topological dimension of definable sets (see Definition~\ref{defn:dim}) behaves as expected with respect to unions and definable maps, it is definable, $0$-dimensional sets are finite, and if $\bX \subseteq \VF^n$ is a non-empty $\LL$-definable set and $\Cl(\bX)$ is its topological closure, then $\dim (\Cl(\bX) \setminus \bX) < \dim \bX$.
\end{prop}

Since $1$-h-minimal theories eliminate $\exists^\infty$ in the $\VF$-sort (by \cite[Lemma~2.5.2]{iCR.hmin}), we even have a uniform bound on the cardinality of $0$-dimensional sets:

\begin{lem}\label{lem:Einf}
If $\bX \subseteq \VF^n$ is an $\LL$-definable set of dimension $0$, then there exists an $m \in \NN$ such that $\#\bX(K) \le m$ for every model $K \models \TT$.
\end{lem}

In large parts of the paper, we would like to only consider models $K \models \TT$ which are spherically complete (meaning that each nested family of balls in $K$ has non-empty intersection); indeed, this is necessary for the riso-triviality space (Definition~\ref{defn:rtsp}) to exist.
By \cite{iBW.sph}, assuming $K$ to be spherically complete is not a severe restriction:

\begin{thm}[{\cite[Theorem~1]{iBW.sph}}]\label{thm:sph}
Every model $K \models \TT$ has a spherically complete immediate elementary extension.
\end{thm}

(More precisely, \cite[Theorem~1]{iBW.sph} has weaker assumptions and a stronger conclusion.)
In particular, we have the following:

\begin{rmk}\label{rmk:sph:unique}
By Theorem~\ref{thm:sph},
an $\Leq$-definable set $\bX$ is uniquely determined by the sets $\bX(K)$,
where $K$ only runs over the spherically complete models of $\TT$ (and similarly for any other kind of $\Leq$-definable object.)
\end{rmk}

\subsection{Risometries and riso-triviality}
\label{sec:riso}

In this subsection, we introduce the central notion of this paper, namely riso-triviality (Definition~\ref{defn:riso-triv}). This is notion is completely independent of model theory, but for it to have good properties, we need the valued field we work in to be spherically complete (as we will see in Section~\ref{sec:prel}).
So in this subsection, we (only) assume that $K$ is  
a spherically complete valued field $K$ of equi-characteristic $0$.

In several of the following definitions, we denote certain subsets of $K^n$ by $B$ (or $B_i$). All those definitions will almost only be applied when each such $B$ is a ball, or maybe the intersection of a ball with an affine subspace of $K^n$. (More general sets $B$ only appear in a few specific places in some proofs.)

\begin{defn}\label{defn:risom}
Let $B_1, B_2$ be subsets of $K^n$ for some $n \ge 0$.
\begin{itemize}
    \item A \emph{risometry} $\phi: B_1 \ra B_2$ 
is a bijection such that, for all $x_1,x_2 \in B_1$ with $x_1 \ne x_2$, we have 
\begin{equation}\label{eq:riso}
|\phi(x_1) - \phi(x_2)-(x_1 - x_2)| < |x_1-x_2|.
\end{equation}
\item If we additionally have maps $\chi_i\colon B_i \to S$ for $i=1,2$, where $S$ is an arbitrary set, then by a \emph{risometry from $\chi_1$ to $\chi_2$}, we mean a risometry $\phi\colon B_1 \to B_2$ satisfying $\chi_1 = \chi_2 \circ \phi$. If such a risometry $\phi$ exists, we say that $\chi_1$ is \emph{risometric} to $\chi_2$.
\item If we are given just one set $B \subseteq K^n$, and if $\chi_1, \chi_2$ are two maps whose domains contain $B$, by a \emph{risometry from $\chi_1$ to $\chi_2$ on $B$}, we mean a risometry between the restrictions $\chi_1|_B$ and $\chi_2|_B$.
\end{itemize}
\end{defn}

Condition (\ref{eq:riso}) is equivalent to the condition that, for all pairs $x_1, x_2 \in B_1$, we have
\[
\rv^{(n)}(\phi(x_1) - \phi(x_2)) = \rv^{(n)}(x_1-x_2).
\]
This is weaker than $\phi(x_1) - \phi(x_2) = x_1-x_2$ and stronger than
$|\phi(x_1) - \phi(x_2)| = |x_1-x_2|$, so any translation is a risometry, and any risometry is an isometry. The latter implies that if there exists a risometry between balls $B_1$ and $B_2$, then those balls are necessarily of the same radius.
This also motivates the term ``risometry'': In the same way as an isometry preserves the valuation of differences -- often denoted by $v(x_1 - x_2)$ --, a \underline{r}isometry preserves $\operatorname{\underline{r}v}(x_1 - x_2)$.

Typically, the risometries we will need will piecewise resemble the following examples.

\begin{exa}\label{exa:riso}
A linear map $\alpha \in K^{n \times n}$ is a risometry if and only if it is invertible in $\valring(K)^{n\times n}$ and it induces the identity map on the residue field, i.e., if and only if
$\alpha \in \ker(\GL_n(\valring(K)) \to \GL_n(\RF(K)))$; see \cite[Remark~2.13]{i.whit}.
\end{exa}

\begin{exa}\label{exa:vert:transl}
Suppose that $B_1, B_2 \subseteq K$ are balls and that $f\colon B_1 \to B_2$ is a map satisfying $|f(x) - f(x')| < |x - x'|$ for every pair of distinct elements $x,x' \in B_1$. Additionally fix some $y_0 \in B_2$. Then for $B := B_1 \times B_2$, one easily verifies that the map
$\phi\colon B \to B, (x,y) \mapsto (x, y - y_0 + f(x))$ is a risometry (see Figure~\ref{fig:vert:transl}). Note also that it sends the horizontal line segment $B_1 \times \{y_0\}$ to the graph of $f$.
\end{exa}

\begin{figure}
 \includegraphics{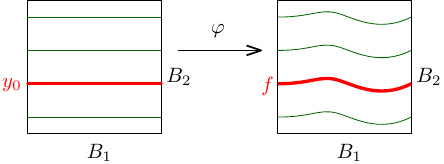}
 \caption{The risometry in Example~\ref{exa:vert:transl} turns a straight line segment into the graph of a function $f$, by translating things vertically.}\label{fig:vert:transl}
\end{figure}

\begin{rmk}
Later, Definition~\ref{defn:risom} will be applied to definable sets $B_i$ and definable maps $\chi_i$ (and the set $S$ will always be a sort of $\RV^{\eq}$).
We emphasize that even then, the involved risometries will not be imposed to be definable, unless stated otherwise. (Example~\ref{exa:def:nnd} shows how this makes a difference.)
\end{rmk}

\begin{conv}\label{conv:sets}
Often, we will have sets $X_i \subseteq B_i$ and want to consider a risometry $\phi\colon B_1 \to B_2$ sending $X_1$ to $X_2$. We will do so by applying Definition~\ref{defn:risom} to the indicator functions
$\mathbbm1_{X_i} : B_i \ra \{0,1\}$. In this way, also all of the following
definitions can be applied to sets instead of maps. Similarly,
we apply the definitions to tuples $(X_1, \dots, X_\ell, \chi_1, \dots, \chi_{\ell'})$ of sets $X_i$ and maps $\chi_j$, meaning that we consider the map
$x \mapsto (\mathbbm1_{X_1}(x), \dots, \mathbbm1_{X_\ell}(x), \chi_1(x), \dots, \chi_{\ell'}(x))$.
\end{conv}

\begin{defn}
Suppose that $B \sseteq K^n$ is a subset and $\chi$ is a map whose domain contains $B$ (and with arbitrary range). Let $V$ be a $K$-vector subspace of $K^n$. We say that $\chi$ is \emph{$V$-translation invariant on $B$} if, for every $x,x' \in B$ with $x-x' \in V$, we have \[\chi(x) = \chi(x').\]
 If the domain of $\chi$ is equal to $B$, we may omit writing ``on $B$''.
 \end{defn}
 
 The above definition will mainly be applied when $B$ is a ball in $K^n$, or possibly a ball in an affine subspace $a + W \subseteq K^n$; in the latter case we will have $V \subseteq W$.
 
\begin{defn}
Given any vector subspace $V\sseteq K^n$ (for $n \ge 0$), we write $\res(V)$ for the image
of $V \cap \valring(K)^n$ under $\res: \valring(K)^n \ra \RF(K)^n$.
(Note that $\res(V)$ is a vector subspace of $\RF(K)^n$ satisfying $\dim_K V = \dim_{\RF(K)} \res(V)$.) Conversely, given a vector subspace $\bV \subseteq \RF(K)^n$, we say that a vector subspace $V\subseteq K^n$ is a \emph{lift} of $\bar{V}$ if $\res(V) = \bV$.
\end{defn}

We now introduce the central notion of this paper, namely riso-triviality.

\begin{defn}\label{defn:riso-triv}
Suppose that $B \subseteq K^n$ (for some $n \ge 0$) is a subset not contained in any affine subspace of $K^n$ and that $\chi$ is a map whose domain contains $B$.
 \begin{enumerate}
     \item 
Let $\bV \sseteq \RF(K)^n$ be a vector sub-space. We say that $\chi$ is \emph{$\bV$-riso-trivial on $B$} if there exists a lift $V \sseteq K^n$ of $\bV$ and a risometry $\phi: B \ra B$ such that $\chi|_B \circ \phi$ is $V$-translation invariant. In this situation, the risometry $\phi$ is called a $V$-\emph{straightener} of $\chi$ on $B$. (We often omit $V$, if $\bV$ is clear from the context and we do not care about the specific lift $V$.) We also say that $\phi$ \emph{witnesses} $\bar V$-riso-triviality of $\chi$ (on $B$).
 \item
 Given $0 \le r \le n$,
  the map $\chi$ is called \emph{$r$-riso-trivial} on $B$, if it is $\bV$-riso-trivial for some vector subspace $\bV \sseteq \RF(K)^n$ of dimension $r$.
  We say that $\chi$ is \emph{not riso-trivial (at all)} on $B$ if it is not $1$-riso-trivial on $B$.
 \item If the domain of $\chi$ is equal to $B$, we may just say ``riso-trivial'' instead of ``riso-trivial on $B$''.
\end{enumerate}
\end{defn}

\begin{rmk}\label{rmk:lift:no:matter}
In (1), the choice of the lift $V$ does not matter:
If $\chi$ is $\bV$-riso-trivial on $B$, then for every lift $V$ of $\bV$, there exists a $V$-straightener of $\chi$ on $B$. This follows from the
fact that if $V$ and $V'$ are two different lifts, then there exists a linear map $\alpha \in
\GL_n(\valring(K))$ which is a risometry (see Example~\ref{exa:riso}) and which
sends $V$ to $V'$; see also \cite[Remark~2.13]{i.whit}.
\end{rmk}

One can naturally make sense of the notion of riso-triviality within an affine subspace of $K^n$, but in that case, one has to be careful to choose the lifts in Definition~\ref{defn:riso-triv} (1) appropriately. While the affine space we are thinking of should always be clear from the context, we prefer to make this precise, as follows.

\begin{defn}\label{defn:riso-triv-sub}
If $B \subseteq K^n$ is contained in 
an affine subspace $F = a + W \subseteq K^n$ (for a point $a \in K^n$ and a vector subspace $W \subseteq K^n$), but not in any proper affine subspace of $F$, we use the terminology from Definition~\ref{defn:riso-triv} with the modification that we only consider subspaces $\bV$ of $\RF(K)^n$ which are subspaces of $\res(W)$, and in (1), we require the lift $V$ of $\bV$ to be a subspace of $W$.
\end{defn}

As stated in Convention~\ref{conv:sets}, Definition~\ref{defn:riso-triv} can also be applied to tuples of maps and sets. For example, a set $X \subseteq K^n$ is said to be $\bV$-riso-trivial on a ball $B \subseteq K^n$ if the indicator function $\mathbbm1_X$ of $X$ is $\bV$-riso-trivial on $B$.

\begin{rmk}
Note that asking a tuple $(\chi_1, \dots, \chi_{\ell'})$ of maps to be $\bar V$-riso-trivial is strictly stronger than asking each $\chi_i$ individually to by $\bar V$-riso-trivial. As an example, pick any
$\bar V \subseteq \RF(K)^2$ of dimension $1$ and pick any two different lifts $V_1, V_2 \subseteq K^2$. Then $V_1$ and $V_2$ are $\bar V$-riso-trivial on $K^2$ individually, but $(V_1, V_2)$ is not.
\end{rmk}

We already state a key property of the notion of riso-triviality. Its proof will essentially be postponed to Section~\ref{sec:prel}.

\begin{prop}\label{prop:rtsp}
Recall that we assume that $K$ is a spherically complete valued field of equi-characteristic $0$. Suppose that $B \subseteq K^n$ is a ball (for some $n \ge 1$) and that $\chi\colon B \to S$ is a map (into any set $S$). Then there exists a maximal vector sub-space $\bar V \subseteq \RF(K)^n$ for which $\chi$ is $\bar V$-riso-trivial on $B$.
\end{prop}

\begin{proof}
It suffices to show that if $\chi$ is $\bV_1$-riso-trivial and $\bV_2$-riso-trivial for some $\bV_1, \bV_2 \subseteq \RF(K)^n$, then it is even $(\bV_1+ \bV_2)$-riso-trivial. This is the statement of Lemma~\ref{lem:V-additive}.
\end{proof}

Proposition~\ref{prop:rtsp} gives rise to the following natural definition:

\begin{defn}\label{defn:rtsp}
Let $B \subseteq K^n$ be a ball (for some $n \ge 1$) and let $\chi$ be a map whose domain contains $B$.
The (unique) maximal vector subspace $\bV$ of $\RF(K)^n$ such that $\chi$
is $\bV$-riso-trivial on $B$ is called the \emph{riso-triviality space} of $\chi$ on $B$. We denote it by $\rtsp_B(\chi)$.
We also apply this notion when $B$ is a ball in an affine subspace $F = a + W \subseteq K^n$ of dimension $\dim F \ge 1$ (using Definition~\ref{defn:riso-triv-sub}).
\end{defn}

In other words, for any subspace $\bW \subseteq \RF(K)^n$, $\chi$ is $\bW$-riso-trivial on $B$ if and only if $\bW \subseteq \rtsp_B(\chi)$.

\begin{conv}\label{conv:sets:rtsp}
As in Convention~\ref{conv:sets}, if $(X_1, \dots, X_\ell, \chi_1, \dots, \chi_{\ell'})$ is a tuple of sets $X_i$ and maps $\chi_j$, we write $\rtsp_B(X_1, \dots, X_\ell, \chi_1, \dots, \chi_{\ell'})$ for the riso-triviality space on $B$ of the map
$x \mapsto (\mathbbm1_{X_1}(x), \dots, \mathbbm1_{X_\ell}(x), \chi_1(x), \dots, \chi_{\ell'}(x))$.
\end{conv}

To provide some intuition about riso-triviality, we end this section by explaining how, in the case of algebraic sets, it relates to smoothness.

\begin{exa}\label{exa:smooth}
Suppose (for concreteness) that $K = \CCt[t^\QQ]$ and that $X = \bX(K) \subseteq K^n$ is a subvariety of pure dimension $d$ defined over $\CC$.
If $X$ is smooth at some point $a \in X$, then there exists a ball $B \subseteq K^n$ around $a$ on which $X$ is $d$-trivial, and $\rtsp_B(X)$ is equal to the residue $\res(T_aX)$ of the tangent space of $X$ at $a$. In the case where $a \in \bX(\CC)$ (which we consider as a subset of $X$), we can take this ball $B$ to be $B_{<1}(a) = a + \maxid(K)^n$ (and we obtain $\rtsp_B(X) = \res(T_aX) = T_a(\bX(\CC))$).
Conversely, if $X$ is $d$-trivial on $B_{<1}(a)$ for every $a \in \bX(\CC)$, then $X$ is smooth.
\end{exa}

A precise proof of these claims is given in Proposition~\ref{prop:riso-mani}. (For the first part, apply the proposition to the smooth locus. For the ``conversely'' part use also Remark~\ref{rmk:C1:smooth}.)
Here, we only give the main idea of how smoothness at $a \in \bX(\CC)$ implies $\rtsp_{B}(X) = T_a(\bX(\CC))$ for $B = B_{<1}(a)$:
Suppose for simplicity that $n = 2$ and $d = 1$, and suppose without loss that $a = 0$ and that $T_a(\bX(\CC)) = \CC \times \{0\}$.
Then one can manually verify (using Hensel's Lemma) that $X \cap \maxid(K)^2$ is the graph of a function $f\colon \maxid(K) \to \maxid(K)$ satisfying the assumption from Example~\ref{exa:vert:transl}.
The risometry provided by that example witnesses $(\CC \times \{0\})$-triviality of $\bX(K)$ on $\maxid(K)^2$.

\begin{qu}
Does the converse part of Example~\ref{exa:smooth} also hold point by point, i.e., does $d$-triviality on a neighbourhood of $a$ imply smoothness at $a$? Note that in a semi-algebraic setting, such a converse would be false, as the origin in Example~\ref{exa:PS} shows.
\end{qu}

\section{Definability of riso-triviality}
\label{sec:main-results}

\subsection{Statement of the definability results}
\label{sec:thms}

We are now ready to state our main definability results precisely.
We use notation from Section~\ref{sec:notn} and assume Hypothesis~\ref{hyp:KandX_can} about the language $\LL \supseteq \Lvf$
and the $\LL$-theory $\TT$ (namely, $\TT$ is $1$-h-minimal).
Recall that by a map into $\RV^\eq$, we mean a map into any sort of $\RV^\eq$. Recall also that $\VF^n$ is considered as a ball, but singletons are not.

The first result is about definability of the riso-triviality space. While we are mostly interested in $\rtsp_B(X)$ for some definable set $X \subseteq K^n$, we more generally consider $\rtsp_B(\chi)$ for a definable map $\chi\colon K^n \to \RV(K)^\eq$. This implies the result for $\rtsp_B(X)$ by letting $\chi := \mathbbm1_{X}$ be the indicator function of $X$ (see Convention~\ref{conv:sets:rtsp}).

\begin{thm}[Riso-Triviality Theorem]
\label{thm:RTT}
Suppose that $\LL$ and $\TT$ satisfy Hypothesis~\ref{hyp:KandX_can}.
Let $\bB \subseteq \VF^n$ be an $\LL$-definable ball (for $n \ge 1$) and $\bchi \colon \bB \to \RV^{\eq}$ an $\Leq$-definable map. Then there exists an $\Leq$-definable vector subspace $\bbU \subseteq \RF^n$ such that for every spherically complete model $K \models \TT$, we have $\rtsp_{\bB(K)}(\bchi_K) = \bbU(K)$.
\end{thm}

By adding a constant symbol to the language, one obtains the following version of the theorem for definable families.

\begin{notn}
Given $0 \le d \le n$,
we write $\bGr_d^n(\RF)$ for the Grassmannian of $d$-dimensional vector subspaces of $\RF^n$, considered as an $\Leq$-definable subset of the power set $\pow(\RF^n)$ (in the sense of Convention~\ref{conv:sets:of:sets}).
\end{notn}

\begin{cor}
\label{cor:RTT}
Let $\bQ$ be any $\Leq$-definable set and let 
$\fB \subseteq \VF^n \times \bQ$ be an $\Leq$-definable family of balls in $\VF^n$ parametrized by $\bQ$, i.e., for every $K \models \TT$ and every $q \in \bQ(K)$, the fiber $\fB(K)_q \subseteq K^n$ is a ball.
Let moreover $\bchi\colon \fB \to \RV^{\eq}$ be an $\Leq$-definable map.
Then there exists an $\Leq$-definable map $\bQ \to \bigcup_{d=0}^n\bGr_d^n(\RF)$ sending, for every
spherically complete model $K \models \TT$,
$q \in \bQ(K)$ to $\rtsp_{\fB(K)_q}(\bchi_{K}(\cdot, q))$.
\end{cor}

\begin{proof}
Without loss, $\bQ \subseteq \VF^m$ for some $m$. (Otherwise, replace it by a preimage in $\VF^m$.)
Add a constant for $q$ to the language and add ``$q \in \bQ$'' to the theory. This new theory is still $1$-h-minimal (see Remark~\ref{rmk:hmin:const}), and its models correspond exactly to pairs $(K, q)$ with $K \models \TT$, $q \in \bQ(K)$. Apply Theorem~\ref{thm:RTT} to obtain an $\Leq(q)$-definable $\bbU \in \bigcup_{d=0}^n\bGr_d^n(\RF)$. Consider the formula defining $\bbU$ as an $\Leq$-formula, with $q$ a variable; this defines the desired map.
\end{proof}

A special case of the above corollary is that $\rtsp_B(\chi)$ depends definably on $B$; in particular, we obtain a partition of the set of all closed balls according to the dimension of the riso-triviality space. While it is often more handy to use Corollary~\ref{cor:RTT} instead of that partition, we nevertheless introduce a name for it, since it is the central idea behind the entire paper.

\begin{cor}
\label{cor:riso-tree}
Let $\bchi \colon \VF^n \to \RV^{\eq}$ be an $\Leq$-definable map, for $n \ge 1$.
Then there exists a partition of the set of all closed subballs of $\VF^n$ (considered as a sort of $\Leq$) into 
$\Leq$-definable sets $\bTr_0, \dots, \bTr_n$ such that for every spherically complete $K \models \TT$, every closed ball $B \subseteq K^n$, and every $r$, we have $B \in \bTr_r(K)$ if and only if $\dim \rtsp_B(\bchi_K) = r$.
\end{cor}

\begin{defn}\label{defn:riso-tree}
The partition $(\bTr_r)_{0\le r\le n}$ from Corollary~\ref{cor:riso-tree} is called the \emph{riso-tree} of $\bchi$. We also apply this terminology within spherically complete models $K \models \TT$: $(\bTr_r(K))_{0\le r\le n}$ is the \emph{riso-tree} of $\bchi_K$.
\end{defn}

Note that even though Corollary~\ref{cor:riso-tree} specifies $\bTr_r(K)$ only for spherically complete models $K \models \TT$, this entirely determines $\bTr_r$ by Remark~\ref{rmk:sph:unique}.

As explained in the introduction,
one can think of $(\bTr_r)_{0 \le r \le n}$ as a better replacement for a (canonical) stratification of $\bX$; better in the sense that it contains more information.
Recall also that in some sense, these $\bTr_r$ form a stratification of the set of all closed subballs of $\VF^n$, and that, instead of formulating precisely for all $r$ what this means, we just formulate it for $r=0$ (in Lemma~\ref{lem:Crig}) and reduce any questions about bigger $r$ to the $r=0$ case by restricting to suitable fibers; this reduction works due to Lemmas~\ref{lem:fiber_compat}, \ref{lem:fiber_check}, and \ref{lem:fiber_triv}.

\medskip

At the same time as we prove Theorem~\ref{thm:RTT}, we also prove that the existence of a risometry is a first-order condition; this is our second main definability result. Again, we formulate it for definable maps into $\RV^\eq$, which can be taken to be indicator functions of definable sets, if this is what one is interested in.

\begin{thm}[Riso-Equivalence Theorem]
\label{thm:RET}
Suppose that $\LL$ and $\TT$ satisfy Hypothesis~\ref{hyp:KandX_can}.
Let $\bB \subseteq \VF^n$ be an $\LL$-definable ball (for $n \ge 1$).
Let $\bchi_1, \bchi_2\colon \bB \to \RV^{\eq}$ be $\Leq$-definable maps.
Then there exists an $\LL$-sentence $\psi$ such that for every spherically complete model $K \models \TT$, we have
$K \models \psi$ if and only if $\bchi_{1,K}$ and $\bchi_{2,K}$ are risometric.
\end{thm}

Again, we also obtain a family version of that result, namely stating that the risometry type within a definable family of maps is definable. By combining with a result from \cite{i.whit} about t-stratifications, we moreover obtain that the risometry type ``lives in $\RV^\eq$'', in the following sense:

\begin{cor}\label{cor:RET}
Let $\bQ$, $\fB \subseteq \VF^n \times \bQ$ and $\bchi\colon \fB \to \RV^\eq$ be as in Corollary~\ref{cor:RTT}.
Then there exists an $\Leq$-definable map $\brho\colon \bQ \to \RV^\eq$ such that for every spherically complete model $K \models \TT$ and every $q, q' \in \bQ(K)$, we have $\brho_K(q) = \brho_K(q')$ if and only if $\bchi_K(\cdot, q)$ and $\bchi_K(\cdot, q')$ are risometric.
\end{cor}

We postpone the proof of this corollary to the end of Subsection~\ref{sec:tstrat}, after we recalled the notion of t-stratifications.

\subsection{A related notion: t-stratifications}
\label{sec:tstrat}

Our notion of riso-triviality is very similar to the notion of translatability from \cite{i.whit}: the only difference is that \cite{i.whit} requires the risometries to be definable. Despite this difference, some of the results from \cite{i.whit} are useful ingredients for the present paper; in this subsection, we recall those results.
As usual, we assume Hypothesis~\ref{hyp:KandX_can}, and we let $K$ be a model of $\TT$. (In this subsection, we do not need to impose that $K$ is spherically complete.)

Firstly, note that in \cite{iCR.hmin}, it has been shown that the central assumption from \cite{i.whit} holds in this setting:

\begin{lem}[{\cite[Lemma~5.5.4]{iCR.hmin}}]\label{lem:tstrat:hyp}
Hypothesis~2.21 from \cite{i.whit} holds in $\TT$.
\end{lem}

We start by recalling the notion of translatability from \cite{i.whit}; however, we prefer to call it ``definable riso-triviality''. 
Note that a version of Proposition~\ref{prop:rtsp} also holds for this notion: there exists a maximal space of definable riso-triviality.

\begin{defn}\label{defn:driso-triv}
Suppose that $B \subseteq K^n$ is either an open or a closed ball (where $n \ge 1$) and that $\chi$ is a definable map whose domain contains $B$.
\begin{enumerate}
\item
Given an $\RF(K)$-vector subspace $\bV \subseteq \RF(K)^n$, we say that $\chi$ is \emph{definably $\bV$-riso-trivial} on $B$ if there exists
a definable (with arbitrary parameters) straightener $\phi\colon B \to B$ witnessing $\bar V$-riso-triviality of $\chi$ on $B$. We say that $\chi$ is definably $r$-riso-trivial on $B$ if it is definably $\bV$-riso-trivial on $B$ for some $r$-dimensional $\bV \subseteq \RF(K)^n$.
\item
The maximal $\RF(K)$-vector subspace $\bV \subseteq \RF(K)^n$ for which $\chi$ is definably riso-trivial on $B$ is denoted by $\drtsp_B(\chi)$. (This exists by \cite[Lemma~3.3]{i.whit}; in \cite{i.whit}, $\drtsp_B(\chi)$ is denoted by $\operatorname{tsp}_B(\chi)$.)
\item We also apply those notions to tuples of maps and sets, as explained in Convention~\ref{conv:sets}.
\end{enumerate}
\end{defn}

Here is an example showing that it does make a difference whether one requires the straightener to be definable or not:

\begin{exa}\label{exa:def:nnd}
Set $K := \kt$, let $\bar f\colon k \to k$ be any function definable in the ring language on $k$ and set $X := \{(x,ty)\in k[[t]]^2 \mid \res(y) = \bar f(\res(x))\}$. Then $X$ is always $\bU$-riso-trivial on $B:=k[[t]]^2$, where $\bU := k\times\{0\}$, but we claim that it is not necessarily definably $\bU$-riso-trivial. Indeed, if it is, then one can find a definable straightener $\psi\colon B \to B$ preserving the $x$-coordinate and with $\psi^{-1}(X) = k[[t]] \times t k[[t]]$ (by \cite[Lemma~3.6]{i.whit}; see Lemma~\ref{lem:fiber_compat} for a version of that result in the non-definable setting). Using $\psi$, we obtain a definable lift $f$ of $\bar f$ to a map from $k[[t]]$ to $k[[t]]$, namely defined by $(x, t\cdot f(x)) = \psi((x, 0))$. (Here, ``lift'' means that $\res(f(x)) = \bar f(\res(x))$ for every $x \in k[[t]]$.) However, if $k$ is sufficiently evil, then not every definable function from $k$ to $k$ can be lifted. (See \cite[Example~3.15]{i.whit} for a concrete example.)
\end{exa}

The following is essentially \cite[Definition~3.12]{i.whit} (and also \cite[Definition~5.5.2]{iCR.hmin}); the differences are that Definition~\ref{defn:tstrat} is formulated uniformly for all models of $\TT$ and that we allow the domain $\bB_0$ to be a ``cut ball''. (Balls in \cite{i.whit} are imposed to be either open or closed.)

\begin{defn}\label{defn:tstrat}
Let $\bB_0 \subseteq \VF^n$ be an $\LL$-definable ball for some $n \ge 1$.
\begin{enumerate}
 \item 
An \emph{$\LL$-definable t-stratification} of $\bB_0$ is a partition $(\bS_i)_{0 \le i \le n}$ of $\bB_0$ into $\LL$-definable sets $\bS_0, \dots, \bS_n$ with the following properties:
\begin{enumerate}
 \item $\dim \bS_r \le r$ for every $r$.
 \item For every model $K \models \TT$, every $r$, and every ball $B \subseteq \bS_{\ge r}(K)$ which is either open or closed, $(\bS_i(K))_i$ is definably $r$-riso-trivial on $B$. Here, $\bS_{\ge r} := \bS_{r}\cup\dots\cup \bS_n$.
\end{enumerate}
\item
Let now additionally $\bchi\colon \bB_0 \to \RV^\eq$ be an $\Leq$-definable map. We say that the t-stratification $(\bS_i)_{0 \le i \le n}$ \emph{reflects} $\bchi$ if for every $K, r, B$ as in (2), $((\bS_i(K))_i, \bchi_K|_B)$ is definably $r$-riso-trivial on $B$.
\end{enumerate}
\end{defn}

\begin{rmk}\label{rmk:tstrat}
As explained below \cite[Definition~3.12]{i.whit},
if $(\bS_i)_{i}$, $\bchi$, $K$ and $B$ are as in Definition~\ref{defn:tstrat} (with $(\bS_i)_{i}$ reflecting $\bchi$), then we have $\drtsp_B((\bS_i(K))_i) = \drtsp_B((\bS_i(K))_i, \bchi_K)$, and the dimension of this space is equal to the minimal $r$ with $\bS_r(K) \cap B \ne \emptyset$. In particular,
$\drtsp_B((\bS_i(K))_i) \subseteq \rtsp_B(\bchi_K)$.
\end{rmk}

The following is essentially the main result of \cite{i.whit}, but in the more general context of $1$-h-minimal theories:

\begin{thm}\label{thm:tstrat}
Given any $\Leq$-definable map $\bchi\colon \bB_0 \to \RV^\eq$ on an $\LL$-definable ball $\bB_0 \subseteq \VF^n$ (for some $n \ge 1$), there exists an $\LL$-definable t-stratification of $\bB_0$ reflecting $\bchi$.
\end{thm}

\begin{proof}
Extend $\bchi$ to an $\Leq$-definable map $\bchi'\colon \VF^n \to \RV^\eq$ (e.g.\ sending $\VF^n \setminus \bB_0$ to $0$).
By Lemma~\ref{lem:tstrat:hyp},
\cite[Corollary~4.13]{i.whit} can be applied to $\bchi'$, yielding an $\LL$-definable t-stratification $(\bS_i)_{i}$ of $\VF^n$ reflecting $\bchi'$. Then $(\bS_i \cap \bB_0)_{i}$ is the desired t-stratification of $\bB_0$ reflecting $\bchi$.
\end{proof}

One reason t-stratifications will be useful in the proof of our main results is the following lemma:

\begin{lem}[{\cite[Lemma~3.14]{i.whit}}]\label{lem:3.14}
If $(\bS_i)_{i}$ is an $\LL$-definable t-stratification of some $\LL$-definable ball $\bB_0 \subseteq \VF^n$,
then the map sending, for each $K \models \TT$, an open or closed ball $B \subseteq \bB_0(K)$ to $\drtsp_B((\bS_i(K))_i)$ is $\Leq$-definable.
\end{lem}

Note that this lemma would be false if instead of $\drtsp_B((\bS_i(K))_i)$, one would consider $\drtsp_B(\bchi(K))$ for an arbitrary $\Leq$-definable map $\bchi\colon \bB_0 \to \RV^\eq$
(even if $\bchi$ is the indicator function of an $\LL$-definable set), as \cite[Example~3.15]{i.whit} shows. In particular, for the Riso-Triviality Theorem to be true, it is crucial to allow non-definable risometries.

This example crucially relies on the existence of definable functions in the residue field which cannot be lifted to the valuation ring. This might seem pretty artificial, since
in the cases we are mostly interested in, namely when the residue field is algebraically closed or real closed, all functions can be lifted. One might therefore wonder whether in those cases, one could stick to the setting of \cite{i.whit} of considering only definable risometries. However, our proof of the Riso-Triviality Theorem would certainly not work anymore, and we were not even able to prove the following statement \emph{a posteriori}:

\begin{qu}\label{qu:autodef}
Suppose that every definable function $f\colon \RF^n(K) \to \RF(K)$
can be lifted to a definable function $\tilde f\colon \valring^n(K) \to \valring(K)$ (in the sense that $\res \circ \tilde f = f \circ \res$), is it true that the existence of an arbitrary risometry
from $\chi_1$ to $\chi_2$ (for definable $\chi_i\colon B_i \subseteq K^n \to \RV^\eq(K)$) implies the existence of a definable risometry from $\chi_1$ to $\chi_2$?
\end{qu}

\medskip

We can now give the proof of Corollary~\ref{cor:RET}, about isometry types ``living in $\RV^\eq$'':

\begin{proof}[Proof of Corollary~\ref{cor:RET}]
Recall that we are given an $\Leq$-definable set $\bQ$ and an $\Leq$-definable map $\bchi\colon \fB \subseteq \VF^n \times \bQ \to \RV^\eq$. Without loss $\bQ \subseteq \VF^m$ for some $m$.
By applying Theorem~\ref{thm:RET} in a language $\LL'$ with constants for $q$ and $q'$ added, and in the theory $\TT' = \TT \cup \{q \in \bQ, q' \in \bQ\}$, we obtain an $\LL$-definable equivalence relation $\bsim$ on $\bQ$ such that for every spherically complete $K \models \TT$, we have
\[
q \bsim_K q' \iff 
\text{$\bchi_K(\cdot, q)$ and $\bchi_K(\cdot, q')$ are risometric}.
\]
Existence of t-stratifications (Theorem~\ref{thm:tstrat}) applied in $\LL'$ yields an $\LL$-definable family of t-stratifications
$(\fS_i)_i$ of $\bchi$ (both considered as families parametrized by $\bQ$), i.e., we have $\fS_i \subseteq \VF^n \times \bQ$, and
for every $K$ and every $q \in \bQ(K)$,
$(\fS_i(K)_q)_i$ is a t-stratification reflecting $\bchi_K(\cdot, q)$.

Applying \cite[Proposition 3.19]{i.whit} (which we can, by Lemma~\ref{lem:tstrat:hyp}) to $(\fS_i)_i$ and $\bchi$ yields an 
$\Leq$-definable map $\tilde\brho\colon \bQ \to \RV^\eq$ such that, for every $K$ and every $q, q' \in \bQ(K)$,
$\tilde\brho_K(q) = \tilde\brho_K(q')$ implies that
$((\fS_{i}(K)_q)_i, \bchi_K(\cdot, q))$ and $((\fS_{i}(K)_{q'})_i, \bchi_K(\cdot, q'))$ are definably risometric. Since this in particular implies that $\bchi_K(\cdot, q)$ and $\bchi_K(\cdot, q')$ are risometric, we obtain that each $\bsim_K$-equivalence class is a union of fibers of $\tilde\brho_K$, so that we can define the desired map $\brho\colon \bQ \to \RV^\eq$ by quotienting the image of $\tilde\brho$ by (the image of) $\bsim$.
\end{proof}

\subsection{Basic properties of riso-triviality}
\label{sec:prel}

In this section, we collect a whole bunch of basic properties of riso-triviality. They are very similar to corresponding results in \cite{i.whit}, but the word ``definable'' has been removed everywhere, and whereas in \cite{i.whit}, definable spherical completeness was a key ingredient, we now use actual spherical completeness instead. No model theory is used in this entire subsection, except in some ``model theoretic remarks''. Thus: For most of this section, the only assumption is that $K$ is a spherically complete valued field of equi-characteristic $0$; for the model theoretic remarks, 
we assume that we are in the setting from Section~\ref{sec:notn}.

\begin{notn}
If $B, B' \subseteq K^n$ are disjoint sets, each of which is either a ball or a singleton, then the set $\{\rv^{(n)}(x - x') \mid x \in B, x' \in B'\}$ is a singleton; in the following lemma and its proof, we denote that element of $\RV^{(n)}(K)$ by $\rv^{(n)}(B - B')$. We define $|B-B'| \in \VG^\times(K)$ similarly. 
\end{notn}

\begin{lem}\label{lem:fin-riso}
Suppose that $B_0 \subseteq K^n$ is a ball,
that $C = B_1 \cup \dots \cup B_m \subseteq B_0$ is a union of finitely many disjoint points and balls $B_i$, and similarly 
$C' = B'_1 \cup \dots \cup B'_m \subseteq B_0$.
Then:
\begin{enumerate}
 \item There exists a risometry $\phi\colon B_0 \to B_0$ satisfying $\phi(B_i)= B'_i$ for each $i$ if and only if the following two conditions hold:
 \begin{enumerate}
 \item for each $i$, $B_i$ and $B'_i$ have the same cut-radius, i.e., $\{|x - y| \mid x,y \in B_i\} = \{|x - y| \mid x,y \in B'_i\}$;
 \item for each $i \ne j$, we have $\rv(B_i - B_j) = \rv(B'_i - B'_j)$.
\end{enumerate}
 \item
   If a risometry $\phi$ as in (1) exists and $\phi'\colon B_0 \to B_0$ is another risometry satisfying $\phi'(C)=C'$, then we also have $\phi'(B_i) = B'_i$ for every $i$.
\end{enumerate}
\end{lem}

\begin{proof}
(1) The implication from left to right is clear. For the implication from right to left, for each $i$, pick an arbitrary element $x_i \in B_i$ and an arbitrary element $x'_i \in B'_i$. For $x \in B_i$, we define $\phi(x) := x - x_i + x'_i$; this already defines a 
risometry $\phi\colon C \to C'$ sending $B_i$ to $B'_i$. This can be extended to a risometry $B_0 \to B_0$ as follows: For every maximal ball $B \subseteq B_0 \setminus C$, choose any $i = i(B)$ for which $|B - B_i| = \min_{j} |B - B_j|$. For $x \in B$, set $\phi(x) := x - x_{i(B)} + x_{i(B)}'$.
One deduces that $\rv(B - B_j) = \rv(\phi(B) - B'_j)$ for every $j$.
Using this, a short computation shows that the total map $\phi\colon B_0 \to B_0$ defined in this way is a risometry.

(2) By replacing $\phi'$ by $\phi^{-1} \circ \phi'$, we may assume that $C' = C$ and that $\phi$ is the identity, i.e., we have to show that a risometry $\phi'\colon B_0 \to B_0$ sending $C$ to itself also sends each $B_i$ to itself.
Suppose that $\phi'(B_i) = B_j$ for some $j \ne i$.
By translating and scaling, we may assume that $|B_i - B_j| = 1$ and that $B_i, B_j \subseteq \valring(K)^n$. In particular, $\res(B_i) \ne \res(B_j)$. Then $\phi'$ sends
$\valring(K)^n$ to itself and it
induces a translation $\bar\phi'\colon \RF(K)^n \to \RF(K)^n$
which sends
$\bar C := \res(C \cap \valring(K)^n)$ to itself. Since $\bar C$ is finite, $\bar\phi'$ is the identity on $\bar C$, so it does not send $\res(B_i)$ to $\res(B_j)$, contradicting $\phi'(B_i) = B_j$.
\end{proof}

\begin{modrmk}\label{rmk:fin-riso}
If, in the above lemma, $C$ and $C'$ are $\LL$-definable, then we can $\LL$-definably express whether there exists a risometry from $C$ to $C'$, namely by introducing one variable for each $B_i$ and each $B'_i$ and using that the conditions
(a) and (b) from the lemma are first-order. Moreover, we obtain that the (unique) induced map from $\{B_1, \dots, B_m\}$ to $\{B'_1, \dots, B'_m\}$ is also $\Leq$-definable.
All this also works uniformly in all models of $\TT$, provided that the number $m$ is bounded independently of the model.
\end{modrmk}

We will use the following version of the Banach Fixed Point Theorem for spherically complete fields:

\begin{lem}\label{lem:BFPT}
Suppose that $B \subseteq K^n$ is any ball and that $g:B\ra B$ is a strictly contracting map, i.e., $|g(x) - g(y)| < |x-y|$ for all $x,y \in B, x \ne y$.
Then $g$ has a unique fixed point in $B$.
\end{lem}

\begin{proof}
See for example \cite{PCR.fixPt}.
\end{proof}

From this, we deduce the following, seemingly harmless statement, but which plays a key role for everything that follows.

\begin{lem}
\label{lem:B_finite}
Suppose $B \sseteq K^n$ is a ball and $f: B\ra X \sseteq B$ is a risometry into $B$. Then $X=B$.
\end{lem}

\begin{proof}
Given $x_0 \in B$, we find a pre-image in $B$ (under $f$) by defining a strictly contracting mapping on $B$ and consider the fixed point, namely as follows. For $x \in B$ set $g(x) := x + x_0 - f(x)$. Note that a fixed point of $g$ is a pre-image of $x_0$.
For $x,y \in B$, we have 
$g(x) - g(y) = x - y - (f(x) - f(y))$, so
using that $f$ is a risometry, one deduces that $g$ is contracting.
Now Lemma \ref{lem:BFPT} yields the desired fixed point of $g$.
\end{proof}

The above lemma is the only place where we use spherically completeness. Note that it would be false without that assumption.

We can now fulfill our promise to prove that the riso-triviality space introduced in Definition~\ref{defn:rtsp} exists. More precisely, recall that the only missing ingredient (for the proof of Proposition~\ref{prop:rtsp}) is the following lemma:

\begin{lem}
 \label{lem:V-additive}
  Let $\bV_1$ and $\bV_2$ be subspaces of $\RF(K)^n$, and let $B \subseteq K^n$ be a ball. Suppose that $\chi: B \ra S$ is a map which is $\bV_1$-riso-trivial and $\bV_2$-riso-trivial. Then $\chi$ is $(\bV_1 + \bV_2)$-riso-trivial.
\end{lem}

\begin{proof}
The proof works exactly as for \cite[Lemma~3.3]{i.whit}, using Lemma~\ref{lem:B_finite} instead of 
\cite[Lemma~2.33]{i.whit}. (Using straighteners witnessing $\bV_1$-riso-triviality and $\bV_2$-riso-triviality, one can easily write down explicitly a candidate $\psi\colon B \to B$ for a straightener witnessing $(\bV_1 + \bV_2)$-riso-triviality. It turns out that the only difficulty is to prove surjectivity of $\psi$, but this is automatic by Lemma~\ref{lem:B_finite}.)
\end{proof}

We fix some conventions for several of the next lemmas:

\begin{conv}\label{conv:trans}
We continue to assume that $K$ is a spherically complete valued field of equi-characteristic $0$.
In the following, we additionally assume that $B \subseteq K^n$ is a ball (for some $n \ge 1$), $\bW \subseteq \RF(K)^n$ is a vector subspace, $\chi\colon B \to S$ is a $\bW$-riso-trivial map, $\bar\pi_{\bW}\colon \RF(K)^n \to \bW$ is a projection and $W \subseteq K^n$ and $\pi_W\colon B\subseteq K^n \to W$ are lifts of $\bW$ and of $\bar\pi_{\bW}$, respectively.
\end{conv}
By $\pi_W$ being a lift of $\bar\pi_{\bW}$, we mean that it is the restriction to $B$ of a linear map $\pi\colon K^n \to W$ satisfying $\res(\pi(x)) = \bar\pi_{\bW}(\res(x))$ for $x \in \valring(K)^n$.

The condition that $\pi_W$ is a lift of a projection $\bar\pi_{\bW}$ (and not just any projection $B \subseteq K^n \to W$) means that the fibers of $\pi_W$ are, in some sense, ``sufficiently transversal'' to $W$. This implies various relations between $\chi$ and its restriction to such fibers, as the following lemmas show.
(Note that, while in the case $\bW = \RF(K)^n$, the lemmas become almost void, they are still true.)

\begin{lem}\label{lem:fiber_compat}
(Applying Convention~\ref{conv:trans}.) Suppose that $\bW$-riso-triviality of $\chi$ is witnessed by some $W$-straightener $\psi\colon B \to B$.
Then there exists a risometry $\phi\colon B \to B$ such that $\chi \circ \phi = \chi \circ \psi$
and such that moreover, 
$\pi_W\circ \phi = \pi_W$. In particular,
\begin{enumerate}
 \item $\phi$ is a straightener of $\chi$, and
 \item for every fiber $F \subseteq B$ of $\pi_W$,
 $\phi|_F$ is a risometry from $(\chi\circ\psi)|_F$ to $\chi|_F$.
\end{enumerate}
\end{lem}

\begin{defn}
If a risometry $\phi\colon B \to B$ satisfies $\pi \circ \phi = \pi$ for some projection $\pi\colon K^n \to W$ (as in the lemma), we say that $\phi$ \emph{respects $\pi$-fibers}.
\end{defn}

\begin{proof}[Proof of Lemma~\ref{lem:fiber_compat}]
The proof works exactly as for \cite[Lemma~3.6]{i.whit} (1), again using Lemma~\ref{lem:B_finite} instead of \cite[Lemma~2.33]{i.whit}. While the statement of \cite[Lemma~3.6]{i.whit} does not impose $\chi\circ\phi = \chi\circ \psi$, the map $\phi$ constructed in its proof has this property.
(The idea is to work with the inverses of the straighteners: One can easily modify the inverse of $\psi$ to make it respect $\pi_W$-fibers. An easy computation shows that the modified inverse is a risometry from $B$ to its image; then use Lemma~\ref{lem:B_finite} to see that its image is all of $B$.)
\end{proof}

\begin{rmk}\label{rmk:fix:fiber}
In the setting of Convention~\ref{conv:trans}, we can always find a $W$-straightener $\psi$ of $\chi$ for our given lift $W$ of $\bW$ (by Remark~\ref{rmk:lift:no:matter}), and then Lemma~\ref{lem:fiber_compat} allows us to additionally assume that $\psi$ respects $\pi_W$-fibers. We can then additionally choose a $\pi_W$-fiber $F \subseteq B$ and
a risometry $\psi'\colon F \to F$ and prescribe that $\psi|_F$ should be equal to $\psi'$. Indeed, any risometry $\phi'\colon F \to F$ extends by translation along $W$ to a risometry $\phi\colon B \to B$ (namely, sending $x+w \in B$ to $\phi'(x) + w$ for $x \in F$ and $w \in W$). The pre-composition of our $\psi$ with any such $\phi$ still straightens $\chi$, and by taking $\phi' := (\psi|_F)^{-1} \circ \psi'$, we obtain $(\psi \circ \phi)|_F = \psi|_F \circ \phi' = \psi'$.
\end{rmk}

\begin{modrmk}
If, in Remark~\ref{rmk:fix:fiber}, $\chi$ is definable, then by imposing $\psi|_F$ to be the identity on $F$, we obtain that $\chi \circ \phi$ is also definable. However, it is not always possible to ensure that $\chi \circ \phi$ is definable over the same parameters as $\chi$, and moreover, it might not be possible to take the risometry $\phi$ itself to be definable, as Example~\ref{exa:def:nnd} shows.
\end{modrmk}

\begin{rmk}\label{rmk:fiber_riso}
By applying Lemma~\ref{lem:fiber_compat} (2)
to two different fibers $F_i := \pi_W^{-1}(w_i)$
(for $w_1, w_2 \in \pi_W(B) \subseteq W$), we obtain a risometry
$\alpha\colon F_1 \to F_2$
from $\chi|_{F_1}$ to 
$\chi|_{F_2}$, namely defined by $x \mapsto \phi(\phi^{-1}(x) + w_2 - w_1)$.
This $\alpha$ has the additional property that it ``does not differ too much from the translation by $w_2-w_1$'' in the following sense: For any $x \in F_1$, we have $|\alpha(x) - (x + w_2-w_1)| < |w_2 - w_1|$. (One sees this easily by setting $x = \phi(y)$.)
\end{rmk}

The risometry type of $\chi$ is determined by the risometry type of its restriction to a $\pi_W$-fiber. More precisely, we have:

\begin{lem}\label{lem:fiber_check}
(Applying Convention~\ref{conv:trans}.)
Suppose that we have balls $B_1, B_2 \subseteq K^n$, maps $\chi_i\colon B_i \to S$ both of which are $\bW$-riso-trivial, and $\pi_W$-fibers $F_i \subseteq B_i$ for $i = 1,2$. Then the following are equivalent.
\begin{enumerate}
 \item $\chi_1$ and $\chi_2$ are in risometry.
 \item $\chi_1|_{F_1}$ and $\chi_2|_{F_2}$ are in risometry.
\end{enumerate}
\end{lem}

\begin{proof}
``(1) $\Rightarrow$ (2)'':
We may assume that $B_2 = B_1$ and that $F_2 = F_1$, namely by choosing $x_i \in F_i$ for $i=1,2$ and replacing $\chi_2$ by the map $B_1 \to S, x \mapsto \chi_2(x + x_2-x_1)$.

Pick a straightener $\psi_1$ of $\chi_1$ and a risometry $\phi\colon B_1 \to B_1$ satisfying $\chi_1 = \chi_2 \circ \phi$. Note that $\psi_2 := \phi \circ \psi_1$ is a straightener of $\chi_2$.
Using Lemma~\ref{lem:fiber_compat}, we find a risometry from $(\chi_i\circ \psi_i)|_{F_1}$ to $\chi_i|_{F_1}$, for $i=1,2$. Now (2) follows using
that $\chi_1\circ \psi_1 = \chi_2\circ \psi_2$.

``(2) $\Rightarrow$ (1)'':
Without loss, $\chi_1$ and $\chi_2$ are $W$-translation invariant. (Otherwise, compose them with a straightener, and notice that this does not change the risometry type of $\chi_i|_{F_i}$, by Lemma~\ref{lem:fiber_compat}.) Now any risometry $\phi'\colon F_1 \to F_2$ from $\chi_1|_{F_1}$ to $\chi_2|_{F_2}$ extends to a risometry from $\chi_1$ to $\chi_2$ by ``translating $\phi'$ along $W$'': send $x+w \in B_1$ to $\phi'(x)+w$ for $x \in F_1$ and $w \in W$.
\end{proof}

\begin{lem}\label{lem:fiber_triv}
(Applying Convention~\ref{conv:trans}.)
For any $\pi_W$-fiber $F \subseteq B$ and any vector subspace $\bU \subseteq \ker \bar\pi_{\bar W}$, 
$\chi|_F$ is $\bU$-riso-trivial on $F$
if and only if $\chi$ is $\bU$-riso-trivial on $B$, if and only if $\chi$ is $(\bW+ \bU)$-riso-trivial on $B$.
\end{lem}

\begin{proof}
This follows from Lemmas~\ref{lem:V-additive} and \ref{lem:fiber_compat}.
\end{proof}

The next lemma provides a simple way to compute the riso-triviality space of a set $X \subseteq K^n$ in certain simple cases.

\begin{lem}\label{lem:surface}
Let $B \subseteq K^n, \bW \subseteq \RF(K)^n, W \subseteq K^n, \pi_W\colon B \to W$ be as in
Convention~\ref{conv:trans}, and suppose that
$X \subseteq K^n$ is a subset such that $\pi_W(X \cap B) = \pi_W(B)$ and such that for every $x_1, x_2 \in X \cap B$, we have $\rv^{(n)}(x_1 - x_2) \in \rv^{(n)}(W)$.
Then $\rtsp_B(X) = \bW$.
\end{lem}

\begin{rmk}\label{rmk:rvW}
Note that the set $\rv^{(n)}(W) = \{\rv^{(n)}(w) \mid w \in W\}$ only depends on $\bW$ (and not on the chosen lift $W$). More precisely, for $x \in K^n$, we have $\rv^{(n)}(x)\in \rv^{(n)}(W)$ if and only if $\res(K\cdot x) \subseteq \bW$. (One sees this by multiplying $x$ with any $r \in K$ satisfying $|r| = |x|^{-1}$.)
\end{rmk}

\begin{proof}[Proof of Lemma~\ref{lem:surface}]
Firstly, note that each fiber $\pi_W^{-1}(w) \subseteq B$ contains exactly one element of $X$.
To see this, note that if $x_1,x_2$ lie in the same $\pi_W$-fiber, then using $\res(K\cdot(x_1-x_2)) \subseteq \res(\ker\pi_W)$, one deduces $\res(K\cdot(x_1-x_2)) \not\subseteq \bW$, contradicting the assumption that $\rv^{(n)}(x_1 - x_2) \in \rv^{(n)}(W)$.

Let $f\colon \pi_W(B) \to B$ be the map sending $w$ to the unique element of $\pi_W^{-1}(w) \cap X$.
Let us assume (without loss) that $0 \in B$. This implies that $\pi_W(B) = W \cap B$ and that
$x \mapsto x - \pi_W(x) + f(\pi_W(x))$ defines a map $\psi\colon B \to B$. The preimage $\psi^{-1}(X \cap B)$ is equal to $W \cap B$, and using that
$\rv^{(n)}(f(w_1) - f(w_2)) \in \rv^{(n)}(W)$ for any $w_1,w_2 \in \pi_W(B)$, one deduces that $\psi$ is a risometry, hence witnessing $\bW$-riso-triviality of $X$. On the other hand,
since for $w \in W \cap B$, the fiber $\pi_W^{-1}(w) \cap X$ is a singleton, it is not riso-trivial at all on $\pi_W^{-1}(w)$, so $\rtsp_B X$ cannot be strictly bigger than $\bW$ (by Lemma~\ref{lem:fiber_triv}).
\end{proof}

Next, we show that if $X \subseteq K^n$ is riso-trivial, then so is its topological closure and also any tubular neighbourhood. More precisely, we have:

\begin{lem}\label{lem:triv_cl}
Suppose that $B \subseteq K^n$ is a ball, that $X \subseteq K^n$ is a set, and that $\psi\colon B \to B$ is a $W$-straightener of $X$ on $B$, for some vector subspace $W \subseteq K^n$. Then $\psi$ is also a $W$-straightener of the following sets:
\begin{enumerate}
 \item $X' := \Cl(X)$, the topological closure of $X$ with respect to the valuation topology;
 \item $X'' := X + B_{\le\lambda}(0) = \{y \in K^n \mid \exists x \in X\colon |x-y| \le \lambda\}$, for any fixed $\lambda \in \VG(K)$.
\end{enumerate}
\end{lem}

\begin{proof}
That $\psi$ is a risometry implies that it is continuous. We therefore have
$\psi^{-1}(\Cl(X) \cap B) = \Cl(\psi^{-1}(X \cap B))$, which is $W$-translation invariant, since
$\psi^{-1}(X \cap B)$ is.
For (2), first note that we may assume that $B$ has a strictly smaller cut radius than $B_{\le\lambda}(0)$, since otherwise, $X''$ is either disjoint from $B$ or contains $B$, and the statement is trivial. With this assumption, if we replace $X$ by $X \cap B$, this also replaces $X''$ by $X'' \cap B$; we may therefore additionally assume that $X \subseteq B$. Then we use that any risometry is an isometry to deduce that
$\psi^{-1}(X + B_{\le\lambda}(0)) = \psi^{-1}(X)+ B_{\le\lambda}(0)$.
This is $W$-translation invariant since $\psi^{-1}(X)$ is.
\end{proof}

We end this section with a series of lemmas stating that riso-triviality and the risometry type of a map on arbitrary ball $B$ are determined by riso-triviality and the risometry type on the closed subballs of $B$.

\begin{lem}\label{lem:closed2all}
Suppose that a ball $B \subseteq K^n$ and a map $\chi\colon B \to S$ are given, and also a vector subspace $\bV \subseteq \RF(K)^n$. 
Then $\chi$ is $\bV$-riso-trivial on $B$ if and only if for every closed ball $B' \subseteq B$, $\chi$ is $\bV$-riso-trivial on $B'$.
\end{lem}

\begin{proof}
``$\Rightarrow$'': Choose a straightener $\phi$ of $\chi$ on $B$.
Then $B'' := \phi^{-1}(B')$ is a closed ball of the same radius as $B'$, so by pre-composing $\phi|_{B''}$ with a translation sending
$B'$ to $B''$, we obtain a straightener of $\chi$ on $B'$.

``$\Leftarrow$'': Fix any $x_0 \in B$ and let $\Omega := \{|x - x_0| \mid x \in B, x \ne x_0\} \sseteq \VG^\times(K)$ be the cut radius of $B$. Choose a well-ordered subset $\Lambda \sseteq \Omega$ which is co-final in $\Omega$, i.e., such that the suprema $\sup(\Lambda)$ and $\sup(\Omega)$ are equal.
Also fix a lift $V \subseteq K^n$ of $\bar V$ and a lift $\pi_V\colon B \to V$ of a projection $\bar\pi_{\bar V}\colon \RF(K)^n \to \bar V$. Denote the fiber of $\pi_V$ containing $x_0$ by $F$.

For each $\lambda \in \Lambda$, choose a $V$-straightener $\phi_\lambda$ of $\chi$ on $B_{\le\lambda}(x_0)$ which is the identity on $F \cap B_{\le\lambda}(x_0)$ (using Remark~\ref{rmk:fix:fiber}).
Then define $\phi: B \ra B$ by $\phi(x) := \phi_{\lambda(x)}(x)$, where $\lambda(x) \in \Lambda$ is minimal such that $x \in B_{\leq \lambda(x)}(x_0)$.
Then an easy computation shows that $\phi$ is a risometry, and to see that $\chi\circ \phi$ is $V$-translation invariant, note that the the maps $\phi_\lambda$ all agree on $F$.
\end{proof}

\begin{cor}\label{cor:closed2all}
Suppose that a ball $B \subseteq K^n$ and a map $\chi\colon B \to S$ are given.
Then there exists a closed ball $B' \subseteq B$ (possibly equal to $B$) with $\rtsp_B (\chi) = \rtsp_{B'} (\chi)$. In particular, $\chi$ is $d$-riso-trivial on $B$ (for some given $d \le n$) if and only if for every closed ball $B' \subseteq B$, $\chi$ is $d$-riso-trivial on $B'$.
\end{cor}

\begin{proof}
Set $\bar V := \rtsp_B (\chi)$ and suppose that for every closed subball $B' \subseteq B$, $\bar V' := \rtsp_{B'}(\chi)$ is strictly bigger than $\bar V$. Choose $B'$ such that $\bar V'$ has minimal dimension, and choose a complement $\bar W$ of $\bar V$ in $\bar V'$. By Lemma~\ref{lem:closed2all}, there exists a closed ball $B_{\bar W} \subseteq B$ such that $\chi$ is not $\bar W$-riso-trivial on $B_{\bar W}$ (since $\chi$ is not $\bar W$-riso-trivial). Let $B''$ be the smallest (closed) ball containing $B'$ and $B_{\bar W}$. Then  $\rtsp_{B''}(\chi)$ is a proper vector subspace of $\bar V'$, contradicting that $\bar V'$ has minimal dimension.
\end{proof}

\begin{lem}\label{lem:closed2all-riso}
Suppose that a ball $B \subseteq K^n$ and two maps $\chi_1, \chi_2\colon B \to S$ are given.
Then $\chi_1$ and $\chi_2$ are risometric if and only if there exist $x_1, x_2 \in B$ such that for every $\lambda \in \VG(K)$ satisfying $B_{\le \lambda}(x_1) \subseteq B$, we have that $\chi_1|_{B_{\le \lambda}(x_1)}$ is risometric to 
$\chi_2|_{B_{\le \lambda}(x_2)}$.
\end{lem}
\begin{proof}
``$\Rightarrow$'' is trivial. (Choose $x_1$ arbitrary and set $x_2 := \phi(x_1)$). 

For ``$\Leftarrow$'', fix such $x_1$ and $x_2$. 
Let $\Omega := \{|x - x_1| \mid x \in B, x \ne x_1\} \sseteq \VG^\times(K)$ be the cut radius of $B$ and choose a well-ordered set $\Lambda \subseteq \Omega$ with which is co-final in $\Omega$.
For each $\lambda \in \Lambda$, choose a risometry $\phi_\lambda$ from $\chi_1|_{B_{\le \lambda}(x_1)}$ to $\chi_2|_{B_{\le \lambda}(x_2)}$, and then define $\phi\colon B \to B$ by $\phi(x) := \phi_{\lambda(x)}(x)$, where $\lambda(x) \in \Lambda$ is minimal such that $x \in B_{\leq \lambda(x)}(x_1)$. An easy computation shows that $\phi$ is a risometry, and it clearly sends $\chi_1$ to $\chi_2$.
\end{proof}

\subsection{Proofs of the main definability results}
\label{sec:proofs}

This subsection contains the proofs of the main definability results, namely the Riso-Triviality Theorem~\ref{thm:RTT} and the Riso-Equivalence Theorem~\ref{thm:RET}. To this end, we now assume Hypothesis~\ref{hyp:KandX_can}, i.e., we assume that $\TT$ is a $1$-h-minimal $\LL$-theory. In addition, $K \models \TT$ will always denote a spherically complete model.

\begin{lem}
\label{lem:locallytrivdim}
If $Y \sseteq K^n$ is a definable set of dimension $r$, then for any ball $B \subseteq K^n$ which has non-empty intersection with $Y$,
we have $\dim \rtsp_B(Y)\le r$.
\end{lem}
\begin{proof}
Choose a projection $\bar \pi_{\bW} \colon \RF(K)^n \to \bW := \rtsp_B(Y)$ and a lift $\pi_W\colon B \to W$.
Since $B \cap Y$ is non-empty, at least one of the $\pi_W$-fibers
$\pi^{-1}(w) \cap Y$ (for $w \in \pi_W(B)$) is non-empty. This implies that they are all non-empty (by Remark~\ref{rmk:fiber_riso}), so $\dim Y \ge \dim (B \cap Y) \ge \dim \pi_W(B) = \dim W = \dim \rtsp_B(Y)$.
\end{proof}

For any map $\chi\colon B_0 \subseteq K^n \to \RV^\eq(K)$, the set of closed balls $B \subseteq B_0$ on which $\chi$ is not riso-trivial at all forms a tree. A crucial ingredient for the entire theory is that if $\chi$ is definable, then this tree has only finitely many branching points. We introduce a notation for the endpoints of its branches and make the statement more precise.

\begin{defn}\label{defn:Crig}
Let $\chi\colon B_0 \to S$ be a map, for some ball $B_0 \subseteq K^n$ and some $n \ge 1$, and let $T_0$ be the set of all those closed balls $B \subseteq B_0$ such that $\rtsp_B\chi = \{0\}$.
We define the \emph{rigid core} $\Crig_\chi \subseteq \pow(B_0)$ of $\chi$ as the set of those balls and points
which are obtained as intersections $\bigcap_{B \in C} B$, where $C$ is a maximal chain in $T_0$. In the case $T_0 = \emptyset$, we set $\Crig_\chi := \emptyset$.
\end{defn}

The word ``rigid'' alludes to the fact that any risometry from a definable $\chi$ to itself sends each $B \in \Crig_\chi$ to itself. (This follows from Lemmas~\ref{lem:Crig} (1) and \ref{lem:fin-riso}.)

\begin{lem}\label{lem:Crig}
Suppose that
$B_0 \subseteq K^n$ is a ball and that
$\chi\colon B_0 \to \RV^\eq(K)$ is a definable map. Then we have the following:
\begin{enumerate}
    \item $\Crig_{\chi}$ is a finite set of pairwise disjoint singletons and balls. Moreover, given an $\LL$-definable ball 
    $\bB_0 \subseteq \VF^n$ and an $\LL^\eq$-definable map $\bchi\colon \bB_0 \to \RV^\eq$, the cardinality of $\Crig_{\bchi_K}$ is bounded independently of $K$, when $K$ runs over all spherically complete models of $\TT$.
   \item For every ball $B \subseteq B_0$ with $\rtsp_B \chi = \{0\}$, there exists a $B' \in \Crig_{\chi}$ such that $B' \subseteq B$.
   \item If $\chi$ is not riso-trivial at all on a ball $B \in \Crig_\chi$, then $B$ is a closed ball and $\chi$ is $1$-riso-trivial on every proper subball $B' \subsetneq B$.
\end{enumerate}
\end{lem}

\begin{proof}

(1) Pairwise disjointness follows from fact that in the definition of $\Crig_\chi$, only maximal chains are considered. For the uniform finiteness statement, pick an $\LL$-definable 
t-stratification $(\bS_i)_i$ reflecting $\bchi$, as provided by Theorem~\ref{thm:tstrat}. We prove the following claim:
\begin{claim}\label{claim:Crig:S0}
For every spherically complete $K \models \TT$ and every $B \in \Crig_{\bchi_K}$, $B \cap \bS_0(K) \ne \emptyset$.
\end{claim}
Since the elements of $\Crig_{\bchi_K}$ are pairwise disjoint, this implies that $\#\Crig_{\bchi_K}$ is bounded by $\#\bS_0(K)$, and this
in turn is bounded uniformly for all $K$, since $\dim \bS_0 = 0$, and since $0$-dimensional sets have bounded cardinality (by Lemma~\ref{lem:Einf}).

\begin{proof}[Proof of Claim \ref{claim:Crig:S0}.]
Fix a model $K$ and let $C$ be a maximal chain in the corresponding $T_0$ (whose intersection is some given element of $\Crig_{\bchi_K}$). For each $B \in C$, $B \cap \bS_{0}(K)$ is non-empty (by definition of a t-stratification). Since $B \cap \bS_{0}(K)$ is moreover finite, we obtain that $\bigcap_{B \in C} (B \cap \bS_{0}(K))$ is non-empty.
\qedhere(Claim 1)
\end{proof}

(2) By Corollary~\ref{cor:closed2all}, there exists a closed $B'' \subseteq B$ on which $\chi$ is not riso-trivial at all. Pick a maximal chain $C \subseteq T_0$ containing $B''$ and set $B' := \bigcap_{B''' \in C} B'''$.

(3) The definition of $\Crig_\chi$ implies that $\chi$ is $1$-riso-trivial on every proper subball $B' \subsetneq B$. If $B$ would not be closed, then
in particular $\chi$ would be $1$-riso-trivial on every closed subball of $B$, so $\chi$ would also be $1$-riso-trivial on $B$ itself, by Corollary~\ref{cor:closed2all}.
\end{proof}

Once we will have proven the Riso-Triviality Theorem, we will obtain various definability results about $\Crig_\chi$. We state those conclusions already here, to keep things thematically grouped.

\begin{lem}\label{lem:Crig-def}
(Assuming that Theorem~\ref{thm:RTT} holds.)
Suppose that $\bB_0 \subseteq \VF^n$ is an $\LL$-definable ball and that
$\bchi\colon \bB_0 \to \RV^\eq$ is an $\Leq$-definable map.
Then we have the following:
\begin{enumerate}
   \item $\Crig_{\bchi_K}$ is $\Leq$-definable, uniformly for all spherically complete $K \models \TT$.
    \item There exists an $\LL$-definable $\bC \subseteq \VF^n$
    such that for every spherically complete $K \models \TT$,
    every element of $\Crig_{\bchi_K}$ contains exactly one element of $\bC(K)$.
\end{enumerate}
\end{lem}

Before we prove the lemma, we use it to define a model-independent version $\bCrig_{\bchi}$ of the rigid core:

\begin{defn}\label{defn:bCrig}
Given an $\Leq$-definable $\bchi\colon \bB_0 \to \RV^\eq$ (for some $\LL$-definable ball $\bB_0 \subseteq \VF^n$), we write $\bCrig_{\bchi}$ for the 
$\Leq$-definable set of subsets of $\bB_0$ satisfying
$\bCrig_{\bchi}(K) = \Crig_{\bchi_K}$ for every spherically complete $K \models \TT$.
\end{defn}

\begin{proof}
(1) Using the Riso-Triviality Theorem (in the form of Corollary~\ref{cor:riso-tree}), we can first define the set $T_0$ appearing in Definition~\ref{defn:Crig} and then the rigid core.

(2) Choose an $\LL$-definable t-stratification $(\bS_i)_{i \le n}$ reflecting $\bchi$. By Claim~\ref{claim:Crig:S0} in the proof of Lemma~\ref{lem:Crig}, for each $K$ and each $B \in \Crig_{\bchi_K}$, the intersection $S_B := B \cap \bS_{0}(K)$ is non-empty. Define $\bC(K)$
to be the set $\{b(S_B) \mid B \in \Crig_{\bchi_K}\}$,
where $b(S_B)$ denotes the barycenter of $S_B$.
\end{proof}

\medskip

We will now simultaneously prove Theorems \ref{thm:RTT} and \ref{thm:RET} by a common induction over the ambient dimension $n$.
More precisely, consider the following statements for $n \ge 1$, where
$K$ runs over all spherically complete models of $\TT$.

\begin{enumerate}
    \item[(r-triv)$_n$]
For every $\LL$-definable ball
$\bB \subseteq \VF^{n}$ and every $\Leq$-definable $\bchi\colon \bB \to \RV^\eq$, there exists an $\Leq$-definable vector subspace $\bbU \subseteq \RF^{n}$ such that for every $K$, we have $\rtsp_{\bB(K)}(\bchi_K) = \bbU(K)$.
\item[(r-equiv)$_n$]
For every $\LL$-definable ball
$\bB \subseteq \VF^{n}$ and every $\Leq$-definable $\bchi, \bchi'\colon \bB \to \RV^\eq$, there exists an $\LL$-sentence $\psi$ such that for every $K$, we have $K \models \psi$ if and only if $\bchi_{K}$ and $\bchi'_{K}$ are risometric.
\end{enumerate}

\begin{lem}\label{lem:t=>r}
Let $n \ge 1$ be given.
If (r-triv)$_{n}$ holds and (r-equiv)$_{n'}$ holds for $1 \le n' < n$, then (r-equiv)$_n$ holds.
\end{lem}

\begin{lem}\label{lem:r=>t}
Let $n \ge 1$ be given.
If (r-equiv)$_{n'}$ holds for $1 \le n' < n$, then (r-triv)$_n$ holds.
\end{lem}

Clearly, those two lemmas together imply Theorems~\ref{thm:RTT}
and \ref{thm:RET}.

Of course, (r-triv)$_n$ and (r-equiv)$_n$ also imply the corresponding family versions (Corollaries~\ref{cor:RTT} and \ref{cor:RET}) up to ambient dimension $n$. Moreover, (r-triv)$_n$ also implies Lemma~\ref{lem:Crig-def} in ambient dimension $n$.

\begin{proof}[Proof of Lemma~\ref{lem:t=>r}]
We are given $\Leq$-definable maps $\bchi, \bchi' \colon \bB \to \RV^\eq$
for some $\LL$-definable ball $\bB \subseteq \VF^n$; we need find an $\LL$-sentence $\psi$ expressing that a risometry from $\bchi$ to $\bchi'$ exists.

Existence of a risometry in some spherically complete model $K$ implies $\rtsp_{\bB(K)} \bchi_K =\rtsp_{\bB(K)} \bchi'_K$.
By (r-triv)$_{n}$, $\rtsp_{\bB} \bchi$ and $\rtsp_{\bB} \bchi'$ are $\Leq$-definable,
so we can put this equality as a first condition into $\psi$ and from now on
assume that $\rtsp_{\bB} \bchi = \rtsp_{\bB} \bchi' =: \bbW$. We now do a case distinction; the different cases correspond to definable conditions and formally become part of $\psi$. (Moreover, Case 3 will use Cases 1 and 2.)

\medskip

\textbf{Case 1:} $\bbW = \RF^n$:

In this case, $\bchi_K$ and $\bchi'_K$ are constant, so they are in risometry if and only if they are equal.

\medskip

\textbf{Case 2:} $\{0\} \subsetneq \bbW \subsetneq \RF^n$:

Given a spherically complete model $K \models \TT$, fix any projection $\bar\pi\colon \RF(K)^n \to \bbW(K)$ and any lift $\pi\colon K^n \to W$ of $\pi$.
Also choose any $x \in \pi(\bB(K))$ and consider the fiber
$F := \pi^{-1}(x) \cap \bB(K)$.
By (r-equiv)$_{n'}$ (for $n' = n - \dim \bbW(K)$), there exists an $\Leq$-formula which, given $\bar\pi$, $W$, $\pi$, $x$ as above, expresses whether a risometry $\bchi_K|_F \to \bchi_K'|_F$ exists. By Lemma \ref{lem:fiber_check}, this is equivalent
to the existence of a risometry $\bchi_K \to \bchi_K'$, so we obtain the desired $\LL$-sentence $\psi$ using quantifiers (universal or existential does not matter) for the above choices of $\bar\pi$, $W$, $\pi$, $x$.

\medskip

\textbf{Case 3:} $\bbW = \{0\}$:

Let us first introduce some notation in some given model $K$ (spherically complete, as always):
Set $\chi := \bchi_K$.
Recall the notion of rigid core $\Crig_{\chi} = \bCrig_{\bchi}(K)$ from Definitions~\ref{defn:Crig} and \ref{defn:bCrig}, let $P_{\chi}$ denote the set of all maximal balls disjoint from $C := \bigcup_{B \in \Crig_{\chi}} B$, and set $Q_{\chi} := \Crig_{\chi} \cup P_{\chi}$.
Note that $Q_{\chi}$ is a partition of $K^n$. Indeed, any $x \in K^n \setminus C$ has some positive distance $\lambda$ to $C$ (since $\Crig_{\chi}$ consists of finitely many points and balls), so $x \in B_{<\lambda}(x) \in P_{\chi}$.
Since $\Crig_{\chi}$ is $\Leq$-definable (uniformly in all models), so are $P_{\chi}$ and $Q_{\chi}$.

We use analogous notation for $\chi' := \bchi_K'$, namely:
$\Crig_{\chi'}$, $P_{\chi'}$ and $Q_{\chi'}$.

Any risometry from $\chi$ to $\chi'$ send balls on which $\chi$ is not riso-trivial at all to balls on which $\chi'$ is not riso-trivial at all, so it sends $\Crig_\chi$ to $\Crig_{\chi'}$.

Since the cardinalities of $\Crig_\chi$ and $\Crig_{\chi'}$ are bounded independently of the model (by Lemma~\ref{lem:Crig}),
Remark~\ref{rmk:fin-riso} yields
an $\LL$-sentence expressing that there exists a risometry from $\bB(K)$ to $\bB(K)$ sending $\Crig_\chi$ to $\Crig_{\chi'}$. We include that $\LL$-sentence in $\psi$ and from now on assume that such a risometry exists. 
Remark~\ref{rmk:fin-riso} also tells us that the induced map $\Crig_\chi \to \Crig_{\chi'}$ is $\Leq$-definable uniformly in $K$.

Next, note that each risometry $\phi$ from $\bB(K)$ to $\bB(K)$ sending $\Crig_\chi$ to $\Crig_{\chi'}$ induces a map $\alpha\colon Q_\chi \to Q_{\chi'}$, and $\alpha$ does not depend on the choice of $\phi$. Moreover,
$\alpha$ is uniformly $\Leq$-definable.
It follows that there exists a risometry from $\chi$ to $\chi'$ if and only if, for every $B \in Q_{\chi}$ and $B' := \alpha(B) \in Q_{\chi'}$, there exists a risometry $\chi|_{B} \to \chi'|_{B'}$.
We put a quantifier running over all $B \in Q_{\chi}$ into our sentence $\psi$ and from now on focus on a single such $B$, i.e., it remains to express that there exists a risometry $\chi|_{B} \to \chi'|_{B'}$.

Using (r-triv)$_n$ once more, we do a case distinction on whether 
$\chi$ is $1$-riso-trivial on $B$. If this is the case, we are done by Case~1 or 2, 
so suppose now that $\chi$ is not riso-trivial at all on $B$.
By Lemma~\ref{lem:Crig}~(2) and (3), we have $B \in \Crig_{\chi}$, it is closed, and 
and $\chi$ is $1$-riso-trivial on every proper subball of $B$. Analogous statements hold for $\chi'$ and $B' = \alpha(B)$.

Let $C_1$ be the set of open subballs $B_1 \subseteq B$ satisfying $\radop B_1 = \radcl B$, and similarly $C'_1$ for $B'$. Any risometry $B \to B'$
induces a bijection $C_1 \to C'_1$ which is a translation, in the sense that there exists a $d \in D := B' - B = \{b' - b \mid b \in B, b' \in B'\}$ such that every $B_1 \in C_1$ is sent to $B_1 + d \in C'_1$.  (This follows from the definition of risometry.) Thus there exists a risometry $\chi|_{B} \to \chi'|_{B'}$ if and only if there exists a $d \in D$ such that for every $B_1 \in C_1$, $\chi|_{B_1}$ is in risometry to $\chi'|_{B_1 + d}$.
Since $\chi|_{B_1}$ and $\chi'|_{B_1+d}$ are $1$-riso-trivial (by Lemma~\ref{lem:Crig}~(3)), this is a definable condition by Cases~1 and 2.
\end{proof}

\begin{proof}[Proof of Lemma~\ref{lem:r=>t}]
We are given an $\LL$-definable ball $\bB \subseteq \VF^n$ and an $\Leq$-definable map $\bchi \colon \bB \to \RV^{\eq}$. We need to show that $\rtsp_{\bB(K)}(\bchi_K)$ is $\Leq$-definable uniformly in $K$.

We fix an $\LL$-definable t-stratification $(\bS_{i})_i$ reflecting $\bchi$, as provided by Theorem~\ref{thm:tstrat}. Then,
by Lemma~\ref{lem:3.14}, the map sending a closed ball $B \subseteq K^n$
to $\drtsp_B((\bS_i(K))_i)$ is $\Leq$-definable (uniformly in $K$).
Using this, the following claim gives a criterion for riso-triviality on $\bB$.

\begin{figure}
 \includegraphics{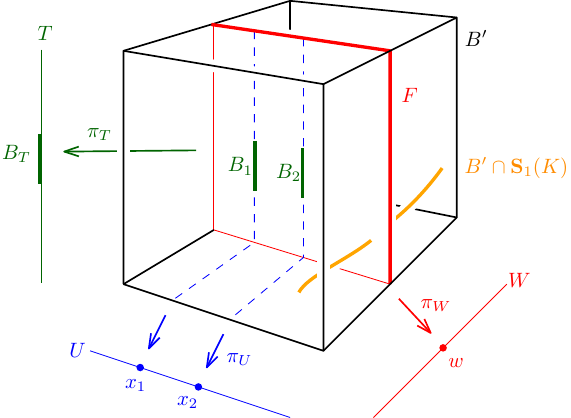}
 \caption{The ingredients of Claim~\ref{claim.ind.UVW.new} (2): $(\bS_{i}(K))_i$ is only $\bW$-riso-trivial, due do $\bS_{\dim \bW}(K)$ intersecting $\bB(K)$, but $\chi_K$ might nevertheless also be $\bU$-riso-trivial. To check this, one needs to verify whether $\bchi_K |_{B_1}$ and $\bchi_K|_{B_2}$ are in risometry for every pair of balls $B_1$ and $B_2$ as in the claim.}
 \label{fig:claim2}
\end{figure}

\begin{claim}\label{claim.ind.UVW.new}
 For every spherically complete model $K$ and every one-dimensional vector sub-space $\bU \subseteq \RF(K)^n$, the following statements are equivalent.
 \begin{enumerate}
  \item $\bchi_K$ is $\bU$-riso-trivial on $\bB(K)$.
    \item For every closed ball $B' \sseteq \bB(K)$, if $\bU$ is not contained in $\bW := \drtsp_{B'}((\bS_i(K))_i)$, then the following holds:
 
    Choose a complement $\bar T \subseteq \RF(K)^n$ of $\bU \oplus \bW$ (i.e., $\bU \oplus \bW \oplus \bar T = \RF(K)^n$), choose lifts $T, U, W \subseteq K^n$ and denote by $\pi_T: B' \twoheadrightarrow T$, $\pi_U: B' \twoheadrightarrow U$, $\pi_W: B' \twoheadrightarrow W$ the projections that are $0$ on the other two vector spaces respectively (see Figure~\ref{fig:claim2}). Moreover, choose $w \in \pi_W(B')$ and denote by $F := \pi_W^{-1}(w)$ the corresponding $\pi_W$-fiber.

    For every pair of distinct points $x_1,x_2 \in \pi_U(B')$ and every open ball $B_T \subseteq \pi_T(B')$ of open radius $|x_2-x_1|$, set $B_i := \pi_T^{-1}(B_T) \cap F \cap \pi_U^{-1}(x_i)$
     (for $i=1,2$). We require that $\bchi_K |_{B_1}$ and $\bchi_K|_{B_2}$ are in risometry.
  \end{enumerate}
\end{claim}

Some explanations concerning (2):
\begin{itemize}
 \item By ``choose'', we mean that the truth of (2) does not depend on those choices.
 \item If $\dim T = n - \dim W - 1$ happens to be $0$, then the ball $B_T$ makes no sense according to our convention that balls have to be infinite. In that case, we use the convention that $B_T$ is equal to the singleton set $T$
 (and $B_i$ becomes the singleton set $F \cap \pi^{-1}_U(x_i)$).
\end{itemize}

The claim implies the lemma: Firstly, note that the entire Condition (2) can be expressed by an $\LL$-sentence, using (r-equiv)$_{n'}$ (for $n' = n - \dim W - 1$) to definably check whether $\bchi_K |_{B_1}$ and $\bchi_K|_{B_2}$ are risometric. Thus we obtain the desired space $\rtsp_{\bB(K)}(\bchi_K)$ as the union of all $\bU$ satisfying (2).

\begin{proof}[Proof of $(1) \Rightarrow (2)$ in Claim \ref{claim.ind.UVW.new}.]
Suppose that $\bchi_K$ is $\bU$-riso-trivial on $\bB(K)$, and let $B' \sseteq \bB(K)$ be a closed sub-ball. By Remark~\ref{rmk:tstrat}, we have $\bW=\drtsp_{B'}((\bS_i(K))_i)\sseteq \rtsp_{B'}(\chi)$. Since $\chi$ is also $\bU$-riso-trivial on $B'\sseteq \bB(K)$, it is $(\bU + \bW)$-riso-trivial on $B'$ by Lemma \ref{lem:V-additive}. Now
set $F_i := F\cap \pi_U^{-1}(x_i)$ and let $B_i \subseteq F_i$ be as in (2), for $i=1,2$.
By Remark~\ref{rmk:fiber_riso}, we find a risometry
$\alpha \colon F_1 \to F_2$ sending $\bchi_K|_{F_1}$ to $\bchi_K|_{F_2}$. This risometry moreover sends $B_1$ to $B_2$, so restricts to the desired risometry from $\bchi_K|_{B_1}$ to $\bchi_K|_{B_2}$.
\qedhere (Claim~\ref{claim.ind.UVW.new}, $(1) \Rightarrow (2)$)
\end{proof}

\begin{proof}[Proof of $(1) \Leftarrow (2)$ in Claim \ref{claim.ind.UVW.new}.]
Suppose that (2) holds.
We prove the following claim by downwards induction on $\ell$, starting from $\ell = n$:
\begin{itemize}
\item[(*)$_\ell$] For every definable ball $B' \subseteq \bB(K)$ satisfying $\dim(\drtsp_{B'}((\bS_i(K))_i)) = \ell$, $\bchi_K$ is $\bU$-riso-trivial on $B'$.
\end{itemize}

Note that once we know (*)$_\ell$ for all $\ell$, we are done, since then, we can in particular take $B' = \bB(K)$.

Let us now prove (*)$_{\ell}$ for some given $\ell$, assuming that
(*)$_{\ell'}$ holds for $\ell' > \ell$.
We start by noting that it suffices to prove the claim for balls $B'$ which are closed. Indeed, if $B'$ is an arbitrary definable ball, then for every closed subball $B'' \subseteq B'$, we have $\ell'' := \dim(\drtsp_{B''}((\bS_i(K))_i)) \ge \dim(\drtsp_{B'}((\bS_i(K))_i))$, so using (*)$_{\ell''}$, we obtain $\bU$-riso-triviality on every closed subball of $B'$, which implies $\bU$-riso-triviality on $B'$ (by Lemma~\ref{lem:closed2all}).

So we now assume that $B'$ is closed; we use all the notation from (2), and we assume that $\ell= \dim W = \dim(\drtsp_{B'}((\bS_i(K))_i))$.

If $\bU \subseteq \bW$, then $\chi$ is trivially $\bU$-riso-trivial on $B'$ (by Remark~\ref{rmk:tstrat}). (This includes the the start of induction, namely $\ell = n$.) So now suppose that $\bU \not\subseteq \bW$.

Fix any $F$, a fiber of $\pi_W : B' \ra W$, as in (2). We will prove that 
\begin{itemize}
 \item[(**)] $\bchi_K|_{F}$ is $\bU$-riso-trivial.
\end{itemize}

Then, together with $\bW$-riso-triviality of $\bchi_K|_{B'}$, we obtain (using Lemma~\ref{lem:fiber_triv}) that $\bchi_K|_{B'}$ is $(\bU + \bW)$-riso-trivial. In particular this implies that $\bchi_K$ is $\bU$-riso-trivial on $B'$, as desired.

\begin{figure}
 \includegraphics{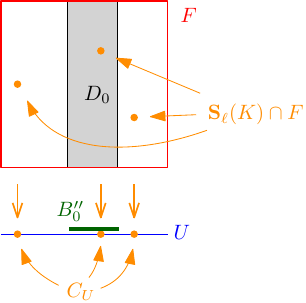}
 \caption{\label{fig:**}Proving (**): Using induction over $\kappa$, we obtain $\bar U$-riso-triviality on stripes $D_0$ containing $\kappa$ elements of $\bS_\ell(K) \cap F$. The case $\kappa = 0$ is obtained from an outer induction  over $\ell$.}
\end{figure}

Since $\ell = \dim(\drtsp_{B'}((\bS_i(K))_i))$, we have that $\bS_{\ell}(K) \cap B' \ne \emptyset$ (by Remark~\ref{rmk:tstrat}), and $\bW$-riso-triviality then also implies that $\bS_{\ell}(K) \cap F \ne \emptyset$. An easy dimension argument yields that $\bS_{\ell}(K) \cap F$ is $0$-dimensional and hence finite (see \cite[Lemma 3.10]{i.whit}).
Let $C_U := \pi_U(\bS_{\ell}(K) \cap F)$ be its projection to $U$.

Let us now assume that we are
given a definable ball $B''_0 \subseteq \pi_U(B')$ and set $D_0 := F \cap \pi_U^{-1}(B''_0) \subseteq F$ (see Figure~\ref{fig:**}).
Our goal is to prove the existence of a straightener $\phi_0\colon D_0 \to D_0$ witnessing $\bU$-riso-triviality of $\bchi_K|_{D_0}$ and satisfying $\pi_U \circ \phi_0 = \pi_U$. 
We do this using another induction, namely on the cardinality $\kappa := \#(C_U \cap B''_0)$. Once we know that claim for $\kappa = \#C_U$, we can take $B''_0 = \pi_U(B')$, yielding $\bU$-riso-triviality on the entire fiber $F$ (and hence proving (**)).

First suppose $\kappa = 0$.
Let $B'' \subseteq B'$ be a ball of the same radius as $B''_0$ which has non-empty intersection with $D_0$. Then 
$B'' \cap F \subseteq D_0$, so
$B'' \cap F \cap \bS_\ell(K) = \emptyset$ (since $\kappa = 0$).
Using $\bW$-riso-triviality, this implies
$B'' \cap \bS_\ell(K) = \emptyset$. Thus
$\dim(\drtsp_{B''}((\bS_i(K))_i)) \ge \ell+1$, meaning that we can apply induction (on $\ell$). This yields that
$\bchi_K|_{B''}$ is $\bU$-riso-trivial, and hence so is
$\bchi_K|_{B'' \cap F}$ (using $\bW$-riso-triviality once more). 
So for each $B''$ as above, we have a straightener $\phi_{B''}$ witnessing $\bU$-riso-triviality of $\bchi_K|_{B'' \cap F}$, and we may assume that $\pi_U \circ \phi_{B''} = \pi_U$. 
Since $D_0$ is a disjoint union of such balls of the form $B'' \cap F$, the union of those $\phi_{B''}$ is a straightener witnessing $\bU$-riso-triviality of $\bchi_K|_{D_0}$.

Now suppose $\kappa \ge 1$.
We fix some $c \in C_U \cap B''_0$, we partition
$B''_0 \setminus \{c\}$ into the (infinitely many) maximal balls $B''_i \subseteq B''_0$ not containing $c$, and we set $D_i := F \cap \pi_U^{-1}(B''_i)$.
By induction, we have a straightener $\phi_i\colon D_i \to D_i$ for each $i$. We ``stitch them together'' to a straightener $\phi_0\colon D_0 \to D_0$ as follows.

To ease notation, we identify $F$ with $\pi_U(B') \times \pi_T(B')$
(using the map $z \mapsto (\pi_U(z), \pi_T(z))$).

We fix an arbitrary $x_i \in B''_i$ for every $i$ and
using Claim (2), we pick a risometry
$\psi_i\colon F \cap \pi_U^{-1}(x_i) \to F \cap \pi_U^{-1}(c)$ compatible with $\bchi_K$. More precisely, for each open ball $B_T \subseteq \pi_T(B')$ of open radius $|x_i - c|$, we use Claim (2) to pick a risometry
$\{x_i\} \times B_T \to \{c\} \times B_T$, and we let $\psi_i$ be the union of all of them.

By Remark~\ref{rmk:fix:fiber}, we can (without loss) prescribe $\phi_i$ on 
$F \cap \pi_U^{-1}(x_i)$. We do so by imposing that
\begin{equation}\label{eq:psiphi}
\phi_i(x_i, y) = \psi_i^{-1}(c, y)
\end{equation}
for $y \in \pi_T(B')$.
We then define $\phi_0\colon D_0 \to D_0$ by
\[
\phi_0 := \id_{F \cap \pi_U^{-1}(c) } \cup \bigcup_i \phi_i.
\]
Then for any $i$ and any $x \in B_i''$, we have
\[
\bchi_K(\phi_0(x, y))
= \bchi_K(\phi_i(x, y))= \bchi_K(\phi_i(x_i, y)) \overset{\eqref{eq:psiphi}}{=} \bchi_K(c, y),
\]
which shows that $\bchi_K \circ \phi_0$ is
$U$-translation invariant. To see that $\phi_0$ is a risometry, first note that each pair of points of the form $z = (x_i,y), z' = (c,y')$ satisfies the risometry condition, i.e., $\rv(\phi_0(z) - \phi_0(z')) = \rv(z - z')$,
since $|\phi_0(x_i, y) - (x_i,y)| < |x_i - c|$ by \eqref{eq:psiphi} and the definition of $\psi_i$. This, together with the maps $\phi_i$ being risometries, implies the risometry condition for all pairs of points.
Indeed, consider e.g.\ a pair of points of the form $z_0 = (x,y) \in D_i, z_4 = (x',y') \in D_j$ with $i \ne j$. (Other pairs work with a similar argument.) Set $z_1 := (x_i, y)$, $z_2 := (c, y)$, $z_3 := (x_j, y)$. Then each successive pair $z_\ell, z_{\ell+1}$ satisfies the risometry condition; together with
$|z_0 - z_4| = \max_\ell|z_\ell- z_{\ell+1}|$, this implies the risometry condition for $z_0,z_4$.

Thus, $\phi_0$ is a straightener of $\phi_0$ on $D$, finishing the proof of (**) and hence also of (*)$_{\ell}$.
\qedhere (Claim~\ref{claim.ind.UVW.new}, $(1) \Leftarrow (2)$)
\end{proof}

As already explained, now that we proved
Claim~\ref{claim.ind.UVW.new}, we are also done with the proof of Lemma~\ref{lem:r=>t}.
\end{proof}

And this in turn finishes the proofs of Theorems~\ref{thm:RTT} and \ref{thm:RET}.

\section{Riso-Stratifications}
\label{sec:can_whit}

Let $K_1$ be either $\RR$ or $\CC$. Then an elementary extension
$K \succneqq K_1$ can be turned into a valued field, by taking the convex closure of $K_1$ as the valuation ring (see below for details).
In this way, a definable set $\bX(K_1) \subseteq K_1^n$ gives rise to a riso-tree, namely the one of $\bX(K)$.
The goal of this section is to construct, from that riso-tree, a
stratification in $K_1^n$, which we call the riso-stratification of $\bX(K_1)$. We will verify that the riso-stratification has some good properties, in particular that it satisfies the Whitney conditions and that $\bX(K_1)$ is a union of strata.
By construction, this riso-stratification is entirely canonical.

The construction works in quite some generality. In particular, we can allow certain expansions of the ring-structure on $K_1$, and
instead of only $\RR$ or $\CC$, we can allow real closed or algebraically closed fields of characteristic $0$. Moreover, we can also work in some other topological fields like $\QQ_p$ or $\kt$, for $k$ an arbitrary field of characteristic $0$. In those cases, the version of the Whitney conditions we obtain is closely related to the Verdier conditions considered by Cluckers--Comte--Loeser \cite{CCL.cones} and Forey \cite{For.motCones}.

Finally, using a variant of the construction in algebraically closed fields, we also associate, to an affine scheme, an ``algebraic riso-stratification'', which captures additional information when the scheme is non-reduced. That variant will be needed in Section~\ref{sec:poincare} for our application to Poincaré series.

\subsection{The assumptions}
\label{sec:whit-ass}

We start by fixing the precise conditions under which we will construct riso-stratifications.
We do the construction in three different settings simultaneously:
\begin{hyp}\label{hyp:can_whit}
In the entire Section~\ref{sec:can_whit}, we assume that $K_1$ and $\Lnoval$ satisfy one of the following assumptions and we set $\Tnoval := \Th_{\Lnoval}(K_1)$:
\begin{enumerate}
\item[(OMIN)] $K_1$ is a real closed field, considered as a structure in a language $\Lnoval$ in which it is power-bounded o-minimal (see Definition~\ref{defn:powbd}).
\item[(ACF)] $K_1$ is an algebraically closed field of characteristic $0$, considered as an $\Lnoval$-structure, where $\Lnoval$ is an extension of the ring language by an arbitrary set of constants from $K_1$.
\item[(HEN)] $K_1$ is a valued field which is elementarily equivalent to $\QQ_p$, $\kt[z]$ or $\kt[z^{\QQ}]$, where $k$ is a field of characteristic $0$, and
$\Lnoval$ is an extension of the valued field language on $K_1$ by an arbitrary set of constants from $K_1$. When applying the valued field notation from Section~\ref{sec:notn} to that valuation, we put an index ``fine'' everywhere, e.g.\ writing $|\cdot|_\fine\colon K_1 \to \VG_\fine(K_1)$ for the valuation map and $\valring_\fine(K_1)$ for the valuation ring.
\end{enumerate}
\end{hyp}

Recall that power-bounded is a generalization of polynomially bounded introduced in \cite{Mil.powBd}:
\begin{defn}\label{defn:powbd}
An o-minimal $\Lnoval$-structure $K_1$ is called power-bounded if every $\Lnoval(K_1)$-definable function from $K_1$ to itself is eventually bounded by an $\Lnoval$-definable endomorphism of the group $K_1^\times$.
\end{defn}

Some readers might mainly be interested in one of the following examples:

\begin{exa}
The following structures on $\RR$ fall into the Case (OMIN):
\begin{enumerate}
\item 
the collection of semi-algebraic sets;
\item $\RR_{\mathrm{an}}$, which consists of all globally subanalytic sets (see \cite{Dri.Ran}, where those sets are called finitely subanalytic).
\end{enumerate}
\end{exa}

The list given in Hypothesis~\ref{hyp:can_whit} (of structures where one can obtain riso-stratifications by our methods) is certainly not exhaustive. We leave it to the readers to extend it. In particular, it would be natural to expect that in the (HEN) case, we can allow any $1$-h-minimal field $K_1$ of characteristic $(0,0)$ or $(0,p)$.

Out of $\LL_0$, we construct a language $\LL$ as follows:

\begin{defn}\label{defn:LL:can_whit}
We set $\LL^- := \LL_0 \cup \{\valring\}$ where $\valring$ is a predicate for a (new) valuation ring and
$\LL := \LL^- \cup \LL_{0,\RF}$, where $\LL_{0,\RF}$ is a copy of the language $\LL_0$ on the residue field corresponding to $\valring$.
(More formally, we should either add the sort $\RF$ and the map $\res\colon \valring \to \RF$ to $\LL$, or we keep $\LL$ one-sorted and pull the ``copy of $\Lnoval$ on $\RF$'' back to $\valring$.)
\end{defn}

\begin{notn}
Given an $\LL$-structure $K$,
the valuation corresponding to $\valring$ will be denoted by $|\cdot|$, and
all the valued field notation introduced in Section~\ref{sec:notn} (like $\RF(K)$, $\valring(K)$, $\res$, \dots) will be applied with respect to that valuation. When we speak of 
``the'' valuation on $K$, we mean $|\cdot|$ (even in the case (HEN)). We consider such $K$ as topological fields, with the valuation topology coming from $|\cdot|$. Given a set $X \subseteq K^n$, we write $\Cl(X)$ for the topological closure of $X$ and $\Int(X)$ for its interior.
\end{notn}

\begin{lem}\label{lem:consKK0}
Given a field $K_1$ in a language $\Lnoval$ as above, there exist elementary extensions $K_1 \prec_{\Lnoval} K_0\prec_{\Lnoval} K$ and a valuation $|\cdot|$ on $K$ with the following properties:
\begin{enumerate}[(P1)]
    \item\label{h:sph}
    $K$ is spherically complete.
    \item\label{h:doag}
    The value group of $(K,|\cdot|)$ is non-trivial and divisible.
    \item\label{h:bij}
    $K_0$ is a lift of the residue field $\RF(K)$ of $K$, i.e., we have $K_0 \subseteq \valring(K)$, and
    the residue map $\res\colon \valring(K) \to \RF(K)$ induces a bijection from $K_0$ to $\RF(K)$.
\item\label{h:res}
We use the bijection from the $\Lnoval$-structure $K_0$ to the residue field $\RF(K)$ to put a copy of the language $\Lnoval$ onto $\RF(K)$, in this way turning $K$ into an $\LL$-structure.
When considering $K$ as an $\LL$-structure in this way, the induced structure on $\RF(K)$
is the pure $\Lnoval$-structure.
    \item\label{h:h}
    $K$ is $1$-h-minimal as an $\LL$-structure, with respect to the valuation $|\cdot|$.
\end{enumerate}
\end{lem}

\begin{proof}
The case (OMIN):

We first fix any $\Lnoval$-elementary extension $K_2 \succ K_1$ which contains infinite elements, i.e., which realizes the partial type $\{x > a \mid a \in K_1\}$.
We turn $K_2$ into a valued field (and consider it as an $\LL$-structure)
by defining the valuation ring $\valring(K_2)$ to be the convex closure of $K_1$ in $K_2$. Clearly, $\valring(K_2) \ne K_2$, i.e., the valuation is non-trivial. As an $\LL$-structure, $K_2$ is a $\TT_0$-convex structure in the sense of \cite{DL.Tcon1}.

Then we choose an $\Lnoval$-elementary substructure $K_0 \prec K_2$ which is ``tame in $K_0$'' in the sense of \cite{DL.Tcon1}, meaning that the residue map $\valring(K_2) \to \RF(K_2)$ induces a bijection $K_0 \to \RF(K_2)$.
Such tame $K_0$ exist by \cite[Theorem~(2.12)]{DL.Tcon1}. We may moreover assume $K_1 \prec_{\Lnoval} K_0$.

Finally, we use Theorem~\ref{thm:sph} to find a spherically complete immediate elementary extension $K \succ_{\LL} K_2$. Immediateness implies $\RF(K) = \RF(K_2)$, so $K_0$ is also tame in $K$.

Now (P\ref{h:sph}) holds by construction,
(P\ref{h:doag}) follows from \cite[Theorem~B]{Dri.Tcon2},
(P\ref{h:bij}) is the tameness of $K_0$ in $K$,
(P\ref{h:res}) is \cite[Theorem~A]{Dri.Tcon2}, and for (P\ref{h:h}), first note that $K$ is $1$-h-minimal as an $\LL^-$-structure
by \cite[Theorem 7.2.4]{iCR.hmin}, and then use that by \cite[Theorem 4.1.19]{iCR.hmin}, extending the language on the residue field preserves $1$-h-minimality.

\medskip

The case (ACF):

We set $K_0 := K_1$ and $K := K_0(\!(t^{\QQ})\!)$ (the Hahn field), and we let $|\cdot|$ be the $t$-adic valuation.

As an $\LL$-structure, $K$ is simply an algebraically closed valued field of equi-characteristic $0$ (in a language with some additional constants).
In particular, we have $K_1 \prec_{\Lnoval} K$ (since they are both algebraically closed fields, with some constants from $K_1$). Properties (P\ref{h:sph}), (P\ref{h:doag}) and (P\ref{h:bij}) are clear by construction of $K$ as a Hahn field.
(P\ref{h:res}) follows from Denef-Pas quantifier elimination, and for (P\ref{h:h}), note that $K$ is $1$-h-minimal by \cite[Corollary 6.2.6]{iCR.hmin}.

\medskip

The case (HEN):

As in the case (ACF), we set $K_0 := K_1$, $K := K_0(\!(t^{\QQ})\!)$, and we let $|\cdot|$ be the $t$-adic valuation.
We first check that $K_0 \prec_{\Lnoval} K$: In the case $K_0 \equiv \QQ_p$, this follows by noting that $K$ is elementarily equivalent to $K_0$ (since it it $p$-adically closed, unramified and has residue field $\FF_p$) and that $K$ has definable Skolem functions.
In the cases $K_0 \equiv \kt[z]$ and $K_0 \equiv \kt[z^{\QQ}]$, by Ax--Kochen/Ershov, it suffices to
verify that we have an elementary extension at the level of value groups of the fine valuations. Indeed, by construction, the value group $\VG_\fine(K)$ is equal to the lexicographical product $\QQ \times \VG_\fine(K_0)$ (where $\QQ$ is the more significant factor), and this is an elementary extension of $\VG_\fine(K_0)$ when $\VG_\fine(K_0) \equiv \ZZ$ or $\VG_\fine(K_0) \equiv \QQ$.\footnote{As a side remark, note that $\Gamma \prec \QQ \times \Gamma$ is not true for arbitrary ordered abelian groups $\Gamma$. A counter-example is the additive group of the polynomial ring $\ZZ[\omega]$, with the ordering suggested by the name ``$\omega$''. We leave verifying this as an exercise, using the results from \cite{iC.oag}.}

As in the (ACF) case, Properties (P\ref{h:sph}), (P\ref{h:doag}), and (P\ref{h:bij}) hold by 
construction of $K$ as a Hahn field.
To prove (P\ref{h:res}) and (P\ref{h:h}), first note that $\LL$ is already definable from the language $\LL' := \Lvf \cup \LL_{0,\RF} \subseteq \LL$, i.e., where we remove the valuation $|\cdot|_\fine$ on $\VF$ but keep it on
$\RF$. Indeed, for any $x \in K$, we have:
\[
x \in \valring_\fine(K) \quad\iff\quad 
x \in \maxid(K) \,\, \vee \,\, (x \in \valring(K) \wedge \res(x) \in \valring_\fine(K_0)).
\]
Since $K$ is $1$-h-minimal as an $\Lvf$-structure (by  \cite[Theorem 4.1.19]{iCR.hmin}), it is also $1$-h-minimal as an $\LL$-structure by \cite[Corollary 6.2.6]{iCR.hmin}, since $\LL$ is (as we just saw), up to interdefinability, obtained from $\Lvf$ by an extension on the residue field only; this shows (P\ref{h:h}).

Moreover, the same property of $\LL$ implies that we can apply resplendent Denef-Pas quantifier elimination in an extension of $\LL$ by an angular component map; this yields (P\ref{h:res}).
\end{proof}

The field $K$ obtained in Lemma~\ref{lem:consKK0} satisfies the following form of uniqueness:

\begin{lem}\label{lem:uniqKK0}
Suppose that $K_1 \equiv_{\Lnoval} K'_1$ are as in Hypothesis~\ref{hyp:can_whit} and that
we have $K_1 \prec_{\Lnoval} K_0\prec_{\Lnoval} K$  and
$K'_1 \prec_{\Lnoval} K'_0\prec_{\Lnoval} K'$ both
satisfying (P\ref{h:sph})--(P\ref{h:res}). Then $K \equiv_{\LL} K'$.
In other words, $\Th_{\LL}(K)$ is determined by $\Th_{\Lnoval}(K_1)$.
\end{lem}

\begin{proof}
In the case (OMIN), this follows from \cite[Corollary~(5.10)]{DL.Tcon1}.
(Note that the language used in that corollary is the extension of $\LL^- = \Lnoval \cup \{\valring\}$ by a predicate for $K_0$; $\LL$ can be considered as a reduct of that language, since the bijection $K_0 \to \RF(K)$ is definable in that language.)

In the cases (ACF), this follows from Ax--Kochen/Ershov, using that, by (P\ref{h:bij}), the residue field $\RF(K)$ is isomorphic to $K_0$. 
In the case (HEN), this follows from Ax--Kochen/Ershov and Denef-Pas quantifier elimination, applied in the same way as in the proof of (P\ref{h:res}) above.
\end{proof}

\begin{rmk}\label{rmk:coarse-fine}
In the case (HEN), the valuation $|\cdot|$ on $K$ is a coarsening of $|\cdot|_\fine$, as the notation suggests. (This is clear from the construction of $K$ given in the proof of Lemma~\ref{lem:consKK0}.)
\end{rmk}

Let us now fix some conventions for the rest of the section.

\begin{conv}\label{conv:sect}
For the entire remainder of Section~\ref{sec:can_whit}, we assume the following:
\begin{enumerate}
  \item $K_1$ and $\Lnoval$ are as in Hypothesis~\ref{hyp:can_whit}, i.e., in one of he cases (OMIN), (ACF), (HEN).
  \item $\LL \supseteq \Lnoval$ is as in Definition~\ref{defn:LL:can_whit}, and $K \succ_{\Lnoval} K_0 \succ_{\Lnoval} K_1$ are as obtained from Lemma~\ref{lem:consKK0} (where $K$ is an $\LL$-structure).
  \item
  The theories used implicitly in the notation from Section~\ref{sec:notn} are $\TT_0 := \Th_{\Lnoval}(K)$ and $\TT := \Th_{\LL}(K)$.
\end{enumerate}
\end{conv}

Note that by Lemma~\ref{lem:consKK0}~(P\ref{h:h}), $\LL$ and $\TT$ satisfy Hypothesis~\ref{hyp:KandX_can}, so that all the previous results of this paper apply in $K$.

The following lemma summarizes some more properties of $K$ that we will need:

\begin{lem}\label{lem:K0:props}
For $K_0$ and $K$ as constructed in Lemma~\ref{lem:consKK0}, we have the following (where $K$ is considered as an $\LL$-structure):
\begin{enumerate}[(P1)]\setcounter{enumi}{5}
\item\label{h:top}
If $X \subseteq K^n$ is $\Lnoval$-definable, then the topological closure $\Cl(X)$ (with respect to the valuation topology)
is also $\Lnoval$-definable.
\item\label{h:val} The induced structure on $\VG^\times(K)$ is o-minimal.
\item\label{h:dcl}
We have $\dclVG(K_0) = \{0, 1\}$.
\item\label{h:orth}
$\RF(K)$ and $\VG(K)$ are orthogonal (i.e., every definable subset of $\RF(K)^m \times \VG(K)^n$ is a finite union of sets of the form $X \times Y$, for some $X \subseteq \RF(K)^m$ and $Y \subseteq \VG(K)^n$).
\item\label{h:dim}
If $X \subseteq K^n$ is an $\LL$-definable set of dimension $d$ (in the sense of Definition~\ref{defn:dim}),
then there exists an $\Lnoval$-definable set $X' \subseteq K^n$ of dimension $d$ containing $X$.
\item\label{h:cl} For any $\Lnoval$-definable $X \subseteq K^n$
and any $y \in K_0^n$, if $B_{<1}(y) \cap X$ is non-empty, then $y$ lies in the topological closure of $X$ (for the valuation topology).
\end{enumerate}
\end{lem}

\begin{proof}
(P\ref{h:top}) This is clear in the case (OMIN),
since topological closure can be expressed in the o-minimal language $\Lnoval$.
In the case (ACF), this follows from
\cite[Chapter I, §10, Theorem 1]{Mum.sheaves},
which essentially states that the closure of a constructible set $X \subseteq K^n$ with respect to any ``reasonable'' topology on $K^n$ is equal to its Zariski closure; valuation topologies are reasonable in that sense.\footnote{Strictly speaking, the proof given in \cite{Mum.sheaves} only applies when every point has a countable neighbourhood basis in that ``reasonable topology''. However, it is straight forward to remove this restriction.}
In the case (HEN), note that by Remark~\ref{rmk:coarse-fine}, and since refining a valuation does not change the valuative topology, we can work with the fine valuation. Then the claim is clear, using that topological closure with respect to the fine valuation is $\Lnoval$-definable.

\medskip

(P\ref{h:val})--(P\ref{h:orth})
In the cases (ACF) and (HEN), all of these follow from Denef--Pas Quantifier Elimination (where for (P\ref{h:val}), we use that the value group is divisible).
In the case (OMIN), they follow from results of \cite{Dri.Tcon2}, namely:
(P\ref{h:val}) and (P\ref{h:dcl}) follow from \cite[Theorem B]{Dri.Tcon2}. (To get 
(P\ref{h:dcl}), apply that theorem in a language having constants for all elements of $K_0$.)
(P\ref{h:orth}) is \cite[Proposition~5.8]{Dri.Tcon2}.

\medskip

(P\ref{h:dim}) 
In the cases (ACF) and (HEN), take $X'$ to be the Zariski closure of $X$; this has the same dimension as $X$ by \cite[Theorem~3.5]{Dri.dimDef}. In the case (OMIN), this is \cite[Theorem~3.8]{Gar.powbd}. (The key idea is the following: By cell decomposition e.g.\ as in \cite[Theorem~5.2.4]{iCR.hmin}, we may assume that $X$ is the graph of an $\LL$-definable function. Such a graph is contained in the union of the graphs of finitely many $\Lnoval$-definable functions, by \cite[Corollary~2.8]{Dri.Tcon2}.)

(P\ref{h:cl})
Set $\lambda := \inf \{|x-y| \mid x \in X\}$.
This infimum exists by (P\ref{h:val}). It lies in $\dclVG(
K_0)$ and it is less than $1$ (by assumption), so by (P\ref{h:dcl}), it has to be $0$. This implies the claim.
\end{proof}

\subsection{The shadow}
\label{sec:sh}

Recall that Convention~\ref{conv:sect} is active, and recall also that the main sort
of $\Lnovaleq$ is denoted by $\bA$.
In this subsection, we introduce a first approximation to a stratification associated to an $\Lnoval$-definable set $\bX \subseteq \bA^n$, which we call the ``shadow''.\footnote{We think of it as the shadow in the residue field of the riso-tree of $\bX$.}
Even though this will be used to understand singularities of $\Lnoval$-definable objects, we will need shadows (and later also riso-stratifications) more generally associated to ($\Leq$-definable) maps $\bchi\colon \VF^n \to \RV^\eq$. Nevertheless, shadows will always be $\Lnoval$-definable.

\begin{defn}\label{defn:Sh}
Fix $n \ge 1$.
The \emph{shadow} of an $\Leq$-definable map $\bchi\colon \VF^n \to \RV^\eq$
is the partition of $\bA^n$ into $\Lnoval$-definable sets $(\bSh_{r})_{0 \le r\le n}$ given by
\begin{equation}\label{eq:Sh}
 z\in \bSh_r(\RF(K))\iff
\dim \rtsp_{\res^{-1}(z)}
(\bchi_K) = r.
\end{equation}
\end{defn}

Note that \eqref{eq:Sh} indeed yields an $\Lnoval$-definable set:
By the Riso-Triviality Theorem (in the form of Corollary~\ref{cor:RTT}), \eqref{eq:Sh} defines an $\Leq$-definable subset of $\RF^n$, and by Lemma~\ref{lem:consKK0} (P\ref{h:res}), such a set is already $\Lnoval$-definable. (Also, since $\TT_0$ is complete, $\bSh_r$ is determined by its $\RF(K)$-valued points.)
In addition, note that by Lemma~\ref{lem:uniqKK0}, the shadow does not depend on the chosen models $K_1$, $K_0$, and $K$.

The shadow of an $\Lnoval$-definable set $\bX \subseteq \bA^n$ can then be defined by first interpreting $\bX$ as a subset of $\VF^n$ and then taking $\bchi$ to be its indicator function. The following convention makes this precise and also allows tuples of sets and maps, similar to, but not exactly like Convention~\ref{conv:sets:rtsp}.

\begin{conv}\label{conv:sets:L0}
If $(\bX_1, \dots, \bX_\ell, \bchi_1, \dots, \bchi_{\ell'})$ is a tuple of finitely many $\Lnoval$-definable sets $\bX_i \subseteq \bA^n$ and $\Leq$-definable maps $\bchi_{j}\colon \VF^n \to \RV^\eq$, then its shadow is defined to be the shadow of the map $\bchi\colon \VF^n \to \RV^\eq$ defined by $\bchi_K(x) := (\mathbbm1_{\bX_{1}(K)}(x), \dots, \mathbbm1_{\bX_{\ell}(K)}(x), \bchi_{1,K}(x), \dots, \bchi_{\ell',K}(x))$.
\end{conv}

\begin{rmk}\label{rmk:Sh}
Since the $\Lnoval$-structure on $\RF(K)$ comes from the one on $K_0$ (via the bijection from Lemma~\ref{lem:consKK0} (P\ref{h:bij})),
we have a second characterization of the shadow $(\bSh_r)_r$ of a map $\bchi$, namely:
\[
 x\in \bSh_r(K_0)\iff
\dim \rtsp_{B_{<1}(x)}
(\bchi_K) = r,\]
where $x$ runs over $K_0^n$. Here, when we write $B_{<1}(x)$, we consider $x$ as an element of $K^n$ (i.e., $B_{<1}(x)$ is a ball in $K^n$).
\end{rmk}

In a stratification, one wants that the union of all skeleta up to a certain dimension is topologically closed.
The following example shows that the shadow does not have this property. (This will be fixed in the next subsection.)

\begin{exa}\label{exa:PS}
Fix $K_0 := \RR \precneqq K$, let $\bX(K) \subseteq K^3$ be the graph of the function
\[
f(x,y) = \begin{cases}
          0 & \text{if } x \le 0\\
          x\cdot y & \text{if } x \ge 0,
         \end{cases}
\]
and let $(\bSh_r)_r$ be the shadow of $\bX$.
A straightforward computation shows: $\bSh_0 = \emptyset$,
$\bSh_1 = \{0\} \times \bA^\times \times \{0\}$,
$\bSh_2 = \bX \setminus \bSh_1$ and $\bSh_3 = \bA^3 \setminus \bX$.
The key point of that computation is that $(0,0,0)$ lies in $\bSh_2(\RR)$, since $\bX(K)$ is $\bar W$-riso-trivial on $B := \res^{-1}((0,0,0))$ for $\bar W = \RR^2 \times \{0\}$, as witnessed by the straightener $\phi\colon B \to B, (x,y,z) \mapsto (x,y,z+f(x,y))$.
That $\phi$ is a risometry follows, as in Example~\ref{exa:vert:transl}, from the fact that we have $|f(x,y) - f(x',y')| < \max\{|x-x'| ,|y-y'|\}$ whenever $x,y,x',y' \in B$.

However, note that the set $\bSh_0(\RR) \cup \bSh_1(\RR) \sseteq \RR^3$ is not closed in the topology of $\RR^3$, contrary to what one would like to have.
\end{exa}

We give a third characterization of the shadow $(\bSh_{r})_r$ of an $\Leq$-definable map $\bchi\colon \VF^n \to \RV^\eq$; this will then allow us to bound the dimension of $\bSh_r$.

\begin{defn}\label{defn:loc}
Given $\bchi\colon \VF^n \to \RV^\eq$ as above,
we partition $K^n$ into sets $\Sloc_r$ (for $0 \le r \le n$) according to ``local riso-triviality'':
    \[\Sloc_r := \{x \in K^n
    \mid \dim \rtsp_B(\bchi_K) = r
    \text{ for every sufficiently small ball $B$ containing $x$}\}.\]
\end{defn}

\begin{lem}\label{lem:PSloc}
For every $r$, we have
$\bSh_{r}(K_0) = \Sloc_r \cap K_0^n$.
\end{lem}

Note that in general, we do not have $\bSh_{r}(K) = \Sloc_r$, as one sees in Example~\ref{exa:PS}, where $\Sloc_1 = \{0\} \times (K \setminus \maxid(K)) \times \{0\} \ne \bSh_{1}(K)$. (However, note that those two sets indeed agree on points in $K_0^3$.)

\begin{proof}[Proof of Lemma~\ref{lem:PSloc}]
Fix $x \in K_0^n$ and $r$. We need to prove:
$\bchi_K$
is $r$-riso-trivial on the ball $B_{<1}(x)$, if and only
if it is $r$-riso-trivial on every sufficiently small ball around $x$.

The implication ``$\Rightarrow$" is trivial. 

To prove ``$\Leftarrow$", we suppose for contradiction that we have $r$-riso-triviality on some sufficiently small ball $B$ around $x$ but no $r$-riso-triviality on $B_{<1}(x)$.
By the Riso-Triviality Theorem (in the form of Corollary~\ref{cor:RTT}), the set
\[
\Xi := \{\epsilon \in \VG^\times(K) \mid \bchi_K\text{ is not $r$-riso-trivial on }B_{\le\epsilon}(x)\}
\]
is $\Leq(x)$-definable. By (P\ref{h:val}), the set $\Xi$ has an $\Leq(x)$-definable infimum in $\VG(K)$; this infimum lies in
$\dclVG(x) \overset{(P\ref{h:dcl})}{=} \{0, 1\}$, since $x \in K_0^n$. The infimum cannot be $0$, since this would contradict $r$-riso-triviality on $B$. On the other hand,
that $\bchi_K$ is not $r$-riso-trivial on $B_{<1}(x)$ implies that the infimum cannot be $1$. Indeed, by Corollary~\ref{cor:closed2all}, there exists a closed ball $B' \subseteq B_{<1}(x)$ on which $\bchi_K$ is not $r$-riso-trivial.
This also implies non-$r$-riso-triviality on the smallest ball $B''$ containing $B'$ and $x$. Since $B''$ is closed and of radius less than $1$, this contradicts the infimum of $\Xi$ being $1$.
\end{proof}

\begin{lem}
$\dim \Sloc_r \le r$.
\end{lem}
\begin{proof}
Let $(\bS_i)_i$ be an $\LL$-definable t-stratification reflecting $\bchi$, as obtained from Theorem~\ref{thm:tstrat}.
In every neighbourhood of any $x \in \Sloc_r$, there exists a point of $\bS_{\le r}(K)$. Thus $\Sloc_r\subseteq \Cl(\bS_{\le r}(K)) = \bS_{\le r}(K)$, and this has dimension at most $r$.
\end{proof}

\begin{lem}\label{lem:shadow-dim}
$\dim \bSh_r \le r$.
\end{lem}

\begin{proof}
From $\dim \Sloc_r \le r$ and (P\ref{h:dim}), we obtain that there exists an $\Lnoval$-definable set $X' = \bX'(K) \subseteq K^n$
of dimension at most $r$ containing $\Sloc_r$.
Now $\bSh_{r}(K_0) = \Sloc_r \cap K_0^n \subseteq X' \cap K_0^n$
implies $\bSh_{r} \subseteq \bX'$ and hence
$\dim \bSh_r \le \dim \bX' \le r$.
\end{proof}

\subsection{The iterated shadow}
\label{sec:iter:sh}

We saw in Example~\ref{exa:PS} that the shadow is not, in general, a stratification. However, we can obtain a stratification in that particular example by taking the shadow of the shadow. This moves the point $(0,0,0)$ from $\bSh_2$ to $\bSh_0$, where it morally belongs. The same strategy also works in general: Taking iterated shadows stabilizes after finitely many steps, and we define the promised canonical stratification as this stabilized shadow. (We will see in the following subsections that it deserves the name stratification.)

Recall that we use Convention~\ref{conv:sect} about $\Lnoval, \LL, K_1, K_0, K, \Tnoval, \TT$. In addition, for the entire subsection, we fix an $\Leq$-definable map $\bchi\colon \VF^n \to \RV^\eq$ for some $n \ge 1$.

\begin{prop}
\label{prop:shadow:stabilizes}
In the above setting, define, recursively over $s \in \NN$, $(\bSh^s_r)_{r \le n}$ to be the shadow of the tuple $(\bchi, (\bSh^{t}_r)_{r \le n, t < s})$ (using Convention~\ref{conv:sets:L0}).
Then we have $(\bSh^{n+1}_r)_{r \le n} = (\bSh^n_r)_{r \le n}$.
\end{prop}

In the rest of this subsection, we use the notation $\bSh_r^s$ from Proposition~\ref{prop:shadow:stabilizes}, and we
apply our usual convention that $\bSh_{\le r}^s := \bigcup_{r'
\le r} \bSh_{r'}^s$ and
$\bSh_{\ge r}^s := \bigcup_{r' \ge r} \bSh_{r'}^s$ (for every $s$). We also set
$\bSh^s_{\le-1} = \emptyset = \bSh^s_{\ge n+1}$. We use non-boldface letters to denote the interpretations of all of these sets in $K$, i.e., $\Sh^s_r = \bSh^s_r(K)$, etc.

The following Lemma is used in the proof of Proposition~\ref{prop:shadow:stabilizes}.
(Recall that $\Cl$ and $\Int$ denote the closure and the interior with respect to the valuation topology.)

\begin{lem}
\label{lem:shadows}
The iterated shadows from Proposition~\ref{prop:shadow:stabilizes} have the following properties for $0 \le i \le n$ and for every $s \ge 0$:
\begin{enumerate}[(Sh1)]
    \item $\Sh_{\ge i}^{s+1} \subseteq \Sh_{\ge i}^{s}$ (or, equivalently: $\Sh_{\le i-1}^{s} \subseteq  \Sh_{\le i-1}^{s+1}$)
    \item $\Cl(\Sh_{i}^s) \setminus \Sh_{i}^s \subseteq \Sh_{\le i - 1}^{s+1}$.
    \item Suppose that for some $\Lnoval$-definable open set $U \subseteq K^n$, we have $\Sh_j^s \cap U = \Sh_j^{s+1} \cap U$ for all $j$; then also $\Sh_j^{s+1} \cap U = \Sh_j^{s+2} \cap U$ for all $j$.
    \item $\Int \Sh_{n}^0 = \Int \Sh_{n}^1$.
\end{enumerate}
\end{lem}

\begin{proof}
(Sh1)
Fix $x \in \bSh_{\le i - 1}^{s}(K_0)$. Since
$\Sh_{\le i - 1}^{s}$ has dimension at most $i - 1$ (by Lemma~\ref{lem:shadow-dim}), it is at most $(i-1)$-trivial on $B_{<1}(x)$ (by Lemma~\ref{lem:locallytrivdim}), so $x \in \bSh_{\le i - 1}^{s+1}(K_0)$ (by Remark~\ref{rmk:Sh}).

(Sh2)
Both sides are $\Lnoval$-definable (for the left hand side, this is by (P\ref{h:top})), so it suffices to verify the inclusion for $K_0$-points.
Suppose that $x \in K_0^n$ is a counter-example, i.e., that $x$ only lies in the left hand side.
From $x \notin \Sh^{s+1}_{\le i-1}$, we in particular obtain that $\Sh_{i}^s$
is $i$-riso-trivial on $B := B_{<1}(x)$.
Using Lemma~\ref{lem:triv_cl},
this implies that also $\Cl(\Sh_{i}^s) \setminus \Sh_{i}^s$ is $i$-riso-trivial on $B$. But this is a contradiction to Lemma~\ref{lem:locallytrivdim}, since $\dim (\Cl(\Sh_{i}^s) \setminus \Sh_{i}^s) < i$ (by Proposition~\ref{prop:hmin:dim}).

(Sh3)
This follows from the shadow being defined locally. More precisely, first note that it suffices to prove the equality for the $K_0$-points. For $x \in U \cap K_0^n$, we have $B := B_{<1}(x) \subseteq U$ (by (P\ref{h:cl})). Hence the assumption
implies that
\[\rtsp_B (\bchi_K, (\Sh_{j}^{s'})_{j,s'\le s}) = \rtsp_B (\bchi_K, (\Sh_{j}^{s'})_{j,s'\le s+1}),\]
which implies that for every $j$, we have $x \in \Sh_{j}^{s+1}$ if and only if $x \in \Sh_{j}^{s+2}$.

(Sh4)
The inclusion $\supseteq$ follows from (Sh1), so it suffices to prove that $\Int \Sh_{n}^0 \subseteq \Sh_{n}^1$, which then implies $\Int \Sh_{n}^0 \subseteq \Int \Sh_{n}^1$. Since both sides are $\Lnoval$-definable, it suffices to prove that elements $x \in \Int \Sh_{n}^0 \cap K_0^n$ also lie in $\Sh_{n}^1$.
For such an $x$, we have $B := B_{<1}(x) \subseteq \Int \Sh_{n}^0$ (by (P\ref{h:cl})), so in particular $(\Sh_{j}^0)_j$ is $n$-riso-trivial on $B$. Moreover, $x \in \Sh_{n}^0$ implies that also $\bchi_K$ is $n$-riso-trivial on $B$. One deduces $n$-riso-triviality of $(\bchi_K, (\Sh_{j}^0)_j)$ (since $n$-riso-triviality of a map on $B$ is equivalent to the map being constant on $B$). Thus, by definition of $\Sh_{n}^1$, we have $x \in \Sh_{n}^1$.
\end{proof}

The four observations from the lemma suffice to prove the proposition (i.e., for its proof, we do not use any other properties of shadows).

\begin{proof}[Proof of Proposition~\ref{prop:shadow:stabilizes}]
We prove the following claim by induction on $s$, for $0 \le s, j \le n$:
\begin{equation}\label{eq.ind}
\Int(\Sh_{\ge n-s}^s) \cap \Sh^s_j = \Int(\Sh_{\ge n-s}^{s+1}) \cap \Sh^{s+1}_j
\end{equation}
In the case $s = n$, this yields the claim of the proposition.

If $s = 0$, then the case $j = n$ is just (Sh4) and the cases $j < n$ are trivial; so suppose now that $s \ge 1$ and that, for all $j$, we have
\begin{equation}\label{eq.as}
\Int(\Sh_{\ge n-s+1}^{s-1}) \cap \Sh^{s-1}_j = \Int(\Sh_{\ge n-s+1}^{s}) \cap \Sh^{s}_j.
\end{equation}
We want to prove (\ref{eq.ind}).

We will prove that, for every $j$,
\begin{equation}\label{eq.a}
\Int(\Sh_{\ge n-s}^{s}) \cap \Sh^{s-1}_j = \Int(\Sh_{\ge n-s}^{s}) \cap \Sh^{s}_j.
\end{equation}

Let us check first how this implies (\ref{eq.ind}):
\begin{proof}[Proof of (\ref{eq.a}) $\Rightarrow$ (\ref{eq.ind})]
Using (Sh3) with $U = \Int(\Sh_{\ge n-s}^{s})$, equality (\ref{eq.a}) implies that
\begin{equation}\label{eq.b}
\Int(\Sh_{\ge n-s}^{s}) \cap \Sh^{s}_j = \Int(\Sh_{\ge n-s}^{s}) \cap \Sh^{s+1}_j.
\end{equation}
To deduce (\ref{eq.ind}), it thus suffices to verify that $\Int(\Sh_{\ge n-s}^{s+1}) = \Int(\Sh_{\ge n-s}^{s})$.
The inclusion ``$\subseteq$'' follows directly from (Sh1), so it remains to check ``$\supseteq$''. We have
\[
\Int(\Sh_{\ge n-s}^{s}) = \Int(\Sh_{\ge n-s}^{s}) \cap \Sh^{s}_{\ge n-s}
\overset{\substack{(\ref{eq.b})\text{ with}\\j=n-s}}{=} \Int(\Sh_{\ge n-s}^{s}) \cap \Sh^{s+1}_{\ge n-s}
\]
and hence $\Int(\Sh_{\ge n-s}^{s}) \subseteq \Sh^{s+1}_{\ge n-s}$. Since the left hand side of this inclusion is an open set, we even have
$\Int(\Sh_{\ge n-s}^{s}) \subseteq \Int(\Sh^{s+1}_{\ge n-s})$, which is what we wanted to show.
\qedhere ((\ref{eq.a}) $\Rightarrow$ (\ref{eq.ind}))
\end{proof}

It remains to prove (\ref{eq.a}), i.e.,
we need to verify that for every $x \in \Int(\Sh_{\ge n-s}^{s})$ and for every $j$, we have
\begin{equation}\label{eq.eq}
x \in \Sh^{s-1}_j \iff x \in \Sh^{s}_j.
\end{equation}
If we additionally assume $x \in \Int(\Sh_{\ge n-s+1}^{s-1})$, then the equivalence follows from (\ref{eq.as}), so we may now assume that
$x \in \Int(\Sh_{\ge n-s}^{s}) \setminus \Int(\Sh_{\ge n-s+1}^{s-1})$. We will prove that such an $x$ lies both 
\begin{enumerate}
    \item[(a)] in $\Sh_{n-s}^{s-1}$ and
    \item[(b)] in $\Sh_{n-s}^{s}$,
\end{enumerate}
so that both sides of (\ref{eq.eq}) hold if and only if $j = n-s$.

We have $x \in \Int(\Sh_{\ge n-s}^{s}) \subseteq \Sh_{\ge n-s}^{s} \overset{\text{(Sh1)}}{\subseteq} \Sh_{\ge n-s}^{s-1}$, so it suffices to prove $x \notin \Sh_{\ge n-s+1}^{s-1}$ to obtain (a); since by (Sh1) this also implies $x \notin \Sh_{\ge n-s+1}^{s}$ we then also obtain (b).

From $x \in \Int(\Sh_{\ge n-s}^{s})\overset{\text{(Sh1)}}{\subseteq} \Int (\Sh_{\ge n-s}^{s-1})$, we obtain $x \notin \Cl(\Sh_{\le n-s-1}^{s-1})$, and from
$x \notin \Int(\Sh_{\ge n-s+1}^{s-1})$, we obtain
$x \in \Cl(\Sh_{\le n-s}^{s-1})$. These two things together imply $x \in \Cl(\Sh_{n-s}^{s-1})$.
Now suppose for contradiction that $x \in \Sh_{\ge n-s+1}^{s-1}$. Then in particular $x \notin \Sh^{s-1}_{n-s}$, so we have $x \in \Cl(\Sh_{n-s}^{s-1}) \setminus \Sh^{s-1}_{n-s} \overset{\text{(Sh2)}}{\subseteq} \Sh^{s}_{\le n-s-1}$. But this contradicts the assumption that $x \in \Int(\Sh_{\ge n-s}^s)$.
\end{proof}

\subsection{The riso-stratification}
\label{sec:riso-strat}

We now define the promised canonical stratification by iteratively taking the shadow until it has stabilized; by Proposition~\ref{prop:shadow:stabilizes} this happens after at most $n$ many steps, where $n$ is the ambient dimension.

\begin{defn}\label{defn:risoStrat}
Let $K_1$ and $\Lnoval$ be given as in Hypothesis~\ref{hyp:can_whit}, and suppose that $\bchi\colon \VF^n \to \RV^\eq$ is an $\Leq$-definable map, for some $n \ge 1$. We choose $K \succ_{\Lnoval} K_0 \succ_{\Lnoval} K_1$ as in Lemma~\ref{lem:consKK0} and we define $\bSh^s_r \subseteq \bA^n$ as in Proposition~\ref{prop:shadow:stabilizes}, i.e., for $s \in \NN$,
$(\bSh^s_r)_{r \le n}$ is the shadow (Definition~\ref{defn:Sh}) of the tuple $(\bchi, (\bSh^{t}_r)_{r \le n, t < s})$.
With this notation, we call
$(\bSh^n_r)_{r \le n}$ the \emph{riso-stratification} of $\bchi$. We proceed as in Convention~\ref{conv:sets:L0} to more generally define the riso-stratification of a tuple $(\bX_1, \dots, \bX_\ell, \bchi_1, \dots, \bchi_{\ell'})$ of $\Lnoval$-definable sets $\bX_i$ and $\Leq$-definable maps $\bchi_j$.
\end{defn}

\begin{conv}\label{conv:risoStrat}
If we call a partition $(\bS_r)_{r\le n}$ of $\bA^n$ into $\Lnoval$-definable sets is a riso-stratification (without reference to any map $\bchi$), then we mean that it is the riso-stratification of some $\Leq$-definable $\bchi\colon \VF^n \to \RV^\eq$. It is not difficult to see that this is equivalent to $(\bS_r)_{r\le n}$ being the riso-stratification of itself.
\end{conv}

Note that by Lemma~\ref{lem:uniqKK0}, the riso-stratification is well-defined, i.e., it only depends on $\Th_{\Lnoval}(K_1)$ and (the formula defining) $\bchi$, but not on the choice of $K$ and $K_0$, and not even of the choice of $K_1$.

\begin{rmk}\label{rmk:riso:non-complete}
One can also make the notion of riso-stratification work uniformly for all models of a non-complete $\Lnoval$-theory $\Tnoval$, in the sense that given $\bchi$, there exist $(\bSh^n_r)_{r \le n}$ which is the riso-stratification of $\bchi$ in each completion of $\Tnoval$.
To see this, one needs to check that the definition of the shadow (Definition~\ref{defn:Sh}) works uniformly.
Recall that there, we are given an
$\Leq$-formula $\varphi$ defining a subset of $\RF^n$ and we rewrite it as an $\Lnoval$-formula $\varphi'$; we need that $\varphi'$ does not depend on the chosen model of $\Tnoval$. In other words, we need a uniform version of Lemma~\ref{lem:consKK0} (P\ref{h:res}). 
Here are some examples where our proof of (P\ref{h:res}) easily adapts to such a uniform version:
\begin{enumerate}
    \item $\Lnoval$ is the valued field language and $\Tnoval$ is the theory of henselian valued fields of equi-characteristic $0$ with value group elementarily equivalent to $\ZZ$;
    \item $\Lnoval$ is the valued field language and $\Tnoval$ is the theory of henselian valued fields of equi-characteristic $0$ with divisible value group.
\end{enumerate}
Moreover, we can always add constants symbols from $\bA$ to the language $\Lnoval$ (without changing $\Tnoval$).
\end{rmk}

We now prove some basic properties
of riso-stratifications; in Section~\ref{sec:whit}, we will then verify that they are Whitney stratifications.
We continue to fix $K$, $K_0$, $\TT$ and $\TT_0$ as in Convention~\ref{conv:sect}.

\begin{notn}
Given a riso-stratification $(\bS_i)_{i\le n}$, we set $\bS_{\le r} := \bigcup_{i \le r} \bS_{i}$, etc.;
we again use non-boldface letters for the interpretations of the above sets in $K$: $S_r = \bS_r(K)$, etc.
\end{notn}

\begin{lem}\label{lem:riso-stratificationrtsp}
Let $(\bS_i)_i$ be the riso-stratification of an $\Leq$-definable map $\bchi\colon \VF^n \to \RV^\eq$.
Given $x \in \bS_d(\RF(K))$ (for some $d$), consider the ball $B := \res^{-1}(x) \subseteq K^n$. Then we have:
\begin{enumerate}
\item $B \subseteq S_{\ge d}$ and $B \cap S_d \ne \emptyset$;
\item $\dim \rtsp_B((S_{i})_i) = d$;
\item $\rtsp_B(S_{d}) =\rtsp_B((S_{i})_i) =  \rtsp_B((S_{i})_i, \bchi_K)$.
\end{enumerate}
\end{lem}

\begin{proof}
Let $z \in K_0^n$ be the preimage of $x$ under the bijection $K_0^n \to \RF(K)^n$ induced by the residue map. Since $z \in S_d$
(by definition of the $\Lnoval$-structure on $\RF(K)$) the second part of (1) holds. Now consider the following obvious inclusions:
\[
\underbrace{\rtsp_B((S_i)_i, \bchi_K)}_{\text{(a)}} \subseteq \rtsp_B((S_i)_i) \subseteq \underbrace{\rtsp_B(S_d)}_{\text{(b)}}.
\]
We use the notation $\bSh^s_r$ from Proposition~\ref{prop:shadow:stabilizes} and set $\Sh^s_r := \bSh^s_r(K)$. Since $z \in S_d = \Sh_d^n = \Sh_d^{n+1}$ (by Definition of $S_d$), the dimension of (a) is at least $d$ by definition of $\Sh_d^{n+1}$. On the other hand, since $\dim S_d \le d$, the dimension of (b) is at most $d$ by Lemma~\ref{lem:locallytrivdim}, so we have equalities everywhere, and all those spaces have dimension $d$.
This already proves (2) and (3).
For (1), note that if $B$ would contain a point from $S_{<d}$, then
Lemma~\ref{lem:locallytrivdim} would imply $\rtsp_B((S_i)_i) < d$, contradicting (2).
\end{proof}

If $B \subseteq K^n$ is an arbitrary ball and $d$ is minimal with $B \cap S_d \ne \emptyset$, then we have $\dim \rtsp_B((S_{i})_i) \le d$ (by Lemma~\ref{lem:locallytrivdim}), but equality does not necessarily hold (in contrast to the case when $B$ is as in Lemma~\ref{lem:riso-stratificationrtsp}):

\begin{exa}\label{exa:trumpet}
Let $X = \bX(K) \subseteq K^3$ be the zero-set of $f(x,y,z) = y^2+z^2 - x^3$ (this is the trumpet from Figure~\ref{fig:trumpet} on p.~\pageref{fig:trumpet}), and let $(\bS_i)_i$ be the riso-stratification of $\bX$.
In this case, one obtains $S_0 = \{(0,0,0)\}$,
$S_1 = \emptyset$, $S_2 = X \setminus S_0$ and $S_3 = K^3 \setminus X$. Fix $\epsilon \in K$ with $|\epsilon| < 1$ and set $B := B_{<|\epsilon|}((\epsilon,0,0))$. Then $B \subseteq S_{\ge 2}$ but $\dim \rtsp_B((S_{i})_i) = 1$.
\end{exa}

We will now prove some topological properties of $(\bS_r)_r$, e.g.\ that $\bS_{\le r}$ is topologically closed. We want to make sense out of this in two different ways simultaneously: on the one hand using the valuative topology on $K$, and on the other hand, using a topology on $K_1$. In the cases (OMIN) and (HEN), $K_1$ comes with a natural topology. In the case (ACF), the Zariski topology seems to be the canonical choice, but for certain statements, we will need a finer, ``more analytic'' one. To be able to treat all cases simultaneously, we introduce a (generalized) norm $|\cdot|_\fine$ on $K_1$ which induces the desired topology in all three cases:

\begin{defn}\label{defn:L0fine}
In the case (OMIN), $|\cdot|_\fine$ is the absolute value;
in the case (HEN), $|\cdot|_\fine$ is the fine valuation already
introduced in Hypothesis~\ref{hyp:can_whit}. In the case (ACF), we fix any real closed subfield $R_1 \subseteq K_1$ such that $[K_1:R_1] = 2$
and define $|x|_\fine := \sqrt{\prod_{\sigma\in \Gal(K_1/R_1)}\sigma(x)}$.
Set $\Lnovalfine := \Lnoval \cup \{|\cdot|_\fine\}$. (In the cases (HEN), (OMIN) $\Lnovalfine$ is interdefinable with $\Lnoval$.)
\end{defn}
If $K_1 = \CC$, it is natural to take $R_1 = \RR$, so that $|\cdot|_\fine$ becomes the complex absolute value and the induced topology is just the usual analytic topology. However, while $|\cdot|_\fine$ will occasionally be used in proofs, none of the statements we prove depends on its choice. Most importantly,
whether an $\Lnoval$-definable set is topologically closed is independent of this choice, since this is equivalent to being closed in the Zariski topology (by \cite[Chapter I, §10, Theorem 1]{Mum.sheaves}).

Having introduced $|\cdot|_\fine$ on $K_1$, we also adapt our choices of $K_0$ and $K$ accordingly, as follows:

\begin{conv}\label{conv:sec:fine}
For the remainder of Section~\ref{sec:can_whit}, in addition to Convention~\ref{conv:sect}, we
assume that $K_1$ is an $\Lnovalfine$-structure as in Definition~\ref{defn:L0fine}, and we also assume that $K_0$ and $K$ are $\Lnovalfine$-structures satisfying $K \succ_{\Lnovalfine} K_0 \succ_{\Lnovalfine} K_1$.
\end{conv}

To see that such $K_0$ and $K$ exist, note that this makes a difference only in the (ACF) case, and there, the proof of Lemma~\ref{lem:consKK0} also works if one replaces $\Lnoval$ by $\Lnovalfine$.

Using this convention, when we are interested in the topological properties of $\bX(K_1)$, for some $\Lnoval$-definable set $\bX \subseteq \bA^n$, we can as well work with $\bX(K)$ instead, and there, by the following lemma, we can equivalently also use the valuative topology instead of the one induced by $|\cdot|_\fine$:
\begin{lem}\label{lem:top:agree}
On $K$, the topology induced by the valuation $|\cdot|$ and the topology induced by $|\cdot|_\fine$ agree.
\end{lem}
\begin{proof}
It suffices to compare neighbourhood bases of $0$. After rescaling, it suffices to verify that we have $\maxid(K) \subseteq \{x \in K  \mid |x|_\fine \le 1\} \subseteq \valring(K)$, which is indeed the case.
\end{proof}

The upshot of the above discussion is that the reader is pretty free in choosing their preferred topology, as long as $\Lnoval$-definable sets are concerned, namely:

\begin{defn}
We call an $\Lnoval$-definable set $\bZ \subseteq \bA^n$ topologically open (resp.\ topologically closed), if the following equivalent conditions hold:
\begin{enumerate}
    \item $\bZ(K_1)$ is open (resp.\ closed) in the topology induced by $|\cdot|_\fine$;
    \item $\bZ(K)$ is open (resp.\ closed) in the valuative topology (induced by $|\cdot|$);
    \item in the (ACF) case only: $\bZ(K_1)$ is open (resp.\ closed) in the Zariski topology.
\end{enumerate}
We call an $\LL$-definable set $\bZ \subseteq \VF^n$ topologically open (resp.\ topologically closed), if (2) holds. We use the notation $\Cl(\bZ)$ and $\Int(\bZ)$ (for closure and interior) accordingly.
\end{defn}

We are now ready to state and prove topological properties of the riso-stratification. We start with the following basic observation (which was false for the shadow):

\begin{lem}
\label{lem:riso-stratificationborder}
If $(\bS_i)_i$ is any riso-stratification in $\bA^n$, then for every $r \geq 0$, the union of strata $\bS_{\le r}$ is topologically closed.
\end{lem}

\begin{proof}
It suffices to prove that $\Cl(\bS_i(K)) \setminus \bS_i(K) \subseteq \bS_{\le i - 1}(K)$ for every $i$.
Using that the shadow has stabilized (i.e., $\bS_i = \bSh_i^n = \bSh_i^{n+1}$ for every $i$), this is exactly the statement of Lemma~\ref{lem:shadows} (Sh2).
\end{proof}

The riso-stratification is local in the following sense:

\begin{lem}\label{lem:riso:local}
Suppose that $\bchi, \bchi'\colon \VF^n \to\RV^\eq$ 
are $\Leq$-definable maps and that $\bfU \subseteq \bA^n$ is an open $\Lnoval$-definable set such that $\bchi_K|_{\bfU(K)} = \bchi'_K|_{\bfU(K)}$.
Let $(\bS_i)_i$ and $(\bS'_i)_i$ be the riso-stratifications of $\bchi$ and of  $\bchi'$, respectively.
Then we have $\bS_i \cap \bfU = \bS'_i \cap \bfU$ for every $i$.
\end{lem}

In the case where $\bchi$ and $\bchi'$ are obtained from tuples $(\bX_1, \dots \bX_k)$ and $(\bX'_1, \dots \bX'_k)$ of $\Lnoval$-definable subsets of $\bA^n$,
the assumption 
$\bchi_K|_{\bfU(K)} = \bchi'_K|_{\bX(K)}$ simply becomes
that $\bX_j \cap \bfU = \bX'_j \cap \bfU$ for every $j$, so the statement becomes independent of the valuation and can be formulated in $K_1$.

\begin{proof}[Proof of Lemma~\ref{lem:riso:local}]
Let $(\bSh_r)_r$ and $(\bSh'_r)_r$ denote the shadows of $\bchi$ and $\bchi'$, respectively. Since the riso-stratification is defined by iteratively taking the shadow, it suffices to prove that $(\bSh_r)_r$ and $(\bSh'_r)_r$ coincide in $\bfU$, i.e., that for every $x \in \bfU(K_0)$ and every $r$, we have $x \in \bSh_r(K_0)$ if and only if  $x \in \bSh'_r(K_0)$. Since $\bfU$ is open and $\Lnoval$-definable, by (P\ref{h:cl}), $\bfU(K)$ contains the entire ball $B := B_{<1}(x)$, so $\bchi_K|_B = \bchi'_K|_B$. Now the claim follows from the definition of the shadow (in the form of Remark~\ref{rmk:Sh}).
\end{proof}

\subsection{Definable manifolds}

In a Whitney stratification, the strata are required to be manifolds. In this intermediate section, we provide a characterisation of manifolds in terms of riso-triviality (Proposition~\ref{prop:riso-mani}). We work with $C^1$-manifolds. More precisely, for our manifolds to behave well in totally disconnected fields, we require that they are definable and \emph{strictly} $C^1$. (We will recall what this means.)

Recall that Convention~\ref{conv:sec:fine} is active; in particular, $K_1$ and $\Lnoval$ are as given by Hypothesis~\ref{hyp:can_whit}, and we fix $|\cdot|_\fine$ and $\Lnovalfine$ as in Definition~\ref{defn:L0fine}.
For the following definitions, we work in $K_1$ and use (the topology induced by) $|\cdot|_\fine$, but recall that as far as definable objects are concerned, it is equivalent to work in the valued field $K$ and with the valuation $|\cdot|$ (which has the advantage of not depending on the choice of $|\cdot|_\fine$).

\begin{defn}\label{defn:C1}
Let $U\subseteq K_1^n$ be open and let $f\colon U \to K_1^m$ be a function. We say that a linear map $L_u\colon K_1^n \to K_1^m$ is the \emph{strict (total) derivative} of $f$ at a point $u \in U$ if
\[
\lim_{\substack{x, x' \to u\\x\ne x'}} \frac{|f(x) - f(x')- L_u(x - x')|_\fine}{|x - x'|_\fine} = 0.
\]
If this is the case, we write $Df_u$ for $L_u$.
The function $f$ is called \emph{strictly $C^1$} if the strict derivative exists everywhere.
\end{defn}
Note that if $f$ is strictly $C^1$, then $u \mapsto Df_u$ is automatically continuous.

Using this, we introduce the corresponding notion of manifolds. We will only need submanifolds of $K_1^n$ (and not abstract manifolds). Those can be defined in terms of charts, but for our purposes, it is easier to formulate the definition in terms of graphs of functions:

\begin{defn}\label{defn:manifold}
\begin{enumerate}
\item
Fix $d \le n$.
A \emph{$d$-dimensional $\Lnoval$-definable strict-$C^1$-submanifold} of $K_1^n$, or \emph{$\Lnoval$-manifold}, for short, is an $\Lnoval$-definable subset $X \subseteq K_1^n$
such that for every $x \in X$,
there exists
an $(n-d)$-dimensional
vector subspace $W \subseteq K_1^n$ 
and 
an open neighbourhood $U \subseteq K_1^n$ of $x$
such that for each $u \in U$, $(u + W) \cap U$ contains exactly one element of $X$, and the induced map $f$ from $U/W := \{u+W \mid u \in U\} \subseteq K_1^n/W$ to $K_1^n$ is strictly $C^1$.
\item
The tangent space $T_xX$ of an $\Lnoval$-submanifold $X \subseteq K_1^n$ at a point $x \in X$ is the image
$\im Df_{x+W}$ of the total derivative of $f$ at $x+W$, for $f$ and $W$ as in (1).
\end{enumerate}
\end{defn}

We leave it to the reader to verify that $T_xX$ is well-defined, i.e., that it does not depend on the choices from (1).

\begin{rmk}\label{rmk:manifold:graph}
In (1), no definability conditions are imposed on $U$ and $f$, but one can take them to be definable in $\Lnovalfine(K_1)$, using that the topology is definable in that language.
\end{rmk}

\begin{rmk}\label{rmk:mani:indep}
The notion of being an $\Lnoval$-manifold is model independent, i.e.,
given an $\Lnoval$-definable $\bX \subseteq \bA^n$ and another model $K_1' \equiv_{\Lnoval} K_1$, $\bX(K_1')$ is an $\Lnoval$-manifold if and only if $\bX(K_1)$ is.
This is clear in the cases (OMIN) and (HEN), since $\Lnovalfine = \Lnoval$. To see that this is also true in the case (ACF), first 
pass to elementary extensions $K \succ_{\Lnovalfine} K_1$ and $K' \succ_{\Lnovalfine} K'_1$ both satisfying the conditions of Lemma~\ref{lem:consKK0}, then 
replace $|\cdot|_\fine$ by $|\cdot|$ (using that they induce the same topology), and finally note that
$K \equiv_{\LL} K'$ by Lemma~\ref{lem:uniqKK0}.\footnote{This is more or less just a (new?)\ proof of the well-known fact that for varieties, being $C^1$ can be characterized algebraically.} 
\end{rmk}

\begin{rmk}\label{rmk:C1:smooth}
In the (ACF) case, an $\Lnoval$-manifold $X$ is automatically smooth (in the sense of algebraic geometry). Indeed, we can without loss assume that the field is $\CC$ (using Remark~\ref{rmk:mani:indep}), where this is well known. (Recall that $C^1$ implies holomorphic.)
\end{rmk}

\begin{rmk}\label{rmk:change:W}
If $X$ is an $\Lnoval$-manifold, then in Definition~\ref{defn:manifold} (1), one can take for $W$ any complement of $T_xX$ in $K_1^n$ (after possibly shrinking $U$). While this can certainly be proven directly in each of the contexts from Hypothesis~\ref{hyp:can_whit} separately (using a version of the Implicit Function Theorem for definable strictly $C^1$ functions), we will see in Remark~\ref{rmk:change:W:proof} how to obtain this in all contexts simultaneously, by passing through $K$.
\end{rmk}

Intuitively, one can think of a $C^1$-manifold as a set which is roughly translation invariant on infinitesimal neighbourhood of points. The following proposition makes this intuition precise:

\begin{prop}\label{prop:riso-mani}
For a $d$-dimensional $\Lnoval$-definable set $\bX \subseteq \bA^n$, the following are equivalent; we set $X := \bX(K)$.
\begin{enumerate}
 \item $\bX$ is an $\Lnoval$-manifold.
 \item For every $x \in X \cap K_0^n$, we have $\dim \rtsp_{B_{<1}(x)}(X) = d$.
\end{enumerate}
Moreover, if this is the case, we have:
\begin{enumerate}\stepcounter{enumi}\stepcounter{enumi}
 \item For every $x \in X$ and every sufficiently small ball $B$ around $x$, we have $\rtsp_B(X) = \res(T_x X)$.
 \item If $x \in X \cap K_0^n$, then (3) holds for $B = B_{<1}(x)$; in other words, in (2), we have $\rtsp_{B_{<1}(x)}(X) = \res(T_x X) = T_x \bX(K_0)$.
\end{enumerate}
\end{prop}

\begin{proof}[Proof of Proposition~\ref{prop:riso-mani}, (1) $\Rightarrow$ (2),(4)]
We start working in $K_0$:
Choose a $|\cdot|_\fine$-neighbourhood $U_0 \subseteq K_0^n$ of $x$, an $(n-d)$-dimensional subspace $W_0 \subseteq K_0^n$ and $f_0\colon U_0/W_0 \to K_0^n$  as in Definition~\ref{defn:manifold}, witnessing that $\bX(K_0)$ is a an $\Lnoval$-manifold near $x$.
Let $F_0 := D(f_0)_{x + W_0} \in \operatorname{Hom}(K_0^n/W_0, K_0^n)$ be the differential of $f_0$ at the image of $x$ in $K_0^n/W_0$. Since $f_0$ is strictly $C^1$,
for any $\epsilon \in K_0^\times$, there exists a $\delta \in K_0^\times$ such that the following holds:
\begin{equation}\label{eq:12ma}
\forall y_1, y_2 \in U_{\delta,0}/W_0, y_1 \ne y_2\colon
\frac{|f_0(y_1) - f_0(y_2) - F_0(y_1 - y_2)|_\fine}{|y_1 - y_2|_\fine} < |\epsilon|_\fine,
\end{equation}
where $U_{0,\delta} := \{z \in K_0^n \mid |z-x|_\fine \le |\delta|_\fine\}$ (which we assume to be contained in $U_0$).

By Remark~\ref{rmk:manifold:graph} we may assume that $f_0$ is definable, so \eqref{eq:12ma} can be considered as an $\Lnovalfineeq$-formula with parameters $x,W_0,\epsilon,\delta$. Since $K$ is an elementary extension of $K_0$, we can lift everything to
$K$. We write $f$, $F$, $W$ and $U_\delta$ for the lifts of
$f_0$, $F_0$, $W_0$ and $U_{0,\delta}$, respectively.

Since the valuative ball $B := B_{<1}(x)$ is contained in 
$U_\delta$ for every $\delta \in K_0^\times$, we obtain that for $y_1, y_2 \in B/W$, the inequality of \eqref{eq:12ma} (with $f$ and $F$ instead of $f_0$ and $F_0$) holds for every $\epsilon \in K_0^\times$. This implies a corresponding valuative inequality:
\begin{equation}\label{eq:12mav}
\frac{|f(y_1) - f(y_2) - F(y_1 - y_2)|}{|y_1 - y_2|} < 1.
\end{equation}
Note that we have $\res(W) \cap \res(\im F) = W_0 \cap \im F_0 =\{0\}$.
From this, one deduces $|y_1 - y_2| = |F(y_1 - y_2)|$, so \eqref{eq:12mav} is equivalent to $\rv(f(y_1) - f(y_2)) = \rv(F(y_1 - y_2))$. In particular, for $x_1, x_2 \in X \cap B$,
we have $\rv(x_1 - x_2) \in \rv(\im F)$.
By Lemma~\ref{lem:surface}, this implies that $\rtsp_B(X) = \res(\im F)$. Since $\dim \im F = d$, this proves (1), and since $\im F = T_x X$, we also already proved (4).
\end{proof}

Before we can continue with the proof of the proposition, we need a lemma:

\begin{lem}\label{lem:fib:single}
Suppose that $\bX \subseteq \bA^n$ is a $d$-dimensional $\Lnoval$-definable set and fix a point
$x_0 \in \bX(K_0)$. Set $B := B_{<1}(x_0) \subseteq K^n$ and $\bU := \rtsp_B(\bX(K))$, and let $W \subseteq K^n$ be a vector subspace such that $\res(W)$ is a complement of $\bU$. Then for every $z \in B$, the intersection $(z + W) \cap B \cap \bX(K)$ is a singleton.
In particular, for any lift $U \subseteq K^n$ of $\bU$ and any $x, x' \in \bX(K) \cap B$, we have $\rv(x - x') \in \rv(U)$.
\end{lem}

Recall that $\rv(U)$ only depends on $\bU$, and not on the choice of the lift $U$
(Remark~\ref{rmk:rvW}).

\begin{proof}[Proof of Lemma~\ref{lem:fib:single}]
By $\bU$-riso-triviality, the statement about $F := (z+W)\cap B \cap \bX(K)$ being a singleton does not depend on the choice of $z \in B$, and neither on $W$ (provided that $\res(W)$ is a complement of $\bU$). We may thus take $z = x_0$ and let $W$ be parallel to some coordinate plane, so that in particular, $F$ is $\LL(K_0)$-definable.
By $\bU$-riso-triviality and a dimension argument, $F$ is finite, so we have $\{|x' - x_0| \mid x' \in F\} \subseteq \dclVG(K_0) = \{0, 1\}$ (by (P\ref{h:dcl})), which implies $F = \{x_0\}$ (since $F \subseteq B_{<1}(x_0)$).

The ``in particular'' part follows by choosing a $U$-straightener $\varphi\colon B \to B$ of $\bX(K)$: we then have $\rv(x - x') = \rv(\varphi^{-1}(x) - \varphi^{-1}(x')) \in \rv(U)$.
\end{proof}

\begin{proof}[Proof of Proposition~\ref{prop:riso-mani}, (2) $\Rightarrow$ (1)]
By assumption, we have a map
$\bX(\RF(K)) \to \bGr_d^n(\RF(K)), \bar x \mapsto \bV_{\bar x}:= \rtsp_{\res^{-1}(\bar x)}(X)$. This map is $\Leq$-definable by the Riso-Triviality Theorem (Corollary~\ref{cor:RTT}), and since it is a map between sorts of $\RF^\eq$, it is even $\Lnovaleq$-definable (by (P\ref{h:res})).
We fix a model independent notation for that map:
Given $K' \models \Tnoval$ and $x' \in \bX(K')$, we write $\bfV_{K',x'} \in \bGr_d^n(K')$ for the image of $x'$ under the interpretation of that map in $K'$.

Now fix $x_0 \in \bX(K_0) \subseteq X$, and set $\bar x := \res(x_0)$, $B := \res^{-1}(\bar x) = B_{<1}(x_0)$, $\bV_{\bar x} := \bfV_{\RF(K),\bar x} \subseteq \RF(K)^n$ (as in the previous paragraph) and $V_{x_0} := \bfV_{K,x_0} \subseteq K^n$.
Note that $V_{x_0}$ is a lift of $\bV_{\bar x}$
(since $\res$ sends $\bfV_{K_0,x_0} = \bfV_{K,x_0} \cap K_0^n$ to $\bV_{\bar x}$).
Additionally fix an $(n-d)$-dimensional vector subspace
$W \subseteq K^n$ such that $\res(W) \cap \bV_{\bar x} = \{0\}$, and write $\rho_W\colon K^n \to K^n/W$ for the canonical projection. By
Lemma~\ref{lem:fib:single}, we have:
\begin{equation}\label{eq:v:to:w}
\forall z_1 \in \rho_W(B)\colon
\exists^{=1} (\underbrace{x_1}_{\!\!\!\!\!\!\!\!=:f_W(z_1)\!\!\!\!\!\!\!\!} \in X \cap B)\colon \rho_W(x_1) = z_1.
\end{equation}

By $\bV_{\bar x}$-riso-triviality, for $z_1$ as in \eqref{eq:v:to:w}, we have
$|f_W(z_1) - x_0| \le |z_1 - \rho_W(x_0)|$. Thus, if we pick $\delta \in K^\times$ with $|\delta| < 1$ and set $B_{\fine,\alpha}(x_0) := \{x \in K^n \mid |x - x_0|_\fine < |\alpha|_\fine\}$ for $\alpha \in K^\times$, we have:
\begin{equation}\label{eq:ex:f}
\underbrace{\forall z_1 \in \rho_W(B_{\fine,\delta^2}(x_0))\colon
\exists^{=1} x_1 \in (X \cap B_{\fine,\delta}(x_0))\colon \rho_W(x_1) = z_1}_
{=:\phi_{\eqref{eq:ex:f}}(x_0,\delta, W)}
\end{equation}
Note that this is an $\Lnovalfineeq$-formula (interpreted in $K$) with parameters $x_0, \delta, W$ (running over $\bX, \bA^\times, \bGr^n_{n-d}$), which we denote by $\phi_{\eqref{eq:ex:f}}(x_0,\delta, W)$, as indicated.

We put this aside for a moment and
note that by the ``in particular'' part of Lemma~\ref{lem:fib:single}, we have
\begin{equation}\label{eq:dirSd}
\{\rv(x_2 - x_1) \mid x_1,x_2 \in X \cap B\} \subseteq \rv(V_{x_0}).
\end{equation}
(Recall that $V_{x_0} = \bfV_{K,x_0}$ is a lift of $\bV_{\bar x}$.)
This implies that for two distinct $z_1, z_2 \in \rho_W(B)$, we have $|(f_W(z_1) - f_W(z_2)) - L(z_1 - z_2)| < |z_1 - z_2|$,
where $f_W$ was defined in \eqref{eq:v:to:w} and $L\colon K^n/W \to V_{x_0}$ is the unique linear map satisfying $\rho_W \circ L = \id_{K^n/W}$.
And this in turn implies that for $\delta$ as before and any fixed $\epsilon \in K_0^\times$, we have
\begin{equation}\label{eq:f:C1}
\underbrace{
\forall z_1,z_2 \in \rho_W(B_{\fine,\delta^2}(x_0)), z_1 \ne z_2
\colon
|(f_W(z_1) - f_W(z_2)) - L(z_1 - z_2)|_\fine < |\epsilon\cdot(z_1 - z_2)|_\fine
}_{=:\phi_{\eqref{eq:f:C1}}(x_0,\epsilon,\delta, W)}
\end{equation}
(since $|\epsilon| = 1$). Moreover, \eqref{eq:dirSd} implies that for every fixed $x_1 \in X \cap B$,
given any $(n-d)$-dimensional vector sub-space $W' \subseteq K^n$, the condition that $\res(W')$ is a complement of $\bV_{\bar x}$ is equivalent to:
\begin{equation}\label{eq:W:comp}
\forall x_2 \in X \cap B, w\in W'\colon 
|x_2 - x_1 - w| \ge  |x_2 - x_1|.
\end{equation}
Now, for the same $x_1$ we consider the following condition on $W'$, for some fixed $\beta \in K^\times$:
\begin{equation}\label{eq:W:comp:fine}
\underbrace{\forall x_2 \in X \cap B_{\fine,\beta}(x_1), w\in W'\colon 
|x_2 - x_1 - w|_\fine \ge  |\beta\cdot(x_2 - x_1)|_\fine}_{=:\phi_{\eqref{eq:W:comp:fine}}(x_1,\beta, W')}:
\end{equation}
If we take $|\beta| < 1$, then \eqref{eq:W:comp} implies 
\eqref{eq:W:comp:fine}, whereas if we take $\beta \in K_0^\times$, then \eqref{eq:W:comp:fine} implies \eqref{eq:W:comp}.

Now we put everything together. We keep our $x_0 \in \bX(K_0)$ fixed, we also continue to fix some $\epsilon \in K_0^\times$, and we additionally fix some $\beta' \in K_0^\times$. Then on the one hand, $K$ satisfies the $\Lnovaleq(x_0)$-formula
\begin{equation}\label{eq:Wex}
\exists \beta \in K^\times, W \in \bGr_{n-d}^n(K)\colon
\forall x_1 \in X \cap B_{\fine,\beta}(x_0)\colon
\phi_{\eqref{eq:W:comp:fine}}(x_1,\beta, W);
\end{equation}
indeed, pick any $W$ such that $\res(W)$ is a complement of $\bV_{\bar x}$ (so that \eqref{eq:W:comp} holds) and pick any $\beta \in K^\times$ such that $|\beta| < 1$.
On the other hand, we can use \eqref{eq:W:comp:fine} to impose the assumption on $W$ needed by \eqref{eq:ex:f} and \eqref{eq:f:C1}, so $K$ also satisfies the $\Lnovaleq(x_0,\epsilon,\beta')$-formulas
\begin{equation}\label{eq:W:f}
\forall W \in \bGr_{n-d}^n(K)\colon
\big(\phi_{\eqref{eq:W:comp:fine}}(x_0,\beta', W)
\Rightarrow 
\forall^{\mathrm{small}} \delta \in K^\times\colon
\phi_{\eqref{eq:ex:f}}(x_0,\delta, W)\big)
\end{equation}
and
\begin{equation}\label{eq:W:C1}
\forall W \in \bGr_{n-d}^n(K)\colon
\big(\phi_{\eqref{eq:W:comp:fine}}(x_0,\beta', W)
\Rightarrow 
\exists \delta \in K^\times\colon \phi_{\eqref{eq:f:C1}}(x_0,\epsilon,\delta, W)\big),
\end{equation}
where by ``$\forall^{\mathrm{small}} \delta \in K^\times$'', we mean that there exists some $\delta_0 \in K^\times$ such that the statement holds for all $\delta \in K^\times$ satisfying $|\delta|_\fine < |\delta_0|_\fine$.

Since $K$ is an elementary extension of $K_0$, \eqref{eq:Wex}, \eqref{eq:W:f} and \eqref{eq:W:C1} also hold in $K_0$, and there, we can now also put ``$\forall x_0 \in \bX\, \forall \epsilon,\beta' \in \bA^\times$'' in front (since we proved everything for all such $x_0, \epsilon, \beta'$ in the structure $K_0$). Afterwards, we transfer those $\Lnovalfineeq$-sentences back to $K$.
For the remainder of the proof, we work in $K$.

We are now ready to prove that $X$ is a manifold. To this end, suppose that $x_0 \in X$ is given and choose
$\beta \in K^\times$ and $W \in \bGr_{n-d}^n(K)$ as provided by ``$\forall x_0\colon$\eqref{eq:Wex}'', so that
\begin{equation}\label{eq:beta:W}
\phi_{\eqref{eq:W:comp:fine}}(x_1,\beta, W) \text{ holds for all } x_1 \in X \cap B_{\fine,\beta}(x_0). 
\end{equation}
As a side remark, note that the rest of the proof works with any $\beta$ and $W$ satisfying \eqref{eq:beta:W}.

Applying \eqref{eq:beta:W} with $x_1 = x_0$ and using ``$\forall x_0,\beta'\colon$\eqref{eq:W:f}'' (with $\beta' = \beta$)
provides some $\delta \in K^\times$ such that
$\phi_{\eqref{eq:ex:f}}(x_0,\delta, W)$ holds, and we may additionally assume $\delta \le \beta$.
As before, we write $\rho_W\colon K^n \to K^n/W$ for the canonical projection. We claim that for
\[
U := B_{\fine,\delta}(x_0) \cap \rho_W^{-1}(\rho_W(B_{\fine,\delta^2}(x_0)))),
\]
$X \cap U$ is the image of a strictly $C^1$-map as required by Definition~\ref{defn:manifold}. To see this, first note that since $\phi_{\eqref{eq:ex:f}}(x_0,\delta, W)$ holds, $X \cap U$ is indeed the image of the map $f_W\colon \rho_W(U) \to X$ from \eqref{eq:v:to:w}.
It thus remains to prove that $f_W$ is strictly $C^1$ at $\rho_W(x_1)$, for every $x_1 \in X \cap U$.
Since $\phi_{\eqref{eq:W:comp:fine}}(x_1,\beta, W)$ holds (by our choices of $\beta$ and $\delta \le \beta$), ``$\forall x_0,\beta',\epsilon\colon$\eqref{eq:W:C1}'' implies
\[
\forall \epsilon \in K^\times\colon
\exists \delta' \in K^\times\colon
\phi_{\eqref{eq:f:C1}}(x_1,\epsilon,\delta', W).
\]
This is just the statement that $f_W$ is strictly $C^1$ at $x_1$, with the total derivative being the map $L$ appearing in \eqref{eq:f:C1}.
This finishes the proof that $X$ is a manifold.
\end{proof}

Before we prove the last missing part of Proposition~\ref{prop:riso-mani}, namely (3), we fulfill the promise we made in Remark~\ref{rmk:change:W} about the ``transversal'' space $W$ appearing in the definition of manifold; we do so working in $K$:

\begin{rmk}\label{rmk:change:W:proof}
Suppose that $X \subseteq K^n$ is an $\Lnoval$-manifold and that $x_0 \in X$. Our proof of Proposition~\ref{prop:riso-mani} ``(1) $\Leftrightarrow$ (2)'' additionally yields that
the condition of Definition~\ref{defn:manifold} holds for
any complement $W \subseteq K^n$ of $T_{x_0}X$.
Indeed, since $X$ satisfies Proposition~\ref{prop:riso-mani} (2), we are in the situation of the above proof of Proposition~\ref{prop:riso-mani} ``(2) $\Rightarrow$ (1)''. By the side remark right after \eqref{eq:beta:W} in that proof, it suffices to verify that $W$ satisfies \eqref{eq:beta:W} for some $\beta \in K^\times$, i.e., that for all $x_1, x_2 \in X$ sufficiently close to $x_0$ and for all $w \in W$, we have $|x_2 - x_1 - w|_\fine \ge  |\beta\cdot(x_2 - x_1)|_\fine$. This follows easily using our definition of manifold, namely that on a neighbourhood of $x_0$, $X$ is the image of a strictly $C^1$-map whose total derivative at $x_0$ has image 
$T_{x_0}X$.
\end{rmk}

\begin{proof}[Proof of Proposition~\ref{prop:riso-mani} (3)]
Choose a neighbourhood $U \subseteq K^n$ of $x$, an $(n-d)$-dimensional subspace $W \subseteq K^n$ and $f\colon U/W \to K^n$ witnessing that $X$ is a an $\Lnoval$-manifold near $x$ as in Definition~\ref{defn:manifold}; in particular, $T_xX = \im Df_{x+W}$. Since we just proved Remark~\ref{rmk:change:W}, we may
additionally assume that $\res(W) \cap \res(T_xX) = \{0\}$.
Using the definition of the strict differential $Df_{x+W}$ (but expressing the limit in Definition~\ref{defn:C1} in the valuation topology), choose $\lambda \in \VG^\times(K)$ such that $B := B_{<\lambda}(x)$ is contained in $U$ and such that for $y_1,y_2 \in B/W$, we have the valuative inequality
\[
\frac{|f(y_1) - f(y_2) - Df_{x+W}(y_1 - y_2)|}{|y_1 - y_2|} < 1.
\]
Now we are in the same situation as at the end of the proof of ``(1) $\Rightarrow$ (2), (4)'', so in the same way, we conclude that $\rtsp_B(X) = \res(T_xX)$.
\end{proof}

\subsection{Riso-stratifications are Whitney stratifications}
\label{sec:whit}

As promised, we now prove that riso-stratifications are (in particular) Whitney stratifications, except for one detail: While one usually imposes that the strata of a Whitney-stratification should be $C^\infty$-manifolds, we only obtain $C^1$-manifolds. More precisely, we use the notion of $\Lnoval$-definable (strict $C^1$) manifold from the previous subsection.

We formulate the definition of Whitney stratifications in such a way that it is suitable for model theoretic considerations; see e.g.\ \cite{BCR.realGeom} for a similar formulation of Whitney's Condition~(b). 

\begin{notn}[Grassmannians]\label{notn:grA}
Given $0 \le r \le n$, we write $\bGr_{r}^n$ for the Grassmannian of $r$-dimensional vector subspaces of $\bA^n$, considered as an $\Lnoval$-definable subset of the power set $\pow(\bA^n)$.
\end{notn}

We now introduce our variant of Whitney stratifications. Its full name maybe should be ``$\Lnoval$-definable strict $C^1$-Whitney stratification'', but as for manifolds, we use a more lightweight terminology:

\begin{defn}\label{defn:whit}
An $\Lnoval$-\emph{Whitney stratification} of $K_1^n$ (for some $n \ge 1$) is a partition of $K_1^n$ into $\Lnoval$-submanifolds $S_r$ (called \emph{skeleta}) of dimension $r$, for $0 \le r \le n$ with the following properties.
\begin{enumerate}
    \item
    For each $r \le n$, the set
    $S_{\le r} := \bigcup_{r' \le r} S_{r'}$ is closed.
    \item For every $0 \le r < s \le n$, the pair $(S_r, S_s)$ satisfies \emph{Whitney's Condition (b)}, i.e.: Consider the map
    \[f_{r,s}\colon S_r \times S_s \to \bGr^n_{1}(K_1) \times \bGr^n_{s}(K_1), (x,y) \mapsto (K_1\cdot(x-y), T_yS_s).\]
    Then for every $x_0 \in S_r$, every $V \in \bGr^n_{1}(K_1)$ and every $W \in \bGr^n_{s}(K_1)$ such that
    $(x_0, x_0, V, W)$ lies in the topological closure of the graph of $f_{r,s}$, we have $V \subseteq W$.
\end{enumerate}
\end{defn}

\begin{rmk}
Like the notion of $\Lnoval$-manifold,
the notion of $\Lnoval$-Whitney stratification does not depend on the chosen model of the $\Lnoval$-theory of $K_1$. Accordingly, we have a notion of whether a partition $(\bS_r)_{r \le n}$ of $\bA^n$ into $\Lnoval$-definable sets $\bS_r$ is an
$\Lnoval$-Whitney stratification.
\end{rmk}

Now we can state the main result of this subsection:

\begin{thm}\label{thm:whit}
We assume that $\Lnoval$ and $\Tnoval$ are as in Hypothesis~\ref{hyp:can_whit}.
Let $(\bS_r)_{r \le n}$ be the riso-stratification of an $\Leq$-definable map
$\bchi\colon \VF^n \to\RV^\eq$, for some $n \ge 1$.
Then $(\bS_r)_r$ is an $\Lnoval$-Whitney stratification.
Moreover, if $\bX \subseteq \bA^n$ is an $\Lnoval$-definable set such that $\bX(K)$ is a union of fibers of $\bchi_K$ (i.e., $\bX(K) = \bchi_K^{-1}(\bchi_K(\bX(K)))$), then for each $r$,
the intersection $\bX \cap \bS_r$ is clopen in $\bS_r$.
\end{thm}

The following remark is the motivation for the moreover-part of the theorem:

\begin{rmk}\label{rmk:riso-sets}
When one considers the riso-stratification of a tuple $(\bX_1, \dots, \bX_k)$ of $\Lnoval$-definable subsets of $\bA^n$, then in the moreover part of the theorem, one can in particular take $\bX = \bX_j$ for any $j$.
If we work over $K_1 \in \{\RR, \CC\}$, where $\Lnoval$-definable sets have finitely many connected components, the moreover part thus says that each $\bX_j(K_1)$ is a union of some of the connected components of the skeleta $\bS_r(K_1)$.
\end{rmk}

The proof of Theorem~\ref{thm:whit} is essentially the same as the proof of \cite[Theorem~7.11]{i.whit}. We nevertheless repeat it, giving some more details. (The most tedious part is already done, namely in Proposition~\ref{prop:riso-mani}.)

We continue using Convention~\ref{conv:sec:fine}; in particular, Definitions~\ref{defn:C1}, \ref{defn:manifold} and \ref{defn:whit} can equivalently be applied in $K$ (instead of $K_1$) and with the valuation topology.

The following lemma is a crucial ingredient to the proof of Whitney's Condition (b):

\begin{lem}\label{lem:A_finite}
Let $X \subseteq K^n$ be $\LL(K_0)$-definable, fix $x \in K_0^n$ and $y \in K^n$, suppose that $\lambda := |y-x|$ is neither $0$ nor $1$ and let $B := B_{<\lambda}(y)$ be the maximal ball around $y$ not containing $x$. Then $\rtsp_B(X)$ contains $\res(K\cdot(y - x))$.
\end{lem}

\begin{proof}
Let $X$ and $x$ be fixed as in the lemma, and consider the set $A$ of all those $\lambda \in \VG^\times(K)$ for which there exists a counter-example to the claim, i.e., such that there exists a $y \in K^n$ satisfying $|y - x| = \lambda$ and such that $\rtsp_{B_{<\lambda}(y)}(X)$ does not contain $\res(K\cdot (y - x))$.

By (P\ref{h:h}) and (P\ref{h:orth}), the assumptions of \cite[Theorem~7.4]{i.whit} are satisfied, which states that $A$ is finite. (More precisely, the theorem states that $A$ would be finite even if we would use $\drtsp_{B_{<\lambda}(y)}(X)$ instead of $\rtsp_{B_{<\lambda}(y)}(X)$.)
The lemma now follows from (P\ref{h:dcl}) and using that $A$ is $\Leq(K_0)$-definable:
Using the total order on $\VG(K)$, each element of $A$ lies in $\dclVG(K_0) \subseteq \{0, 1\}$.
\end{proof}

We also quickly do a small computation that we will use:
\begin{lem}\label{lem:v1v2}
If, for $v_1, v_2 \in K^n$, we have $\res (Kv_1) \ne \res(Kv_2)$, then $\res (K(v_1+v_2))$ is contained in $\res (Kv_1) + \res(Kv_2)$.
\end{lem}
\begin{proof}
If $|v_1| > |v_2|$, then $\res (K(v_1+v_2)) = \res(Kv_1)$, so the claim holds, and similarly if $|v_2| >|v_1|$. In the last case, we may assume $|v_1| = |v_2| = 1$ (otherwise dividing $v_1$ and $v_2$ by some $r \in K$ of valuation $|v_1|$). Then $\res (Kv_1) \ne \res(Kv_2)$ implies that we also have $|v_1 + v_2| = 1$ and the lemma follows using that for $v \in K^n$ satisfying $|v| = 1$, we have $\res(Kv) = \RF(K)\res(v)$.
\end{proof}

From now on, we fix $\bchi$ and $(\bS_r)_r$ as in Theorem~\ref{thm:whit}, and we set $S_r := \bS_r(K)$.
Note that by Lemma~\ref{lem:riso-stratificationrtsp}, each $\bS_r$ satisfies the assumptions of Proposition~\ref{prop:riso-mani} (2), so we already have:

\begin{lem}\label{lem:S_r-mani}
For each $r$, $S_r$ is a $r$-dimensional $\Lnoval$-manifold. Moreover, for every $x \in S_r$,
there exists a ball $B$ around $x$ such that
$\rtsp_B(S_r) = \res(T_xS_r)$.
\end{lem}

\begin{proof}
This follows from Proposition~\ref{prop:riso-mani}.
\end{proof}

\begin{lem}\label{lem:riso-stratification_b}
For every $0 \le r < s \le n$, the pair
$(S_r, S_s)$ satisfies Whitney's Condition~(b), i.e., Definition~\ref{defn:whit}~(3).
\end{lem}

Note that the Grassmannian ``$\bGr_{s}^n(K_0)$'' is somewhat ambiguous: its elements can either be considered as subspaces of $K_0^n$ or, via the natural inclusion $\bGr_{s}^n(K_0) \subseteq \bGr_{s}^n(K)$, as subspaces of $K^n$. In the following proof, it almost does not matter which interpretation one uses, though we rather think of the second one.

\begin{figure}
 \includegraphics{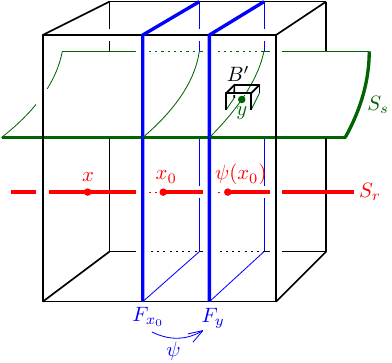}
 \caption{Proving Whitney's condition (b).}
 \label{fig:whit-b}
\end{figure}

\begin{proof}[Proof of Lemma~\ref{lem:riso-stratification_b}]
Let \[f_{r,s}\colon S_r \times S_s \to \bGr^n_{1}(K) \times \bGr^n_{s}(K), (x,y) \mapsto (K\cdot (x-y), T_yS_s)\] be as in Definition~\ref{defn:whit}, let $F$ be the topological closure of the graph of $f_{r,s}$
and recall that we need to prove that for every $x_0 \in S_r$, every $V \in \bGr^n_{1}(K)$ and every $W \in \bGr^n_{s}(K)$ such that $(x_0, x_0, V, W) \in F$, we have $V \subseteq W$. 
Since this condition is an $\Lnovalfine$-sentence, it suffices to prove it in $K_0$. We thus assume that $x_0 \in \bS_r(K_0)$, $V \in \bGr_{1}^n(K_0)$
and $W \in \bGr_{s}^n(K_0)$. Since $\res$ induces an isomorphism $K_0 \to \RF(K)$, it suffices to show that $\res(V) \subseteq \res(W)$.

Set $B := B_{<1}(x_0) \subseteq K^n$.
The set
\[
B \times B \times \{V' \in K^n \mid \res(V') = \res(V)\}
 \times \{W' \in K^n \mid \res(W') = \res(W)\}
\]is an open neighbourhood of $(x_0, x_0, V, W)$ in $K^n \times K^n \times \bGr_1^n(K) \times \bGr_s^n(K)$, so by definition of $F$, it contains an element of the graph of $f_{r,s}$.
In other words, there exist $x \in B \cap S_r$ and $y \in B \cap S_{s}$ such that 
$\res(K\cdot (x - y)) = \res(V)$
and $\res(T_yS_{s}) = \res(W)$. We fix such $x$ and $y$.
We moreover fix a ball $B' \subseteq B$ (using Lemma~\ref{lem:S_r-mani})
such that $\res(T_yS_s) = \rtsp_{B'}(S_s)$ (see Figure~\ref{fig:whit-b}).
What we have to show (the inclusion $\res V \subseteq \res W$) can now be expressed as
$\res(K\cdot (x - y)) \subseteq \rtsp_{B'}(S_s)$.

Set $\bU := \rtsp_B((S_i)_i)$, choose a projection $\bpi_{\bU}\colon \RF(K)^n \to \bU$ and choose lifts $U\subseteq K^n$ and $\pi_U\colon K^n \to U$ which are $\Lnoval(K_0)$-definable. (Such lifts can be found by first pulling back $\bU$ and $\bpi_{\bU}$ to $K_0^n$ using $\res|_{K_0^n}\colon K_0^n \to \RF(K)^n$, and then let $U$ and $\pi_U$ be defined by the same formulas as those pull-backs.)

Given any $z \in B$, we denote by $F_z := \{z' \in B \mid \pi_U(z') = \pi_U(z)\}$ the corresponding $\pi_U$-fiber in $B$.
By $\bU$-riso-triviality of $(S_{i})_i$,
there exists a risometry $\psi\colon F_{x_0} \to F_{y}$ sending $S_i \cap F_{x_0}$ to $S_i \cap F_{y}$ for every $i$ (by Remark~\ref{rmk:fiber_riso}). We write $x-y$ as the sum of $v_1 := x-\psi(x_0)$ and $v_2 := \psi(x_0) - y$. 
Then we have $\res(K\cdot v_1) \subseteq \bU$
by Lemma~\ref{lem:fib:single} and $v_2 \in \ker \pi_U$.
In particular, $\res(K\cdot v_1) \ne \res(K\cdot v_2)$, which implies (by Lemma~\ref{lem:v1v2}) that 
to obtain $\res(K\cdot (v_1 + v_2)) \subseteq \rtsp_{B'}(S_s)$, it suffices
to prove $\res(K\cdot v_j) \subseteq \rtsp_{B'}(S_s)$ for $j=1,2$. For $j = 1$, this is clear:
$\res(K\cdot v_1) \subseteq \bU \subseteq \rtsp_{B'}(S_s)$. So it remains to prove the claim for $v_2$.

Set $y_0 := \psi^{-1}(y) \in F_{x_0}$.
Since $\psi$ is a risometry, we have $\rv(v_2) = \rv(x_0 - y_0)$ and $\rtsp_{F_y \cap B'}(S_s) = \rtsp_{B_0}(S_s)$, where $B_0 \subseteq F_{x_0}$ is the ball around $y_0$ of the same radius as $B'$.
To obtain $\res(K\cdot v_2) \subseteq \rtsp_{B'}(S_s)$, it thus suffices to prove
$\res(K\cdot (x_0 - y_0)) \subseteq \rtsp_{B_0}(S_s)$.
This follows by applying Lemma~\ref{lem:A_finite} to
$S_s \cap F_{x_0}$, $x_0$ and $y_0$ (working in $F_{x_0}$ as the ambient space). Indeed, note that this $S_s \cap F_{x_0}$ is $\LL(K_0)$-definable and that $|y_0-x_0| < 1$, so by the lemma, we have $\res(K\cdot (x_0 - y_0)) \subseteq \rtsp_{B'_0}(S_s)$ for some ball $B'_0 \subseteq F_{x_0}$ containing $B_0$.
\end{proof}

We now are already almost done with the proof of Theorem~\ref{thm:whit}.

\begin{proof}[Proof of Theorem~\ref{thm:whit}]
By Lemma~\ref{lem:S_r-mani}, $S_r$ is an $r$-dimensional manifold.
By Lemma~\ref{lem:riso-stratificationborder}, $S_{\le r}$ is closed, and by Lemma~\ref{lem:riso-stratification_b}, Whitney's Condition (b) is satisfied (i.e., Definition~\ref{defn:whit} (2)). It remains to verify the moreover part of the theorem, namely: if $X \subseteq K^n$ is an $\Lnoval$-definable set which is a union of fibers of $\bchi_K$, then
$X \cap S_r$ is clopen in $S_r$ for every $r$.
We show that $X \cap S_r$ is open in $S_r$; closedness is then obtained by applying the same argument to the complement of $X$.

We need to show that no $x \in S_r \cap X$ lies in the topological closure of $S_r \setminus X$.
Since $K_0 \prec_{\Lnoval} K$, it suffices to verify this for $x \in S_r \cap X \cap K_0^n$, so fix such an $x$
and set $B := B_{<1}(x)$. 
By Lemma~\ref{lem:riso-stratificationrtsp},
$(S_r, \bchi_K)$ is $r$-riso-trivial on $B$ and hence so is $(S_r, X)$ (since $X$ is a union of fibers of $\bchi_K$).
Let $U \subseteq K^n$ be a lift of $\bU := \rtsp_B(S_r, X)$, let $\pi_U\colon B \to U$ be a lift of a projection $\RF(K)^n \to \bU$, and let $F \subseteq B$ be the $\pi_U$-fiber containing $x$. Since $F \cap S_r$ is a singleton by Lemma~\ref{lem:fib:single}, it is equal to $\{x\}$; in particular, $F \cap S_r$ is a subset of $F \cap X$. Using $\bU$-riso-triviality of $(S_r, X)$, we deduce that
that $B \cap S_r \subseteq B \cap X$. This shows that $B$ is an open neighbourhood of $x$ disjoint from $S_r \setminus X$.
\end{proof}

\subsection{Riso-stratifications of affine schemes}
\label{sec:alg:riso}

In this subsection, we work in the (ACF) case of Hypothesis~\ref{hyp:can_whit}, where $\Lnoval$ is the ring language extended by constants for some ring $R \subseteq K_1$ (and $\Tnoval =\Th_{\Lnoval}(K_1)$).
We want to associate a riso-stratification to a closed sub-scheme $\bX$ of the affine space $\mathbb{A}^n_R$. We could simply take the riso-stratification of the corresponding $\Lnoval$-definable subset of $\bA^n$, but it seems that this would not be good enough for our application to Poincaré series (in Section~\ref{sec:poincare}); in particular, it would only capture the reduced structure $\bX^{\mathrm{red}}$ of $\bX$. If $\bX$ is a hypersurface defined by some polynomial $f \in R[x_1, \dots, x_n]$, then the problem can be fixed by taking the riso-stratification of the $\Leq$-definable map
$\bchi\colon \VF^n \to \RV, a \mapsto \rv(f(a))$.
We will do something similar for general $\bX$, using generators of the ideal defining $\bX$. To this end, we start by verifying that the result does not depend on the choice of generators. As usual, $K$, $\LL$ and $\TT$ are as in Convention~\ref{conv:sect}.

\begin{lem}\label{lem:ideal:riso}
Suppose that $I \subseteq R[x_1, \dots, x_n]$ is an ideal and that $(f_1, \dots, f_\ell)$ and $(f'_1, \dots, f'_{\ell'})$ are two sets of generators of $I$.
Define
\[
\bchi\colon \VF^n \to \RV^{(\ell)},
a \mapsto \rv^{(\ell)}(f_1(a), \dots, f_\ell(a)))
\]
and
\[
\bchi'\colon \VF^n \to \RV^{(\ell')},
a \mapsto \rv^{(\ell')}(f'_1(a), \dots, f'_{\ell'}(a))).
\]
Then, for $a, b \in \valring(K)^n$ satisfying $|a-b| < 1$, we have
\begin{equation}\label{eq:ideal:chi}
\bchi_K(a) = \bchi_K(b) \iff \bchi'_K(a) = \bchi'_K(b).
\end{equation}
In particular:
\begin{enumerate}
 \item For any ball $B \subsetneq \valring(K)^n$, we have $\rtsp_B(\bchi_K) = \rtsp_B(\bchi'_K)$.
 \item The riso-stratifications of $\bchi$ and of $\bchi'$ agree.
\end{enumerate}
\end{lem}

\begin{proof}
Fix any $g \in I$ and write it as
$g = \sum_i h_if_i$, for $h_i \in R[x_1, \dots, x_n]$.
Then for $a \in \valring(K)^n$, we have $h_i(a) \in \valring(K)$ (since
$R \subseteq \valring(K)$), so we obtain $|g(a)| \le \max_i |f_i(a)|$. Applying this to $g := f'_j$ for $j=1,\dots,\ell'$, and also applying the same argument with $(f_i)_i$ and $(f'_j)_j$ swapped, we deduce $\max_i |f_i(a)| = \max_j |f'_j(a)| =: \lambda_a$ for every $a \in \valring(K)^n$.

Now consider again $g$ and $a$ as before and additionally fix some $b \in \valring(K)^n$ satisfying $|a - b| < 1$ and $\bchi_K(b) = \bchi_K(a)$.
Then we have $|h_i(a)-h_i(b)| \le |a - b| < 1$, since the constant terms of $h_i(a)$ and $h_i(b)$ cancel, and we obtain
\begin{align*}
|h_i(b)f_i(b) - h_i(a)f_i(a)|
&=
|h_i(b)(f_i(b) -f_i(a)) + (h_i(b)-h_i(a))f_i(a)|\\
&\le \max\{
\underbrace{|h_i(b)|}_{\le 1}\cdot
\underbrace{|f_i(b) -f_i(a)|}_{< \lambda_a}, \underbrace{|h_i(b)-h_i(a)|}_{< 1}\cdot
\underbrace{|f_i(a)|}_{\le \lambda_a}\} < \lambda_a.
\end{align*}
By summing this over $i$, we obtain $|g(b) - g(a)|< \lambda_a$. Applying this to $g := f'_j$ for every $j$, and using $\max_j |f'_j(a)| = \lambda_a$, we deduce $\bchi'_K(a) = \bchi'_K(b)$. The same argument with $(f_i)_i$ and $(f'_j)_j$ swapped yields the other direction of \eqref{eq:ideal:chi}.

Concerning the ``in particular'' parts:

(1) By \eqref{eq:ideal:chi}, a map $B \to B$ straightens $\bchi_K$ if and only if it straightens $\bchi'_K$.

(2) Since the shadow is defined in terms of riso-triviality on balls $B$ of the size of the maximal ideal and contained in $\valring(K)^n$, (1) implies that the shadows of $\bchi$ and $\bchi'$ are equal. A similar argument yields that also the iterated shadows are equal.
\end{proof}

Using the lemma, we define:

\begin{defn}\label{defn:alg:riso}
Suppose that $\bX \subseteq \mathbb A^n_R$ is a closed sub-scheme (for some $n \ge 1$), defined by an ideal $I \subseteq R[x_1, \dots, x_n]$. Then the riso-stratification $(\bS_i)_i$ obtained from any finite set of generators of $I$ as in Lemma~\ref{lem:ideal:riso}
is called the \emph{algebraic riso-stratification} of the affine embedded scheme $\bX$.
\end{defn}

\begin{rmk}
Given that $\bS_i$ is $\Lring(R)$-definable (and locally closed), one can consider it as a reduced locally closed subscheme of $\mathbb A_R^n$. This the present setting, this seems more natural, to stay in the algebraic world.
\end{rmk}

Here is an example showing that the algebraic riso-stratification sees the non-reduced structure of schemes:

\begin{exa}\label{ex:nonred}
Let $\bX \subseteq \mathbb A^2$ be defined by the ideal $(y^2, xy) \subseteq \ZZ[x,y]$ and let $(\bS_i)_i$ be its algebraic riso-stratification. We claim that $(0,0) \in \bS_0(K_1)$. Indeed, consider the map $\chi\colon K^2 \to \RV^{(2)}(K), (a,b)\mapsto \rv^{(2)}(b^2, ab)$ corresponding to the above two generators. To see the claim, it suffices to check that $\chi$ is not riso-trivial at all on the ball $\maxid(K)^2$. Indeed, one easily verifies that even the map $(a,b) \mapsto |(b^2, ab)| = |b|\cdot|(b,a)|$, which factors over $\chi$, is not riso-trivial at all.

In contrast, note that the corresponding reduced scheme $\bX^{\mathrm{red}}$ is just the $x$-axis (defined by $(y)$), so the algebraic riso-stratification of $\bX^{\mathrm{red}}$ has $\bS_0 = \emptyset$.
\end{exa}

While we introduced algebraic riso-stratification only for affine schemes coming with a closed embedding into some affine space, one could ask whether, if one replaces $\bS_i$ by $\bS_i \cap \bX$, one obtains a notion which does not depend on the embedding. If this is the case, one might moreover hope to obtain a notion of riso-stratification for any (not necessarily affine) scheme of finite type. We leave this to follow-up work.

\section{Application to motivic measure and Poincaré series}
\label{sec:appl}

A motivation for this paper was to get a better understanding of Poincaré series. In this section, we finally provide the concrete application mentioned in the introduction. To this end, we first need do show that risometries preserve motivic volumes (Section~\ref{sec:mot}), which also seems to be of independent interest.
In addition, we need that riso-triviality is preserved under restriction to (suitable) subfields (Section~\ref{sec:fields}). It turns out that both of these statements need the same additional technical assumption
(in addition to Hypothesis~\ref{hyp:KandX_can}), namely the one given in Hypothesis~\ref{hyp:RFlin}.
We start by explaining how this assumption is useful.

\subsection{Assumptions and the rigid partition}
\label{sec:prig}

In the entire Section~\ref{sec:appl}, in addition to Hypothesis~\ref{hyp:KandX_can}, we impose the following condition on our theory $\TT$:

\begin{hyp}\label{hyp:RFlin}
We assume that for every model $K \models \TT$, the only definable 
(with parameters) additive subgroups of $\RF(K)$ are $\{0\}$ and $\RF(K)$ itself.
\end{hyp}

Note that this is true for example if the residue field is algebraically closed and carries the pure field language, or if it is real closed and carries an o-minimal language.

A first consequence of this hypothesis is the following:

\begin{lem}\label{lem:RFlin}
For every model $K \models \TT$, every definable additive subgroup of $\RF(K)^n$ is a vector sub-space of $\RF(K)^n$.
\end{lem}

\begin{proof}
We need to verify that any definable additive subgroup $U \subseteq \RF(K)^n$ is closed under scalar multiplication. Pick $u \in U$ and apply Hypothesis~\ref{hyp:RFlin} to the set $\{r \in \RF(K) \mid ru \in U\}$;
since it contains $1$, it is equal to $\RF(K)$.
\end{proof}

Hypothesis~\ref{hyp:RFlin} implies that given any definable map $\chi\colon K^n \to \RV^\eq(K)$, the family of
maximal $1$-riso-trivial balls is ``rigid with respect to risometries'', i.e., any risometry sending $\chi$ to itself sends each such ball to itself. In the rest of this subsection, we make this precise.

\begin{defn}\label{defn:Prig}
Let $K \models \TT$ be a spherically complete model.
Given a map $\chi\colon B_0 \to S$ for some ball $B_0 \subseteq K^n$ (where $n \ge 1$),
we define the \emph{rigid partition} $\Prig_\chi \subseteq \pow(B_0)$ of $\chi$ as the set of the following balls and singletons:
The balls are all maximal balls $B \subseteq B_0$ on which $\chi$ is $1$-riso-trivial; the singletons are are those elements of $B_0$ which are not contained in any ball on which $\chi$ is $1$-riso-trivial.
\end{defn}

For the following lemma, recall the rigid core
$\Crig_{\chi}$ introduced in Definition~\ref{defn:Crig}.

\begin{lem}\label{lem:Prig}
Suppose that $K \models \TT$ be a spherically complete model, that
$B_0 \subseteq K^n$ is an $\LL$-definable ball and that
$\chi\colon B_0 \to \RV^\eq(K)$ is an $\Leq$-definable map. Then we have the following:

\begin{enumerate}
\item
$\Prig_\chi$ is an $\Leq$-definable partition of $B_0$. (Here $\Leq$-definability is in the sense of Convention~\ref{conv:sets:of:sets}, as a subset of the power set $\pow(B_0)$.)
\item Every $B \in \Prig_{\chi} \setminus \Crig_{\chi}$ is an open ball, and the smallest (closed) ball strictly containing $B$ contains an element of $\Crig_{\chi}$ as a subset.
\item 
Suppose that $\chi'\colon B'_0 \to \RV^\eq(K)$ is a second $\Leq$-definable map and that there exists a risometry $\phi\colon B_0 \to B_0'$ sending $\chi$ to $\chi'$. Then for each $B \in \Prig_\chi$, the image $\phi(B)$ is an element of $\Prig_{\chi'}$. Moreover, the induced map
$\tilde\phi\colon \Prig_\chi \to \Prig_{\chi'}$
is $\Leq$-definable, and it is already determined by $\chi$ and $\chi'$, i.e., for any other risometry $\phi'\colon B_0 \to B_0'$ sending $\chi$ to $\chi'$, we have $\tilde\phi' = \tilde\phi$.
\item
The definability statements in (1) and (3) hold uniformly in all spherically complete models of $\TT$, i.e., given $\Leq$-definable
$\bB_0, \bB'_0 \subseteq \VF^n$ and $\bchi\colon \bB_0 \to \RV^\eq,\bchi'\colon \bB'_0 \to \RV^\eq$,
the partition $\Prig_{\bchi_K}$ of $\bB_0(K)$ can be defined by an $\Leq$-formula not depending on $K$, and if for each spherically complete $K\models \TT$ there exists a risometry $\phi_K\colon \bchi_K \to \bchi'_K$, then
the induced map $\tilde\phi_K\colon  \Prig_{\bchi_K} \to \Prig_{\bchi'_K}$ can be defined by an $\Leq$-formula not depending on $K$.
\end{enumerate}
\end{lem}

Before we prove the lemma, we use it to introduce some definitions:

\begin{defn}\label{defn:bPrig}
Using Lemma~\ref{lem:Prig}, we define the following:
\begin{itemize}
 \item Given $\chi$ and $\chi'$ as in (3) which are in risometry, we denote by $\Prig_{\chi\to\chi'}$ the map from $\Prig_\chi$ to $\Prig_{\chi'}$ induced by any risometry from $\chi$ to $\chi'$.
\item
  Given $\bchi\colon \bB_0 \to \RV^\eq, \bchi'\colon \bB'_0 \to \RV^\eq$ as in (4), we write $\bPrig_\bchi$ the $\Leq$-definable partition of $\bB_0$ satisfying
  $\bPrig_\bchi(K) = \Prig_{\bchi_K}$ for every spherically complete $K \models \TT$, and if $\bchi_K$ and $\bchi'_K$ are in risometry for every
  such $K$, then we write $\bPrig_{\bchi\to\bchi'}$ for the $\Leq$-definable map from $\bPrig_\bchi$ to $\bPrig_{\bchi'}$ induced by such risometries.
\end{itemize}
\end{defn}

\begin{proof}[Proof of Lemma~\ref{lem:Prig}]
We prove the uniformity statements from (4) along with (1) and (3), respectively. (In those uniformity proofs, $K$ always runs over all spherically complete models of $\TT$.)

(1)
Definability of $\Prig_\chi$ follows from the Riso-Triviality Theorem
(more specifically from Corollary~\ref{cor:RTT}). Note that this also yields the uniform definability from (4). 

To see that $\Prig_\chi$ is a partition, first note that it is a collection of disjoint balls and singletons, so it suffices to verify that given any $x \in B_0$, we can find a $B \in \Prig_\chi$ containing $x$.
If $x$ does not lie in $C := \bigcup_{B \in \Crig_\chi} B$, then we have $x \in B_{<\lambda}(x) \in \Prig_\chi$, where $\lambda$ is the distance from $x$ to $C$,
so now suppose that $x \in B \in \Crig_\chi$. If $B$ is a singleton, then it is not contained in any $1$-riso-trivial ball (by definition of $\Crig_\chi$) and hence it is an element of $\Prig_\chi$. Now assume that $B$ is a ball.
If $\chi$ is $1$-riso-trivial on $B$, then by definition of $\Crig_\chi$, $\chi$ is not riso-trivial on every ball strictly containing $B$, so again $B \in \Prig_\chi$. Otherwise, by Lemma~\ref{lem:Crig} (3), $B$ is a closed ball, and the maximal proper subball of $B$ containing $x$ lies in $\Prig_\chi$.

The above case distinction also proves (2).

(3)
Clearly, $\phi$ sends $\Crig_\chi$ to $\Crig_{\chi'}$.
By Lemma~\ref{lem:fin-riso}, the induced map
$\psi\colon \Crig_\chi \to \Crig_{\chi'}$ is
independent of $\phi$, and $\psi$ is $\Leq$-definable uniformly in $K$ (see Remark~\ref{rmk:fin-riso}).
This already defines $\Prig_{\chi\to\chi'}(B)$ for $B \in \Prig_\chi \cap \Crig_\chi$.
Now consider a $B \in \Prig_\chi$ which is 
disjoint from $C := \bigcup_{B' \in\Crig_\chi} B'$.
Since $B$ is a maximal ball disjoint from $C$ and $\phi$ is a risometry,
$\psi$ uniquely determines $\phi(B)$, so $\phi(B)$ is independent of $\phi$, too, and the restriction of $\Prig_{\chi\to\chi'}$ to the set of those $B$ is also $\Leq$-definable (also uniformly in $K$). It remains to consider $\Prig_{\chi\to\chi'}(B)$ for balls $B$ that are strictly contained in a ball $\tilde B \in \Crig_\chi$.
Recall (from the proof of (1)) that this only happens when $\tilde B$ is a closed ball and $B$ is a maximal proper subball of $\tilde B$.

Let $R$ and $R'$ be the sets of maximal proper subballs of $\tilde B$ and $\phi(\tilde B) = \psi(\tilde B)$, respectively. The map $R \to R'$ induced by $\phi$ is a translation (in the sense that there exists a $z \in K^n$ such that $\phi(B_1) = B_1 + z$ for every $B_1 \in R$). We will show that there exists only one translation $\tau\colon R \to R'$ such that $\chi|_{B_1}$ is in risometry with $\chi'|_{\tau(B_1)}$ for every $B_1 \in R$. This implies that $\phi(B) = \tau(B)$ does not depend on $\phi$ and that $\phi(B)$ is definable from $B$ (namely, using the Riso-Equivalence Theorem to define $\tau$).

To show that only one such translation exists, we may assume without loss that $\tilde B = \valring(K)^n$, so that $R$ can be identified with $\RF(K)^n$.
Let $X$ be the set of those $x \in \RF(K)^n$ such that $\chi|_{\res^{-1}(y)}$ and $\chi|_{\res^{-1}(x+y)}$ are in risometry for every $y \in \RF^n(K)$.
If two different translations $\tau_1, \tau_2\colon R = \RF(K)^n \to R'$ as above exist, then the composition $\tau_1^{-1}\circ \tau_2\colon \RF(K)^n \to \RF(K)^n$ is a translation by
a non-zero element of $X$, so to obtain that only one $\tau$ exists, it suffices to show that $X = \{0\}$.

Using that risometries can be composed, one obtains that $X$ is an additive subgroup of $\RF(K)^n$. By Lemma~\ref{lem:RFlin}, it is therefore an $\RF(K)$-vector subspace. Suppose for contradiction that $X$ is strictly bigger than $\{0\}$ and fix any $1$-dimensional vector subspace $\bar V \subseteq X$.
We will prove that for every $a \in \RF(K)^n$, $\chi$ is $\bar V$-riso-trivial on $B_a := \res^{-1}(a)$. Together with $\bar V \subseteq X$, this then implies that $\chi$ is $\bar V$-riso-trivial on all of $\valring(K)^n$, which contradicts that $\chi$ is not riso-trivial at all on $\tilde B = \valring(K)^n$.

Fix $a \in \RF(K)^n$ and suppose that $\bar V$ is not contained in $\bar W:= \rtsp_{B_a}(\chi)$.
The we can choose a projection $\bar\pi\colon \RF(K)^n \to \bar W$ satisfying $\bar V \subseteq \ker \bar\pi$. Let $\pi\colon K^n \to W$ be a lift, and choose a fiber $F \subseteq K^n$ of $\pi$ satisfying $F \cap B_a \ne \emptyset$. Then $\chi|_{F \cap B_a}$ is not riso-trivial at all.
Moreover, for every $a' \in \RF(K)^n$ satisfying $a' - a \in \bar V \subseteq X$, we also have $\rtsp_{B_{a'}}(\chi)=\bar W$ (by definition of $X$)
and $F \cap B_{a'} \ne \emptyset$ (by the choice of $\pi$),
so  $\chi|_{F \cap B_{a'}}$ is not at all riso-trivial either.
We thus just obtained infinitely many disjoint balls $F \cap B_{a'}$ in $F$ on each of which $\chi$ is not riso-trivial at all. This contradicts that the rigid core $\Crig_{\chi|_F}$ of $\chi$ restricted to $F$ is finite.
\end{proof}

\subsection{Restricting to subfields}
\label{sec:fields}

In this subsection, we prove that 
under suitable assumptions, riso-triviality in a valued field $K$ restricts nicely to substructures $K_0 \subseteq K$. This will be needed to relate the
riso-stratifications of $\CC^n$ constructed in Section~\ref{sec:can_whit} to Poincaré series, since 
riso-stratifications are defined using a nonstandard extension of $\CC$, which has a divisible value group, whereas Poincaré series are defined
using $\CCt$.

We work in the following setting:
\begin{itemize}
 \item[($\star$)]
We assume Hypotheses~\ref{hyp:KandX_can} and \ref{hyp:RFlin},
we fix a spherically complete model $K \models \TT$, and we moreover fix a subfield
$K_0 \subseteq K$ (not necessarily a model of $\TT$) satisfying $\dclVF(K_0) = K_0$
and $\dclRV(K_0) = \rv(K_0)$.
\end{itemize}

\begin{defn}
In this subsection,
given any set $X \subseteq K^n$, we write $X_0 := X \cap K_0^n$ for the corresponding set in $K_0$.
We call a bijection $\phi\colon X \subseteq K^n \to X' \subseteq K^{n'}$ \emph{compatible with $K_0$} if it restricts to a bijection $X_0 \to X'_0$.
\end{defn}

The main result of this subsection is the following:

\begin{prop}\label{prop:fields}
Suppose that $K$ and $K_0$ are as in ($\star$).
Let $B \subseteq K^n$ (for some $n \ge 1$) be a $K_0$-definable ball satisfying $B_0 \ne \emptyset$ and let $\chi, \chi'\colon B \to \RV^\eq(K)$ be $K_0$-definable maps. Then the following hold:
\begin{enumerate}
    \item If $\chi$ is $d$-riso-trivial for some $d$, then there exists a straightener $\phi\colon B \to B$ of $\chi$ witnessing this $d$-riso-triviality which is compatible with $K_0$ and such that moreover, $\chi \circ \phi$ is definable with parameters from $K_0$.
    \item If $\chi$ and $\chi'$ are in risometry, then there exists a risometry from $\chi$ to $\chi'$ which is compatible with $K_0$.
\end{enumerate}
\end{prop}

Recall (Definition~\ref{defn:Prig}, Lemma~\ref{lem:Prig}) that if there exists a risometry from $\chi$ and $\chi'$, then we have an induced map $\Prig_{\chi\to \chi'}\colon \Prig(\chi) \to \Prig(\chi')$. The following lemma relates this to $K_0$:

\begin{lem}\label{lem:fields:Crig}
Suppose that $\chi$ and $\chi'$ are as in Proposition~\ref{prop:fields} and that $\phi\colon B \to B$ is a risometry from $\chi$ to $\chi'$.
 For $B' \in \Prig(\chi)$, the following are equivalent:
\begin{enumerate}
    \item $B'_0 \ne \emptyset$.
    \item $B'$ is $K_0$-definable.
    \item $\phi(B')_0 \ne \emptyset$.
    \item $\phi(B')$ is $K_0$-definable.
\end{enumerate}
\end{lem}

\begin{proof}
The equivalence (2) $\Leftrightarrow$ (4) follows from
$\Prig_{\phi\to\phi'}$ being $K_0$-definable (by Lemma~\ref{lem:Prig}), so we only need to prove (1) $\Leftrightarrow$ (2) (since (3) $\Leftrightarrow$ (4) then holds by symmetry).

The implication
``(1) $\Rightarrow$ (2)'' follows from $\Prig(\chi)$ being $K_0$-definable: Pick $a \in B' \cap K_0^n$; then $B'$ is the (unique) element of $\Prig(\chi)$ containing $a$.

``(2) $\Rightarrow$ (1)'': 
By Lemma~\ref{lem:Crig-def} (2), we find a $K_0$-definable set $C$ consisting of one point in each element of $\Crig_\chi$. In the case $B' \in \Crig_\chi$ we thus obtain that the unique element of $C \cap B'$ lies in $\dclVF(K_0)$ and hence in $K_0$.
If $B' \notin \Crig_\chi$, then $B'$ is an open ball, and the smallest (closed) ball $B''$ strictly containing $B'$ contains an element of $\Crig_\chi$ as a subset (by Lemma~\ref{lem:Prig}). In particular,
the intersection $C' := C \cap B''$ is non-empty. Note that it is also $K_0$-definable, so its barycenter, which we denote by $c$, lies in $\dclVF(K_0)$. If $c \in B'$, we are done. Otherwise, we can write $B'$ as $B' = c + \rv^{-1}(\xi)$, for some $\xi \in \RV(K)$. By $K_0$-definability of $B'$ and $c$, we obtain $\xi \in \dclRV(K_0) = \rv(K_0)$ (where the last equality holds by ($\star$)). Pick $b \in K_0$ satisfying $\rv(b) = \xi$. Then $c + b \in B'$ does the job.
\end{proof}

\begin{proof}[Proof of Proposition~\ref{prop:fields}]
Set $\bar V := \rtsp_B(\chi) \subseteq \RF(K)^n$ and $d := \dim \bar V$. We prove both statements in a common nested induction: an outer induction on $n$, and for fixed $n$, an induction on $n-d$. If $n-d=0$, both statements are trivial: the straightener in (1) and the risometry in (2) can be taken to be the identity.

As a preparation common for the inductive steps of both statements, we choose $\bpi\colon \colon \RF(K)^n \to \bar V$ and a lift $\pi\colon B \to V$, all of which are $K_0$-definable. Those can be obtained as follows:
By the Riso-Triviality Theorem~\ref{thm:RTT}, $\bar V$ is $K_0$-definable.
After a suitable permutation of coordinates,
$\bar V$ is the graph of a (linear) function $f\colon \RF(K)^d \to \RF(K)^{n-d}$ and we can define $\bpi\colon \RF(K)^n \to \bar V$ to be the projection whose kernel is $\{0\}^d \times \RF(K)^{n-d}$.
Denote by $\bar e_1, \dots, \bar e_d$ the standard basis of $\RF(K)^d$. Then the basis $(\bar v_i)_i := (\bar e_i, f(\bar e_i))_i$ of $\bar V$ lies in $\dclRF(K_0) \subseteq \dclRV(K_0)$. Since by ($\star$), this is equal to $\rv(K_0)$, there exist $v_i \in \valring(K_0)^n$ satisfying $\res(v_i) = \bar v_i$. Let $V$ be the vector sub-space of $K^n$ spanned by those $v_i$ and let $\pi\colon B \to V$ be the projection whose kernel is $\{0\}^d \times K^{n-d}$.

\medskip

Proof of (1):

We may assume $d \ge 1$ (otherwise, take $\phi$ to be the identity map).

Let $\phi\colon B \to B$ be a $V$-straightener of $\chi$ respecting
$\pi$-fibers (i.e., $\pi \circ \phi = \pi$). We fix some $y_1 \in \pi(B_0)$ and additionally assume that
$\phi$ is the identity on $\pi^{-1}(y_1)$ (using Remark~\ref{rmk:fix:fiber}). In particular, $\chi \circ \phi$ is $K_0$-definable, since it can be defined in terms of its restriction to the fiber over $y_1$ and $V$ (both of which are $K_0$-definable).

We will now partition $B$ into balls and singletons $B'$ and modify $\phi$ on each of those $B'$ to make it compatible with $K_0$. More precisely, for each $B'$, we will find a risometry $\phi'\colon B' \to \phi(B')$ compatible with $K_0$ and satisfying $\chi \circ \phi' = \chi \circ \phi|_{B'}$. (If $B'$ is a singleton, then obviously we have to set $\phi' = \phi|_{B'}$, but we still have to check that this $\phi'$ is compatible with $K_0$.) We then let
$\phi''\colon B \to B$ be the union of all those $\phi'\colon B' \to B'$. This is a risometry compatible with $K_0$, and since $\chi \circ \phi'' = \chi \circ \phi$, it is a straightener of $\chi$, as desired.

We choose the partition of $B$ to consist of those 
balls and singletons $B'$ satisfying the following property: We pick any $y \in \pi(B')$ and require that,
for $F := \pi^{-1}(y) \subseteq B$,
$B' \cap F$ is an element of $\Prig((\chi\circ\phi)|_{F})$.
Note that from $V$-translation invariance of $\chi \circ \phi$,
one obtains that the choice of $y \in \pi(B')$ does not matter: given another $y' \in \pi(B')$ and $F' := \pi^{-1}(y')$,
we have $B' \cap F \in \Prig((\chi\circ\phi)|_{F})$ if and only if $B' \cap F' \in \Prig((\chi\circ\phi)|_{F'})$.
That those sets $B'$ form a partition of $B$ then follows from $\Prig((\chi\circ\phi)|_{F})$ being a partition of $F$ for each $\pi$-fiber $F$.

Now fix a ball $B'$ from our partition of $B$; we need to find a $\phi'\colon B' \to \phi(B')$ with the desired properties, i.e., compatible with $K_0$ and satisfying $\chi \circ \phi' = \chi \circ \phi|_{B'}$.
If $\pi(B')_0 = \emptyset$, then we have $B'_0 = \emptyset$ and $\phi(B')_0 = \emptyset$ (since $\pi(\phi(B')) = \pi(B')$) and we can simply take $\phi' = \phi$, so suppose now that $\pi(B')_0 \ne \emptyset$.

Pick $y \in \pi(B')_0$ and set $F := \pi^{-1}(y)$.
By applying Lemma~\ref{lem:fields:Crig} to $\phi|_{F}\colon(\chi\circ\phi)|_{F} \to  \chi|_{F}$, we obtain that
$B' \cap F$ is disjoint from $K_0^n$ if and only if
$\phi(B' \cap F)$ is. In the disjoint case, we can again simply take $\phi' = \phi$,
so suppose now that $B' \cap F \cap K_0^n \ne \emptyset$.
In particular (again by Lemma~\ref{lem:fields:Crig}), $B' \cap F$ and $\phi(B' \cap F)$ are $K_0$-definable, and hence so are $B'$ and $\phi(B')$.

If $B'$ is a singleton, this implies $B' \subseteq K_0^n$ and
$\phi(B') \subseteq K_0^n$ so we are done: $\phi' := \phi|_{B'}$ is compatible with $K_0$. Otherwise, 
since $B' \cap F \in \Prig((\chi\circ\phi)|_{F})$, 
$(\chi\circ\phi)|_{B' \cap F }$ is $1$-riso-trivial. Therefore, 
$(\chi\circ\phi)|_{B'}$ is $(d+1)$-riso-trivial, so by induction on $n-d$, we can apply (2)
to $\phi|_{B'}$ to obtain (as desired) our risometry $\phi'\colon B' \to \phi(B')$ compatible with $K_0$ and sending $(\chi \circ \phi)|_{B'}$
to $\chi|_{\phi(B')}$.

\medskip

Proof of (2):
Let $\chi$ and $\chi'$ be in risometry.

\medskip

Case $d \ge 1$:

By (the proof of) (1), there exists a risometry $\phi$ compatible with $K_0$ such that $\chi \circ \phi$ is $V$-translation invariant and $K_0$-definable, so we may without loss assume that $\chi$ itself is $V$-translation invariant, and similarly for $\chi'$. Fix any $y_0 \in\pi(B_0)$. By induction on $n$, we find a risometry $\phi$
from $\chi|_{\pi^{-1}(y_0)}$
to $\chi'|_{\pi^{-1}(y_0)}$ which is compatible with $K_0$. We then extend $\phi$ to a map $B \to B$ in a $V$-translation invariant way. This $\phi$ is clearly a risometry from $\chi$ to $\chi'$, and it is compatible with $K_0$ by our choice of $V$.

\medskip

Case $d = 0$: We define the risometry $\phi$ from $\chi$ to $\chi'$
by specifying $\phi|_{B_1}\colon B_1 \to B_1'$
separately for each $B_1 \in \Prig_\chi$, and where $B_1' := \Prig_{\chi \to \chi'}(B_1)$.
If $B_1$ is a singleton, then we let $\phi|_{B_1}$ be the map sending that singleton to $B'_1$; this is compatible with $K_0$ by Lemma~\ref{lem:fields:Crig}.
If $B_1$ is a ball disjoint from $K_0^n$, then so is $B'_1$, and
we can pick $\phi|_{B_1}$ to be an arbitrary risometry from $\chi|_{B_1}$
to $\chi'|_{B_1'}$.
Finally, if $B_1$ is a ball which has non-empty intersection with $K_0^n$, then by Lemma~\ref{lem:fields:Crig}, both $B_1$ and $B'_1$ are $K_0$-definable, so we can find $\phi|_{B_1}\colon B_1 \to B_1'$ using the above case $d \ge 1$.
\end{proof}

\subsection{Risometries preserve motivic measure}\label{sec:mot}
In this subsection, we make precise and prove the following statement:
If Hypothesis~\ref{hyp:RFlin} holds and $\bX_1$ and $\bX_2$ are two definable sets which are in risometry (where the risometry is not assumed to be definable), then they have the same motivic measure in the sense of Cluckers--Loeser motivic integration.

We start by fixing an appropriate setting (where motivic integration works) and by recalling some notation and terminology around motivic integration.

\begin{hyp}\label{hyp:mot:int}
In this subsection, in addition to Hypotheses~\ref{hyp:KandX_can} and \ref{hyp:RFlin}, we assume:
\begin{enumerate}
 \item We have an $\Leq$-definable angular component map $\ac\colon \VF \to \RF$. In particular, we can and will identify $\RV$ with $(\RF^\times \times \VG) \cup \{0\}$.
 \item Each model $K \models \TT$ is elementarily equivalent to an $\LL$-structure whose underlying valued field is of the form $\kt$, for $k$ a field of characteristic $0$. In particular, $\TT$ states that the value group is elementarily equivalent to $\ZZ$.
 \item
 The structure on $\VG$ induced by $\Leq$ is the pure ordered abelian group structure.
 \item $\VG$ and $\RF$ are orthogonal.
\end{enumerate}
\end{hyp}

By \cite[Theorem~5.8.2]{iCR.hmin}, those conditions imply that the main assumptions of \cite{CL.mot0p} are satisfied, meaning that Cluckers--Loeser motivic integration is well-defined and works as desired.

The models of $\TT$ we will consider in this subsection will all be of the form $K = \kt$. Note that those are in particular spherically complete, since they are complete and the value group is discrete.
Let us write $\mathcal S$ for the set of models of $\TT$ of this form.
By Hypothesis~\ref{hyp:mot:int}~(2), an $\Leq$-definable set $\bX \subseteq \VF^m \times \RF^n \times \VG^r$ is determined by the family of sets $(\bX(K))_{K \in \mathcal S}$. Such a family $(\bX(K))_{K \in \mathcal S}$ is called a ``definable subassignment of $h[m,n,r]$'' in \cite{CL.mot} and an ``$\mathcal S$-definable set'' in \cite{iC.eval}. In the present paper, we apply terminology from  \cite{CL.mot,iC.eval} writing $\bX$ instead of $(\bX(K))_{K \in \mathcal S}$.
In particular, we write $\cC(\bX)$ for the 
ring of constructible motivic functions on $\bX$ (introduced in \cite[Section~5]{CL.mot}). This ring depends on the theory $\TT$. If we want to consider the corresponding ring for a theory $\TT' \supseteq \TT$ (possibly in an extended language), we denote it by $\cC_{\TT'}(\bX)$.

Recall that \cite{CL.mot} provides a notion of integrability of constructible motivic functions, and that given an integrable $\bff \in \cC(\bX)$,
its motivic integral, denoted by $\pr_!\bff$, is an element of $\cC(\pt)$, where $\pt$ stands for any $\LL$-definable singleton set (e.g.\ $\pt = \{0\} \subseteq \VF$; the specific choice of $\pt$ does not matter) and where $\pr\colon \bX \to \pt$ is the projection. Instead of writing $\pr_!\bff$, we will use the more suggestive notation $\int_{\bX} \bff$ for the motivic integral of $\bff$. We will also use other suggestive notation, e.g., $\int_{\bY} \bff$ for integral of the restriction of $\bff$ to $ \bY$, when $\bY \subseteq \bX$.

Given $\bff \in \cC(\bX \times \bY)$ and
$\pr\colon \bX \times \bY \to \bY$ the projection,
there is also a notion of relative integral: If $\bff$ is relatively integrable over $\bY$, then we have a relative integral $\pr_{!\bY}\bff = \pr_!\bff \in \cC(\bY)$, which one should think of as
``$y \mapsto \int_{\bX_y} \bff(x,y)\,dx$''; see \cite[Section~14.2]{CL.mot}.

Motivic integration yields a notion of motivic measure of a definable set, namely
$\mu(\bX) := \int_{\bX} \mathbbm1_{\bX}$, where $\mathbbm1_{\bX} \in \cC(\bX)$ is the constant $1$ function. (Here, we assume that $\mathbbm1_{\bX}$ is integrable; this is always the case if $\bX$ is bounded).
Recall that the motivic measure is normalized in such a way that the valuation ring has measure $1$ and that $\bbL \in \cC(\pt)$ denotes the inverse of the motivic measure of the maximal ideal.

In this paper, we only consider the ``top-dimensional motivic measure'', i.e., if $\bX \subseteq \VF^m$ has dimension less than $m$, then its measure $\mu(\bX)$ is $0$.
Formally, in terms of the graded ring $C(\bX)$ of constructible uppercase-F-Functions from \cite[Section~6]{CL.mot}, we consider an $\bff \in \cC(\bX)$ as an element of $C^m(\bX)$ when $\bX \subseteq \VF^m \times \RF^n \times \VG^r$.

We can now state a first version of the main result of this subsection; recall that we assume Hypothesis~\ref{hyp:mot:int}:

\begin{prop}\label{prop:presMeas}
Let $\bB_1, \bB_2 \subseteq \VF^n$ be $\LL$-definable balls (for some $n \ge 1$) and let $\bX_i \subseteq \bB_i$ be $\LL$-definable sets (for $i=1,2$). Suppose that for every model of the form $K = \kt \models \TT$, there exists a risometry $\phi\colon \bB_{1}(K) \to \bB_{2}(K)$
sending $\bX_{1}(K)$ to $\bX_{2}(K)$.
Then the motivic measures $\mu(\bX_1), \mu(\bX_2) \in \cC(\pt)$ are equal.
\end{prop}

This can be generalized to a statement about integrals of motivic functions $\bff_i$ on $\bB_i$ being equal for $i=1,2$, though there are some subtleties. Here is a precise formulation:

\begin{prop}\label{prop:presInt}
Let $\tilde \bff \in \cC(\RV^N)$ be a motivic function for some $N \ge 1$.
For $i=1,2$, let
$\bB_i \subseteq \VF^n$ be an $\LL$-definable ball, let $\bh_i\colon \bB_i \to \RV^N$ be an $\Leq$-definable map, and set 
$\bff_i := \bh_i^*(\tilde \bff) \in \cC(\bB_i)$. Suppose that for every model of the form $K = \kt \models \TT$, there exists a risometry $\phi\colon \bB_{1}(K) \to \bB_{2}(K)$
from $\bh_{1,K}$ to $\bh_{2,K}$.
Then the motivic integrals 
$\int_{\bB_1} \bff_1, \int_{\bB_2} \bff_2 \in \cC(\pt)$ are equal.
\end{prop}

Recall that $\bh_i^*(\tilde \bff)$ is the pull-back of the motivic function $\tilde \bff$ along $\bh_i$. Note also that the condition that $\bff_i$ is of the form $\bh_i^*(\tilde \bff)$ is not a restriction: every motivic function in $\cC(\bB_i)$ can be written like this for some $N$. However, asking 
$\phi$ to be a risometry from $\bh_{1,K}$ to $\bh_{2,K}$ seems to be stronger than just asking that it is a risometry from $\bff_{1,K}$ to $\bff_{2,K}$; to be more precise, it is not even clear what the latter should mean.

Clearly, Proposition~\ref{prop:presMeas} is a special case of Proposition~\ref{prop:presInt}, namely, by taking $\bh_i\colon \bB_i \to \RV$ to be the indicator function of $\bX_i$ and $\tilde\bff \in \cC(\RV)$ the indicator function of $\{1\} \subseteq \RV$.

\medskip

In the proof, we will use that motivic functions $\bff \in \cC(\bX)$ are determined by their values in the sense of \cite{iC.eval}; we quickly recall this result. 
By a point of $\bX$, we mean an element $a \in \bX(K)$
for some model of the form $K = \kt \models \TT$. Given such a point $a$ (we think of $K$ as implicitly being part of the datum of $a$), the value of $\bff$ at $a$ is defined as $\bff(a) := \bff|_{\{a\}} \in \cC_{\TT'}(\pt)$, where $\TT' = \Th_{\LL(a)}(K)$ is the theory of $K$ in the language $\LL$ extended by a tuple of constant symbols for $a$ (so that $\{a\}$ is in $\LL(a)$-definable bijection to $\pt$).\footnote{This notion of evaluation of motivic functions is not the same as the one from \cite[Section~5.4]{CL.mot}.}
By \cite[Theorem~1]{iC.eval}, if two motivic functions $\bff_1, \bff_2 \in \cC(\bX)$ take the same value at every $a \in \bX(K)$ for every $K = \kt \models \TT$, then $\bff_1 = \bff_2$ in $\cC(\bX)$.

\begin{proof}[Proof of Proposition~\ref{prop:presInt}]
In the entire proof, $K$ will always be a model of $\TT$ of the form $K = \kt$.

The existence of the risometry $\phi\colon \bh_{1,K} \to \bh_{2,K}$ implies that we have $\rtsp_{\bB_1(K)} \bh_{1,K} = \rtsp_{\bB_2(K)} \bh_{2,K}$.
By the Riso-Triviality Theorem~\ref{thm:RTT}, this vector space is $\Leq$-definable; as such, denote it by $\bbU$, i.e., $\bbU(K) = \rtsp_{\bB_1(K)} \bh_{1,K} = \rtsp_{\bB_2(K)} \bh_{2,K}$ for every $K$. Given any $\LL$-sentence $\theta$, the desired equality $\int_{\bB_1} \bff_1= \int_{\bB_2} \bff_2$ holds in $\cC(\pt)$ if and only it holds in 
$\cC_{\TT \cup \{\theta\}}(\pt)$
and in $\cC_{\TT \cup \{\neg\theta\}}(\pt)$.
This allows us to assume that $d := \dim \bbU(K)$ does not depend on $K$ and that there exists a coordinate projection $\pi\colon \VF^n \to \VF^d$
such that for every $K$, the corresponding projection $\bar \pi\colon \RF(K)^n \to \RF(K)^d$ induces a bijection $\bbU(K) \to \RF(K)^d$.

Note that the existence of the risometry $\phi$ implies that $\mu(\bB_1) = \mu(\bB_2)$ in $\cC(\pt)$.

\medskip

\textbf{Case 1:} $d = n$.

In this case, the maps $\bh_{1,K}$ and $\bh_{2,K}$ are constant and equal for each $K$, i.e., there exists an $\Leq$-definable element $\bh_0 \in \RV^N$ such that
$\bh_{i,K}(x) = \bh_{0,K}$ for every $K$ and every $x \in \bB_{i}(K)$. 
Writing $\bj\colon \pt \to \RV^N$ for the $\Leq$-definable map sending the point to $\bh_0$, we obtain
\[
\int_{\bB_i}\bff_i = \mu(\bB_i)\cdot \bj^*(\tilde \bff) \in \cC(\pt).
\]
Since $\mu(\bB_1) = \mu(\bB_2)$, the proposition follows.

\medskip

\textbf{Case 2:} $1 \le d  < n$.

We use an induction over $n$, more precisely assuming that the proposition holds when the ambient space has dimension $n - d$.

For $i=1,2$, set $\bfg_i := \pr_i^*\pi_!(\bff_i)\in \cC(\pi(\bB_1) \times \pi(\bB_2))$,
where $\pr_i\colon \pi(\bB_1) \times \pi(\bB_2) \to \pi(\bB_i)$ is the corresponding projection.
(In informal notation, we have $\bfg_i(y_1, y_2) = \int_{\pi^{-1}(y_i) \cap \bB_i}\bff_i(x)\,dx$; in particular, $\bfg_i$ only depends on the variable $y_i$.) In particular,
we have
\begin{equation}\label{eq:g:f}
\int_{\pi(\bB_1) \times \pi(\bB_2)} \bfg_1
= \mu( \pi(\bB_2))\cdot \int_{\pi(\bB_1)}\pi_! (\bff_1)
=\mu( \pi(\bB_2))\cdot\int_{\bB_1} \bff_1,
\end{equation}
and analogously for $\bfg_2$.

Fix a point $(b_1, b_2) \in \pi(\bB_{1}(K)) \times \pi(\bB_{2}(K))$ for some $K  = \kt \models \TT$,
set $\TT' := \Th_{\LL(b_1, b_2)}(K)$, and
denote by $\bF_i := \pi^{-1}(b_i) \cap \bB_{i}$ the corresponding fibers, considered as $\LL(b_1, b_2)$-definable sets in the theory $\TT'$.
By Lemma~\ref{lem:fiber_check},
there exists a risometry
from $\bh_{1,K}|_{\bF_1(K)}$ to $\bh_{2,K}|_{\bF_2(K)}$, so by induction (applied in $\TT'$), we obtain
\[\bfg_1(b_1, b_2) = \int_{\bF_1} \bff_1 = \int_{\bF_2} \bff_2 = \bfg_2(b_1, b_2)\]
in $\cC_{\TT'}(\pt)$.

Since this equality holds for every point $(b_1, b_2)$
of $\pi(\bB_1) \times \pi(\bB_2)$,
\cite[Theorem~1]{iC.eval} implies that $\bfg_1$ and $\bfg_2$ are equal as elements of $\cC(\pi(\bB_1) \times \pi(\bB_2))$.
Combining this with \eqref{eq:g:f} yields
\[
 \mu(\pi(\bB_2))\cdot \int_{\bB_1} \bff_1 
=
\int_{\pi(\bB_1) \times \pi(\bB_2)} \bfg_1 =
\int_{\pi(\bB_1) \times \pi(\bB_2)} \bfg_2 =
 \mu(\pi(\bB_1))\cdot \int_{\bB_2} \bff_2 .
\]
Since $\mu(\pi(\bB_1))$ and $\mu(\pi(\bB_2))$
are equal and invertible in $\cC(\pt)$, we obtain $\int_{\bB_1} \bff_1 = \int_{\bB_2} \bff_2$, as desired.

\medskip

\textbf{Case 3:} $d = 0$.

Let $\bB_i' \subseteq \bB_i$ be the union of all those elements of $\bPrig_{\bh_i}$ which are balls. Since the difference $\bB_i \setminus \bB_i'$ is finite, we have 
$\int_{\bB_i} \bff_i = \int_{\bB_i'} \bff_i$,
so it suffices to prove $\int_{\bB_1'} \bff_1 = \int_{\bB_2'} \bff_2$.

Given $K = \kt \models \TT$ and $a \in \bB'_{i}(K)$, we write $B_a \subseteq K^n$
for the ball in $\bPrig_{\bh_i}(K)$ containing $a$.

Set $\bZ := \{(x, y) \in \bB'_1 \times \bB'_1 \mid B_x = B_{y}\}$ and denote by $\mathbbm1_{\bZ} \in \cC(\bB'_1 \times \bB'_1)$ its indicator function.
Let $\bfg \in\cC(\bB'_1 \times \bB'_1)$ be the motivic function defined by
\[
\bfg(x, y) = \bff_1(y) \cdot \mathbbm1_{\bZ}(x,y) \cdot \mu(B_y)^{-1}.
\]
More formally, by ``$\mu(B_y)^{-1}$'', we mean $\bbL^{n\cdot \balpha}$, where $\balpha\colon \bB'_1 \to \VG$ is the definable function sending $y$ to $\radcl(B_y)$, and where we temporarily use additive notation for $\VG$.
The (total) integral of $\bfg$ is equal to
\[
\int_{\bB'_1}\int_{\bB'_1}\bfg(x, y)\, dx \, dy
= \int_{\bB'_1}  \bff_1(y) 
\cdot \underbrace{\int_{\bB'_1} \mathbbm1_{\bZ}(x,y) dx}_{=\mu(B_y)} \cdot \mu(B_y)^{-1}\,dy = \int_{\bB'_1} \bff_1.
\]
In a similar way, we set $\bZ' := \{(x, y) \in \bB'_1 \times \bB'_2 \mid \bPrig_{\bh_1\to \bh_2}(B_x) = B_{y}\}$, define $\bfg' \in\cC(\bB'_1 \times \bB'_2)$ by
\[
\bfg'(x, y) = \bff_2(y) \cdot \mathbbm1_{\bZ'}(x,y) \cdot \mu(B_y)^{-1}
\]
and obtain
\[
\int_{\bB'_2}\int_{\bB'_1}\bfg'(x, y)\, dx \, dy
= \int_{\bB'_2} \bff_2.
\]
So it remains to show that the integrals of $\bfg$ and $\bfg'$ are equal. We do this by showing that the two functions become equal (in $\cC(\bB'_1)$) after integrating out the $y$-variable.
Since motivic functions are determined by their values \cite[Theorem~1]{iC.eval}, it suffices to show that for every $K = \kt \models \TT$ and every $a \in \bB'_{1}(K)$, we have
\begin{equation}\label{eq:g=g'}
\int_{\bB'_1}  \bff_1(y) \cdot \mathbbm1_{\bZ}(a,y) \cdot \mu(B_y)^{-1}\, dy
=\int_{\bB'_2}  \bff_2(y) \cdot \mathbbm1_{\bZ'}(a,y) \cdot \mu(B_y)^{-1}\, dy
\end{equation}
in $\cC_{\TT'}(\pt)$, where $\TT' = \Th_{\LL' }(K)$ and $\LL' = \LL(a)$.
From now on, we work in those $\LL'$ and $\TT'$. Write $\bB_a$ for $B_a$ considered as a model-independent $\LL'$-definable set (so that $B_a = \bB_{a}(K)$).
Let us consider the left hand side of \eqref{eq:g=g'} first: We have
\[
\int_{\bB'_1}  \bff_1(y) \cdot \mathbbm1_{\bZ}(a,y) \cdot \mu(B_y)^{-1}\, dy
=
\int_{\bB_a}  \bff_1(y) \cdot \mu(B_y)^{-1}\, dy
=\mu(\bB_a)^{-1}\cdot\int_{\bB_a}  \bff_1,
\]
where the last equality uses that $B_y = B_a$ for every $y \in B_a$. In a similar way, we obtain that the right hand side of \eqref{eq:g=g'} is equal to
$\mu(\bB'_a)^{-1}\cdot\int_{\bB'_a}  \bff_2$,
where $\bB'_a := \bPrig_{\bh_1\to \bh_2}(\bB_a)$.
Now equality of the two sides follows
from $\mu(\bB_a) = \mu(\bB'_a)$
and by applying the $d \ge 1$ case
to the restrictions $\bff_1|_{\bB_a}$ and $\bff_2|_{\bB'_a}$.
\end{proof}

\subsection{Applications to Poincaré series}
\label{sec:poincare}

Finally, we have all the tools we need to show
how riso-stratifications provide information about Poincaré series. We start by defining precisely the two notions of local motivic Poincaré series to which our result applies. 
To reduce the amount of notation to fix, we write things down working over $\CC$, but in this entire subsection, one can consider $\CC$ as being an arbitrary algebraically closed field of characteristic $0$.

For the entire subsection, we fix a closed sub-scheme $\bX \subseteq \mathbb A^n_{\CC}$ for some $n \ge 1$.
Let us also fix polynomials $f_1, \dots, f_\ell \in \CC[x_1, \dots, x_n]$ generating the ideal defining $\bX$. To this $\bX$, we associate two different subsets of $(\CC[t]/t^r)^n$, for every $r \in \NN$:
\[
\tilde Y_r :=
\bX(\CC[t]/t^r) =
\{
a \bmod t^r \mid a \in \CC[[t]]^n, \forall i\colon f_i(a) \equiv 0 \mod t^r
\}
\]
and
\[
Y_r :=
\im(\bX(\CC[[t]]) \to \bX(\CC[t]/t^r)) =
\{
a \bmod t^r \mid a \in \CC[[t]]^n,  \forall i\colon f_i(a)= 0
\}.
\]

To define the Poincaré series, we need a notion of motivic measure. 
We apply Section~\ref{sec:mot} in the following setting:
\begin{conv}\label{conv:LL:int}
We work in the field $\CCt$, which we consider as an $\LL$-structure, where $\LL := \Lvf(\CC) \cup \{\ac\}$ consists of the valued field language $\Lvf$ (see Notation~\ref{notn:Lvf}) together with the standard angular component map
$\ac\colon \CCt \to \RF(\CCt) = \CC$ and a constant in the $\VF$-sort for each element of $\CC \subseteq \CCt$.\footnote{To formally make sense of $\ac$, one can either add $\RF$ as a sort to $\LL$ or make it $\Leq$-definable by putting a suitable predicate onto $\VF^2$.}
We set $\TT := \Th_{\LL}(\CCt)$.
\end{conv}
Since $\TT$ is complete, an $\LL$-definable set $\bZ \subseteq \CCt^n$ in this context is already determined by $Z = \bZ(\CCt)$. In particular, we can allow ourselves to write $\mu(Z)$ (instead of $\mu(\bZ)$) for the motivic measure of $\bZ$.

\begin{defn}\label{defn:poincare}
Given $\bX$, $\tilde Y_r$ and $Y_r$ as above and a $\CC$-valued point $a_0 \in \bX(\CC) \subseteq \CC^n$, we set
\[\tilde X_{a_0,r} :=
\{a \in \CC[[t]]^n \mid \res(a) = a_0 \wedge (a \bmod t^r) \in \tilde Y_r\}
\] and \[
X_{a_0,r} :=
\{a \in \CC[[t]]^n \mid \res(a) = a_0
\wedge (a \bmod t^r) \in Y_r\},
\]
and then define the following two \emph{local motivic Poincaré series} of $\bX$ at $a_0$, both of which are elements of $\cC(\pt)[[T]]$:
\[
\tilde P_{\bX,a_0}(T) := \sum_{r  \ge 1}
\bbL^{rn}\cdot
\mu(\tilde X_{a_0,r})\cdot T^r
\]
and
\[
P_{\bX,a_0}(T) := \sum_{r  \ge 1}
\bbL^{rn}\cdot
\mu(X_{a_0,r})\cdot T^r.
\]
\end{defn}

In that definition, for the motivic measures to make sense, we need the sets
$\tilde X_{a_0,r}$ and $X_{a_0,r}$ to be $\LL$-definable.
Indeed, they can be defined as follows (using that $|t|$ can be defined as being the biggest element less than $1$ in the value group):
\begin{equation}\label{eq:Xtil:def}
\tilde X_{a_0,r} = \{a \in \CC[[t]]^n \mid \res(a) = a_0 \wedge\max_i |f_i(a)| \le |t^r|\}
\end{equation}
and
\begin{equation}\label{eq:X:def}
X_{a_0,r} = \{a \in \CC[[t]]^n \mid \res(a) = a_0 \wedge
\exists b \in \bX(\CC[[t]])\colon |b-a| \le |t^r|\}
\end{equation}
(where we consider $\bX(\CC[[t]])$ as a subset of $\CC[[t]]^n$).

Before stating our main result, let us fix some (abuse of) notation:
\begin{notn}\label{notn:abuse}
\begin{enumerate}
\item We let $\Lnoval := \Lring(\CC)$ be the ring-language together with constants for $\CC$ and set $\Tnoval := \Th_{\Lnoval}(\CC)$. (Recall that we write $\bA$ for the only sort of $\Lnoval$.)
\item
Given an affine scheme $\bX \subseteq \mathbb A_{\CC}^n$, say, defined by $f = (f_1, \dots, f_\ell) \in (\CC[x_1, \dots, x_n])^\ell$, we also write $\bX$ for the $\Lnoval$-definable set
    $\{a \in \bA^n \mid f(a) = 0\}$ corresponding to $\bX$.
\item We identify the $\CC$-valued points $a \in \bX(\CC)$ with closed points of $\bX$.
\item Given such a point $a \in \bX(\CC)$,
we write $T_a\bX$ for the tangent space of $\bX$ at $a$ considered as a subscheme of $\mathbb A_{\CC}^n$ (or as an $\Lnoval$-definable subset of $\bA^n$), namely, defined by
$Df_a(x) = 0$, where $Df_a$ is the total derivative of $f$ at $a$. Thus, in our notation, the Zariski tangent space $(\mathfrak{m}_a/\mathfrak{m}_a^2)^*$ is denoted by $T_a\bX(\CC)$
(where $\mathfrak{m}_a$ is the maximal ideal of the local ring of $\bX$ at $a$).
\end{enumerate}
\end{notn}

Our main result states (in two variants) that if $(\bS_d)_d$ is the riso-stratification of $\bX$, then its local motivic Poincaré series at a point $a_0 \in \bS_d(\CC)$ can be computed from the local Poincaré of $\bX \cap \bfW$ at $a_0$ in a natural way, where $\bfW$ is any linear subspace through $a_0$ which is transversal to $\bS_d$.

\begin{thm}\label{thm:poincare}
Let $\bX \subseteq \mathbb A^n_{\CC}$ be a closed sub-scheme, and let $(\bS_i)_i$ be either
\begin{enumerate}
    \item its algebraic riso-stratification (Definition~\ref{defn:alg:riso}), or   \item the riso-stratification of $\bX$
    in the sense of Definition~\ref{defn:risoStrat}, where we consider $\bX$ as an $\Lnoval$-definable
    set as described in Notation~\ref{notn:abuse}.
\end{enumerate}
Fix some $a_0 \in  \bX(\CC) \cap \bS_d(\CC)$ for some $0 \le d \le n$,
and fix an $(n-d)$-dimensional linear subspace $\bfW \subseteq \mathbb A^{n}_\CC$ containing $a_0$
such that the tangent spaces
$T_{a_0}\bfW(\CC)$ ($= \bfW(\CC) - a_0$)
and
$T_{a_0}\bS_d(\CC)$
are complements of each other in $\CC^n$.

Then we have
\begin{equation}\label{eq:thm:P}
P_{\bX,a_0}(T) = \bbL^{-d}\cdot  P_{\bX\cap\bfW,a_0}(\bbL^{d}\cdot T).
\end{equation}
If $(\bS_i)_i$ is the algebraic riso-stratification of $\bX$ (Case~(1)), then we additionally have
\begin{equation}\label{eq:thm:Ptilde}
\tilde P_{\bX,a_0}(T) = \bbL^{-d}\cdot  \tilde P_{\bX\cap\bfW,a_0}(\bbL^{d}\cdot T).
\end{equation}
\end{thm}

\begin{rmk}
Note that the right hand sides of \eqref{eq:thm:P} and \eqref{eq:thm:Ptilde} are just the local Poincaré series of ``$\bX$ made translation invariant'', i.e., of $(\bX \cap \bfW) + T_{a_0}\bS_d \cong (\bX \cap \bfW) \times \mathbb A_{\CC}^d$. 
\end{rmk}

The proof of Theorem~\ref{thm:poincare} can easily be adapted to other kinds of Poincaré series, e.g., the multi-variable series from \cite[Theorem~14.4.1]{CL.mot}. We leave it to the reader to work out the details.

\begin{proof}[Proof of Theorem~\ref{thm:poincare}]
We start by working in the setting of 
Section~\ref{sec:can_whit} (see Convention~\ref{conv:sect}), taking
$\Lnoval = \Lring(\CC)$ and $\Tnoval = \Th_{\Lnoval}(\CC)$ as in Notation~\ref{notn:abuse}, and with $K \succneqq_{\Lnoval} K_1 = K_0 = \CC$. (In particular, we are currently \emph{not} following Convention~\ref{conv:LL:int}.) Note that these choices also fit to Section~\ref{sec:alg:riso}, where the algebraic riso-stratification is defined.

Let $\bX$, $(\bS_i)_i$, $a_0$, and $\bfW$ be given as in the theorem. More precisely, we first prove \eqref{eq:thm:Ptilde}, so let us assume that $(\bS_i)_i$ is the algebraic riso-stratification of $\bX$, i.e., the riso-stratification of the $\Leq$-definable map $\bchi\colon \VF^n \to \RV^{(\ell)}, a \mapsto \rv^{(\ell)}(f_1(a), \dots, f_\ell(a))$. (Recall that we fixed generators $f_i$ of the ideal defining $\bX$.)

Set $B := \{a \in \valring(K)^n \mid \res(a) = a_0\}$, let
$\bfU := T_{a_0}\bS_d \subseteq \mathbb A_\CC^n$ be the tangent space of $\bS_d$ at $a_0$ (in the sense of Notation~\ref{notn:abuse}) and set $\bU := \bfU(\CC)$.
By Proposition~\ref{prop:riso-mani}, we have $\bU = \rtsp_B((\bS_i(K))_i)$,
so by Lemma~\ref{lem:riso-stratificationrtsp}, $\bchi_K$ is $\bU$-riso-trivial on $B$.

Since $K$ is (spherically) complete, we can embed $\CCt$ into $K$ by sending $t$ to any element of valuation less than $1$; fix such an embedding and consider $\CCt$ as a subfield of $K$ in this way. Working in the $\LL$-structure $K$, we have $\dclVF(\CCt) = \CCt$ and $\dclRV(\CCt) = \rv(\CCt)$, so Proposition~\ref{prop:fields} applies, yielding that there exists a straightener $\phi\colon B \to B$ witnessing $\bar U$-riso-triviality of $\bchi_K$ which respects $\CCt$, i.e.,
$\phi$ restricts to a map $\phi_0$ from $B_0 := B \cap \CCt^n$ to itself.

From now on (and for the remainder of the proof), we work in the structure $\CCt$, and we let $\LL$ and $\TT= \Th_\LL(\CCt)$ be as in Convention~\ref{conv:LL:int}. We also from now on consider $\bchi$ as an $\Leq$-definable map for these $\LL$ and $\TT$.
With these conventions, the map $\phi_0$ witnesses that $\chi := \bchi_{\CCt}$ is $\bar U$-riso-trivial on $B_0$.

The set $\tilde X_{a_0,r}$ appearing in the definition of the Poincaré series $\tilde P_{\bX,a_0}(T)$ can be written as
\[
\tilde X_{a_0,r} = \{a \in B_0 \mid |\chi(a)| \le |t^r|\},
\]
where the natural map from $\RV^{(\ell)}$ to $\VG$ (sending $\rv^{(\ell)}(b)$ to $|b|$ for $b \in K^\ell$) is denoted by $|\cdot|$. Therefore, $\bU$-riso-triviality of $\chi$ on $B_0$ also implies $\bU$-riso-triviality of $\tilde X_{a_0,r}$ on $B_0$. 
Set $\bfV := T_{a_0}\bfW\subseteq \mathbb A_\CC^n$ (so that $\bfW(\CC) = \bfV(\CC) + a_0 \subseteq \CC^n$), $U := \bfU(\CCt)$, $V := \bfV(\CCt)$, denote by $\pi_U\colon \CCt^n \to U$ the projection whose kernel is $V$,
and let $\psi\colon B_0 \to B_0$ be a $U$-straightener of $\tilde X_{a_0,r}$, i.e., such that $X' := \psi^{-1}(\tilde X_{a_0,r})$ is $U$-translation invariant.
By the assumption that $\bfV(\CC)$ is a complement of $\bU = \bfU(\CC)$, we may additionally suppose that $\psi$ respects $\pi_U$-fibers
and that $\psi$ is the identity on the $\pi_U$-fiber $F := \bfW(\CCt) \cap B_0$ (see Remark~\ref{rmk:fix:fiber}).
This ensures that $X'$ is $\LL$-definable.
More precisely, if we define $g\colon B_0 \to F$ to send $a \in B_0$ to the unique $b \in F$ satisfying $b - a \in U$, then
we have $X' = g^{-1}(\tilde X_{a_0,r} \cap F)$.
This allows us to express the motivic measure of 
$\tilde X_{a_0,r}$ in terms of the motivic measure of $\tilde X_{a_0,r} \cap F$, as follows.
Since motivic measure is preserved by risometries (Proposition~\ref{prop:presMeas}), and since each fiber of $g$ is a $d$-dimensional maximal ideal sized ball (and hence has measure $\bbL^{-d}$), we obtain
\begin{equation}\label{eq:muX}
\mu(\tilde X_{a_0,r}) = \mu(X') =
\bbL^{-d}\cdot \mu_{n-d}(\tilde X_{a_0,r} \cap F),
\end{equation}
where by $\mu_{n-d}$, we mean the $(n-d)$-dimensional motivic measure on $F$ (obtained e.g.\ by identifying $F$ with $(t\CC[[t]])^{n-d}$).

Now we do a similar computation for the
Poincaré series of $\bX \cap \bfW$.
Denoting by $\tilde Z_{a_0,r}$ the sets whose measures we need to consider (so that
$\tilde P_{\bX\cap \bfW,a_0}(T)=\sum_{r \ge 1} \bbL^{rn}\cdot \mu(\tilde Z_{a_0,r})\cdot T^r$), we have
\[
\tilde Z_{a_0,r} = \{a \in B_0 \mid |\chi(a)| \le |t^r| \wedge |\bh_{\CCt}(a)| \le |t^r|\},
\]
where $\bh\colon \bA^n \to \bA^{d}$ is a tuple of linear polynomials over $\CC$ defining $\bfW$, considered as an $\Lnoval$-definable map.
Note that the fibers of $\bh_{\CCt}$ are exactly the fibers of $\pi_U$ and that we have
\[
\tilde Z_{a_0,r} = \tilde X_{a_0,r} \cap \pi_U^{-1}(B_U),
\]
for $B_U := \{u \in U \mid |u - \pi_U(a_0)| \le |t^r|\}$.
Since $\psi$ respects $\pi_U$-fibers, the set
$\psi^{-1}(\tilde Z_{a_0,r})$ is equal to
$X' \cap \pi_U^{-1}(B_U)$, and a similar computation as before (now using the restriction of $g$ to $B_0 \cap \pi_U^{-1}(B_U)$, whose fibers are $d$-dimensional balls of closed radius $|t^r|$)
yields 
\[
\mu(\tilde Z_{a_0,r}) =
\bbL^{-dr}\cdot \mu_{n-d}(\tilde X_{a_0,r} \cap F).
\]
Combining this with \eqref{eq:muX} yields
\[
\mu(\tilde X_{a_0,r}) = \mu(\tilde Z_{a_0,r})\cdot \bbL^{d(r-1)},
\]
and this implies \eqref{eq:thm:Ptilde}:
\begin{align*}
\tilde P_{\bX,a_0}(T) &= \sum_{r  \ge 1}
\bbL^{rn}\cdot
\mu(\tilde X_{a_0,r})\cdot T^r
= \sum_{r  \ge 1}
\bbL^{rn}\cdot
\mu(\tilde Z_{a_0,r})\cdot  \bbL^{d(r-1)}\cdot T^r\\
&=  \bbL^{-d}\cdot  \sum_{r  \ge 1}
\bbL^{rn}\cdot
\mu(\tilde Z_{a_0,r})\cdot (\bbL^d T)^r
= \bbL^{-d}\cdot P_{\bX\cap\bfW, a_0}(\bbL^d T).
\end{align*}

\medskip

To obtain a proof of \eqref{eq:thm:P}, now assuming that $(\bS_i)_i$ is either the riso-stratification or the algebraic riso-stratification of $\bX$, we adapt the above proof of \eqref{eq:thm:Ptilde} as follows:

If we work with the (non-algebraic) riso-stratification of $\bX$, we define $\bchi$ to be the indicator function of $\bX$. (Otherwise, we keep $\bchi$ as before.)

Instead of working with $\tilde X_{a_0,r}$, we need to use
$X_{a_0,r}$. While this set cannot be expressed directly in terms of $\chi = \bchi_{\CCt}$, we
can use Lemma~\ref{lem:triv_cl} to deduce that a straightener $\phi_0\colon B_0 \to B_0$ of $\chi$ also straightens the tubular neighbourhood $X_{a_0,r}$ of $\bX(\CCt) \cap B_0$.
We can therefore continue the proof as for $\tilde P_{\bX,a_0}(T)$ and obtain
\[
\mu(X_{a_0,r}) =
\bbL^{-d}\cdot \mu_{n-d}(X_{a_0,r} \cap F).
\]
As before, we then do a similar computation for
$P_{\bX\cap \bfW,a_0}(T)=\sum_{r \ge 1} \bbL^{r(n-d)}\cdot \mu(Z_{a_0,r})\cdot T^r$. 
Using that $Z_{a_0,r}$ is a tubular neighbourhood of $\bX(\CCt) \cap \bfW(\CCt) \cap B_0$, we again obtain
$Z_{a_0,r} = X_{a_0,r} \cap \pi_U^{-1}(B_U)$
(for the same $B_U \subseteq U$ as before), and then the rest of the proof is unchanged.
\end{proof}

We end this section with an example showing that the conclusion of Theorem~\ref{thm:poincare} would not hold if $(\bS_i)_i$ is assumed to be an arbitrary Whitney or Verdier stratification.
The set $\bX$ in that example is due do Mostowski \cite{Mos.biLip}.

\begin{exa}\label{exa:notVerdier}
Set $\bX(\CC) := \{(x,0,0,w)\mid x,w \in \CC\} \cup \{(x,x^3,xw,w)\mid x,w \in \CC\}$, $\bS_1 := \{(0,0,0)\} \times \CC$ and $\bS_2 := \bX \setminus \bS_1$. This is a Verdier stratification of $\bX$, but the local motivic Poincaré series $P_{\bX,(0,0,0,0)}(T)$ does not satisfy the formula \eqref{eq:thm:P} from Theorem~\ref{thm:poincare}.
\end{exa}

We leave it to the reader to verify that this $(\bS_r)_{r}$ is a Verdier stratification (this is straight forward), but we will explain below how to compute the local motivic Poincaré series without too much effort.

Note that the failure of formula \eqref{eq:thm:P} with respect to $(\bS_r)_{r}$ shows (via Theorem~\ref{thm:poincare}) that it is not a riso-stratification. One can also check directly that $\bX(K)$ is not $1$-riso-trivial on on the infinitesimal neighbourhood $\maxid(K)^4$ of the origin (e.g.\ for $K = \CCt[t^\QQ]$), and hence that 
the riso-stratification of $\bX$ puts the origin into the $0$-dimensional skeleton.

\begin{proof}[Proof of the second statement of Example~\ref{exa:notVerdier}]
We compute the local motivic Poincaré series of $\bX(\CC)$ around $0$.

First note that for any point $p = (x,y,z,w) \in (t\CC[[t]])^4$, the (valuative) distance from $p$ to $X := \bX(\CCt)$ is realized by a point in the same $x$-$w$-plane, so distances appearing in the definition of the Poincaré series (in \eqref{eq:X:def})
can be computed entirely within those planes.

The intersection of $X$ with such a plane is of the form $\{(0,0), (x^3, xw)\}$, i.e., it consists of two (possibly equal) points of valuative distance $|(x^3,xw)| =: |t^r|$ which both lie in $(t\CC[[t]])^2$.
The local motivic Poincaré series around $(0,0)$ of such a two-point-set is
\[
P_r(T) := \frac{T+T^{r+1}}{1-T}.
\]
(This is an easy computation involving some geometric series.)
If $|x| = |t^s|$ and $|w| = |t^{s'}|$, then
$|(x^3,xw)| = |t^{\min\{3s,s+s'\}}| =: |t^{r(s,s')}|$, so
one obtains the local motivic Poincaré series of $\bX$ by summing over $s$ and $s'$ as follows:
\[
P_{\bX, (0,0,0,0)}(T) = \sum_{s, s' \ge 1}(1-\bbL^{-1})^2\bbL^{-s-s'} P_{r(s,s')}(\bbL^{2} T).
\]
(One way to obtain this formula is to express the measures of the sets appearing in the Poincaré series of $\bX$ as integrals running over $x$ and $w$, of the corresponding measure within the $x$-$w$-plane.)
A tedious but straight forward computation\footnote{A possible strategy is to split the sum in the following way: $P_{\bX, (0,0,0,0)}(T) =
\sum_{s,s' \ge 1} (1-\bbL^{-1})^2\bbL^{-s-s'}T(1-T)^{-1} +
\sum_{s,s' \ge 1} (1-\bbL^{-1})^2\bbL^{-s-s'}T^{s + s'}(1-T)^{-1}+
\sum_{s \ge 1,s' > 2s} (1-\bbL^{-1})^2\bbL^{-s-s'}(T^{3s} - T^{s + s'})(1-T)^{-1}
$
} then yields
\[
P_{\bX, (0,0,0,0)}(T) = \frac{T- \bbL T^2 + (\bbL^4 - 2\bbL^3 + \bbL^2)T^3 + (\bbL^5 - \bbL^4 - \bbL^3)T^4 + (-\bbL^5 +2\bbL^4)T^5}{(1-\bbL^2T)(1-\bbL T)(1 - \bbL^3T^3)}.
\]

Using the same methods, one also obtains the local motivic Poincaré series of the intersection of $\bX$ with $\bW := \bA^3 \times \{0\}$, namely:
\[
P_{\bX\cap \bfW,(0,0,0,0)}(T) = 
\sum_{s \ge 1}(1-\bbL^{-1})\bbL^{-s}P_{3s}(\bbL T)
=\frac{T + (\bbL - 2)T^4}{(1-\bbL T)(1 - T^3)}.
\]
Clearly, $\bbL^{-1}P_{\bX\cap \bfW,(0,0,0,0)}(\bbL T) \ne P_{\bX,(0,0,0,0)}(T)$, showing that Theorem~\ref{thm:poincare} indeed would not hold for the stratification of $\bX$ given in the example.
\end{proof}

 \bibliographystyle{amsplain}
\bibliography{references}

\end{document}